\documentclass[11pt, reqno]{article}

\usepackage{fullpage}
\usepackage{enumerate}
\usepackage{setspace}

\usepackage{amsmath,amsfonts,amssymb,graphics,amsthm, comment}
\usepackage{hyperref}
\hypersetup{
    colorlinks=true,
    linkcolor=blue,
    citecolor=red,
    urlcolor=blue,
    pdfborder={0 0 0}
}

\usepackage{graphicx}
\usepackage[dvipsnames]{xcolor}
\usepackage{bbm}
\usepackage{wrapfig} 

\usepackage[font=sf, labelfont={sf,bf}, margin=1cm]{caption}

\usepackage{cleveref}
  \crefname{theorem}{Theorem}{Theorems}
  \crefname{thm}{Theorem}{Theorems}
  \crefname{lemma}{Lemma}{Lemmas}
  \crefname{lem}{Lemma}{Lemmas}
  \crefname{remark}{Remark}{Remarks}
  \crefname{prop}{Proposition}{Propositions}
\crefname{notation}{Notation}{Notations}
\crefname{claim}{Claim}{Claims}
  \crefname{defn}{Definition}{Definitions}
  \crefname{corollary}{Corollary}{Corollaries}
  \crefname{section}{Section}{Sections}
  \crefname{figure}{Figure}{Figures}
    \crefname{assumption}{Assumption}{Assumptions}

\newtheorem{thm}{Theorem}[section]

\newtheorem{conj}[thm]{Conjecture}

\newtheorem{lemma}[thm]{Lemma}
\newtheorem{corollary}[thm]{Corollary}
\newtheorem{prop}[thm]{Proposition}
\newtheorem{defn}[thm]{Definition}

\newtheorem{problem}[thm]{Problem}

\numberwithin{equation}{section}

\theoremstyle{definition}
\newtheorem{remark}[thm]{Remark}

\def \min {\text{min}}

\def\cZ{\mathcal{Z}}

\def\cT{\mathcal{T}}
\def\cS{\mathcal{S}}
\def\cR{\mathcal{R}}

\def\cP{\mathcal{P}}

\def\cN{\mathcal{N}}
\def\cM{\mathcal{M}}

\def\cG{\mathcal{G}}
\def\cF{\mathcal{F}}
\def\cE{\mathcal{E}}
\def\cD{\mathcal{D}}

\def\P{\mathbb{P}}
\def\Q{\mathbb{Q}}
\def\E{\mathbb{E}}
\def\C{\mathbb{C}}
\def\R{\mathbb{R}}
\def\Z{\mathbb{Z}}
\def\N{\mathbb{N}}

\def\H{\mathbb{H}}

\def  \p- {p\textunderscore}

\def\eps{\varepsilon}

\DeclareMathOperator{\Pf}{Pf}

\def\diff{\textnormal{d}}

\def \hg {\Phi}

\def\WE{\circ} 
\def\WO{\times}
\def\BO{\color{red}{\times}} 
\def\BE{\color{red}{\circ}} 

\DeclareMathOperator{\var}{Var}

\DeclareMathOperator{\dist}{dist}

\DeclareMathOperator{\cov}{Cov}

\DeclareMathOperator{\m}{\mathbf{m}}

\DeclareMathOperator{\sgn}{sgn}

\def\Ve{V_{\textnormal{even}}}
\def\Vo{V_{\textnormal{odd}}}

\def\Zev{Z_{\textnormal{even}}}
\def\Zodd{Z_{\textnormal{odd}}}

\def\vb{v_{\bullet}}
\def\ub{u_{\bullet}}

\def\Gev{G_{\textnormal{even}}}
\def\Godd{G_{\textnormal{odd}}}
\def\Aev{A_{\textnormal{even}}}
\def\Aodd{A_{\textnormal{odd}}}

\def\Reff{\cR_{\text{eff}}}

\def\br{ x \to \bar y; t}
\def\brp{ x' \to \bar y; t}

\begin{document}

\title{Free boundary dimers: \\
random walk representation and scaling limit}

\author{Nathana\"{e}l Berestycki\thanks{Universit\"{a}t Wien} \and Marcin Lis\footnotemark[1] \and Wei Qian\thanks{ICNRS and Laboratoire de Math\'{e}matiques d'Orsay, Universit\'{e} Paris-Saclay}}

\date{\today}

\maketitle
\begin{abstract}
We study the dimer model on subgraphs of the square lattice in which vertices on a prescribed part of the boundary (the free boundary) are possibly unmatched. Each such unmatched vertex is called a monomer and contributes a fixed multiplicative weight $z>0$ to the total weight of the configuration. A bijection described by Giuliani, Jauslin and Lieb \cite{GJL16} relates this model to a standard dimer model but on a non-bipartite graph. The Kasteleyn matrix of this dimer model describes a walk with transition weights that are negative along the free boundary. Yet under certain assumptions, which are in particular satisfied in the infinite volume limit in the upper half-plane, we prove an effective, true random walk representation for the inverse Kasteleyn matrix.
In this case we further show that, independently of the value of $z>0$, the scaling limit of the height function is the Gaussian free field with Neumann (or free) boundary conditions, thereby answering a question of Giuliani et al.

\end{abstract}

\tableofcontents

\section{Introduction}

\subsection{Free boundary dimers} \label{S:MD}
Let $\cG=(V,E)$ be a finite, connected, planar bipartite graph (in our analysis we will actually only consider subgraphs of the square lattice $\mathbb{Z}^2$).
Let $\partial \cG$ be the set of boundary vertices, i.e., vertices adjacent to the unique unbounded external face, and let $\partial_{\textnormal{free}} \cG  \subseteq \partial \cG$ be
a fixed set called the \textbf{free boundary}.
A \textbf{boundary monomer-dimer cover} of~$\cG$ is a set $M \subseteq E$ such that
\begin{itemize}
\item each vertex in $V \setminus \partial_{\textnormal{free}} \cG$ belongs to \emph{exactly} one edge in $M$,
\item each vertex in $\partial_{\textnormal{free}} \cG$ belongs to \emph{at most} one edge in $M$.
\end{itemize}
We write $\textnormal{mon}(M) \subseteq \partial_{\textnormal{free}} \cG$ for the set of vertices that do not belong to any edge in $M$, and call its elements \textbf{monomers}.
Let $\mathcal{MD}(\cG)$ be the set of all boundary monomer-dimer covers of $\cG$.
We will often call such configurations simply monomer-dimer covers, keeping in mind that monomers are only allowed on the free boundary.
Finally let $\mathcal{D}(\cG)$
be the set of all \textbf{dimer covers}, i.e.\ monomer-dimer covers $M$ such that $\textnormal{mon}(M)=\emptyset$.

We assign to each edge $e\in E$ a weight $w_e \geq 0$,
and to each vertex $v\in \partial_{\textnormal{free}}\cG$ a weight $z_v \geq0$.
The \textbf{dimer model with a free boundary} (or \textbf{free boundary dimer model}) is a random choice of a boundary monomer-dimer cover from $\mathcal{MD}(\cG)$ according to the following probability measure:
\[
\mathbf{P}(M) = \frac{1}{\mathcal Z} \prod_{e \in M } w_e \prod_{v\in\textnormal{mon}(M)} z_v,
\]
where $\mathcal Z$ is the normalizing constant called the \textbf{partition function}. For convenience we will always assume that the graph is {dimerable} meaning that $\cD(\cG) \neq \emptyset$. 
In this work we will only focus on the homogeneous case $w_e= 1$, for all $ e\in E$, and $ z_v=z>0$  for all $ v\in \partial_{\textnormal{free}} \cG $ (with the exception of the technical assumption on the weight of corner monomers described in the next section).
Then
\begin{equation}\label{lawMD}
\mathbf{P}(M) = \frac{1}{\mathcal Z}z^{|\textnormal{mon}(M)|} .
\end{equation}
The \textbf{dimer model} on $\cG$ can be now defined as the free boundary dimer model conditioned on $\mathcal{D}(\cG)$, i.e., the event that there are no monomers.

The main observable of interest for us will be the \textbf{height function} of a boundary monomer-dimer cover which is an integer-valued function defined (up to a constant) on the bounded faces
of $\cG$. Its definition is identical to the one in the dimer model (see \cite{Thurston}). We simply note that the presence of monomers on the boundary does not lead to any
topological complication (i.e., the height function is not multivalued): if $u$ and $u'$ are two faces of the graph, and $\gamma$ and $\gamma'$ are two distinct paths in the dual graph connecting
$u$ and $u'$, the loop formed by connecting $\gamma$ and~$\gamma'$ (in the reverse direction) does not enclose any monomer.
More precisely, we view a configuration $M \in \cM \cD (\cG)$ as an antisymmetric flow (in other words a 1-form) $\omega_M$ on the directed edges of $\cG$ in the following manner: if $e=\{ w,b\}\in M$,
then $\omega_M(w,b) =1$ and $\omega_M (b,w) = -1$ where $b$ is the black vertex of $e$ and $w$ its white vertex (since $\cG$ is bipartite, a choice of black and white vertices can be made in advance).
Otherwise, we set $\omega_M(e) = 0$. Equivalently, we may view $\omega_M$ as an antisymmetric flow on the directed dual edges, where if $e^\dagger$ is the dual edge of $e$
(obtained by a counterclockwise $\pi/2$ rotation of $e$), then $\omega_M(e^\dagger) = \omega_M(e)$. To define the height function we still need to fix a \textbf{reference flow} $\omega_0$
which we define to be $\omega_0= \mathbf E[ \omega_M]$, i.e.,
the expected flow of $M$ under the free boundary dimer measure.
Now, if $u$ and $u'$ are two distinct (bounded) faces of $\cG$, we simply define
$$
h(u) - h(u') = \sum_{e^\dagger \in \gamma} (\omega_M(e^\dagger)-\omega_0(e^\dagger))
$$
where $\gamma$ is any path (of dual edges) connecting $u$ to $u'$. This definition does not depend on the choice of the path since the flow $\omega_M(e^\dagger)-\omega_0(e^\dagger)$ is closed
(sums over closed dual paths vanish), and hence yields a function $h$ up to an additive constant, as desired. Note that our choice of the reference flow automatically guarantees that the height function
is centered, i.e., $\mathbf E(h(u) - h(u'))=0$ for all faces $u$ and $u'$.

We finish this short introduction to the free boundary dimer model with a few words on its history and the nomenclature. In the original model studied in~\cite{HeiLie70,HeiLie72} monomers could occupy any vertex of the graph, and hence the name \textbf{monomer-dimer model}. This generalization poses two major complications from our point of view. Firs of all, the height function is not well defined, and secondly the model does not admit a Kasteleyn solution as was shown in~\cite{Jer}.
From this point of view, it would therefore be natural if the version of the model studied here was called the boundary-monomer-dimer model. However we choose to use the less cumbersome name of free boundary dimers.

\subsection{Boundary conditions}
\label{S:boundary}
We now state conditions on the graph $\cG=(V,E)$ which will be enforced throughout this paper. First, we assume that $\cG$ is a subgraph of the square lattice $\Z^2$, and without loss of generality that $0 \in V$ and is a black vertex. This fixes a unique black/white bipartite partition of $V$.
We also assume that
\begin{itemize}
\item $V$ is contained {in} the upper half plane $\H = \{ z \in \C: \Im (z) \ge 0\}$.

\item $\partial_{\textnormal{free}} \cG= V \cap \R$, so the monomers are allowed only on the real line. Furthermore, we assume $\partial_{\textnormal{free}} \cG$ is a connected set of vertices. The leftmost and rightmost vertices of $V \cap \R = \partial_{\textnormal{free}} \cG$ will be referred to as the \textbf{monomer-corners} of $\cG$.

\item $\cG$ has at least one black {dimer-corner} and one white {dimer-corner} (where a \textbf{dimer-corner} is a vertex $v\in V$ that is not a monomer-corner, and is adjacent to the outer face of $\cG$, and has degree either 2 or 4 in $\cG$).

\end{itemize}

See Figure \ref{fig:Gtriangle} for an example of a domain satisfying these assumptions (ignore the bottom row of triangles for now, which will be described later). We make a few comments on the role of the last assumption that there are corners of both colours. For this it is useful to make a parallel with Kenyon's definition of \textbf{Temperleyan domain} \cite{Kenyon_ci,KenyonGFF}. In that case, this condition ensured that the associated random walk on one of the four possible sublattices of $\Z^2$ (the two types of black and the two types of white vertices) was killed somewhere on the boundary.
As we will see, in our case the random walk may change the lattice from black to white when it is near the real line, resulting in only two different types of walks.
Then the role of the third assumption (at least one dimer-corner of each type) is to ensure that each of the two walks is killed on at least some portion of the boundary (possibly a single vertex). This follows from an observation that the boundary condition of a walk on a black (resp.\ white) sublattice changes from Neumann to Dirichlet (and vice-versa) at a white (resp. black) corner.
See Figure~\ref{fig:GN} for an example of a vertex with Neumann and Dirichlet bondary conditions.

\subsection{Statement of main results}
\label{sec:statement}

The free boundary dimer model as defined above was discussed (with minor modifications) in a paper of Giuliani, Jauslin and Lieb \cite{GJL16}. It was shown there that the partition function $\mathcal Z$ can be computed as a Pfaffian of a certain matrix. Furthermore, a bijection was provided to a non-bipartite dimer model (the authors indicate that this bijection was suggested by an anonymous referee). Hence using Kasteleyn theory the correlation functions can be expressed as Pfaffians of the inverse Kasteleyn matrix $K^{-1}$. The bijection, which is a central tool of our analysis, will be defined in Section~\ref{S:IKM} where we will also recall the precise definition of the Kasteleyn matrix~$K$.

We will now state our first main result which gives a full random walk representation for $K^{-1}$. Suppose that $\cG$ is a graph satisfying the assumptions from the previous section. Fix $z>0$ and assign weight $z$ to every monomer on $\partial_{\textnormal{free}} \cG$ except at either monomer-corner, where (for technical reasons which will become clear in the proof) we choose the weight to be
\begin{equation}\label{zprime}
z' = \frac{z}2 + \sqrt{1 + \frac{z^2}{4}}.
\end{equation}
For $k \in \N = \{0, 1, \ldots \}$, let us call $V_k =V_k(\cG) = \{v \in V: \Im (v) = k \}$, so $\partial_{\textnormal{free}} \cG=V_0$, where $\Im(v)$ denotes the imaginary part of the vertex $v$ seen as a complex number given by the embedding of the graph. Let us call $\Ve =\Ve(\cG) = V_0 \cup V_2 \cup \ldots$ and $\Vo =\Vo(\cG)= V_1 \cup V_3 \cup \ldots$.

\begin{thm}[Random walk representation of the inverse Kasteleyn matrix]\label{T:finitevol_intro}
  There exist two random walks $Z_\textnormal{even}$ and $Z_{\textnormal{odd}}$ on the state spaces $\Ve(\cG)$ and $\Vo(\cG)$ respectively, whose transition probabilities will be described in Section \ref{S:inverseKfinite} (see \eqref{BulkRW} and  \eqref{Rev}), such that the following holds.
Consider the monomer-dimer model on $\cG$ where the monomer weight is $z>0$ on $V_0(\cG)$ except at its monomer-corners where the monomer weight is $z'$, as defined in \eqref{zprime}. Let $K$ be the associated Kasteleyn matrix, and $D = K^* K$, so that $K^{-1} = D^{-1} K^*$. Then for all $u,v \in V$, we have
  \begin{align} \label{RWrep}
  D^{-1}(u,v) =
  \begin{cases}
      G_{\textnormal{odd}} (u,v) & \text{if } u,v\in \Vo, \\
(-1)^{\Re (u-v)}G_{\textnormal{even}} (u,v)     & \text{if } u,v \in \Ve, \\
    0 & \text{otherwise.}
  \end{cases}
  \end{align}
  where $G_{\textnormal{even}}, G_{\textnormal{odd}} $ are the Green's functions of $Z_{\textnormal{even}}$ and $Z_{\textnormal{odd}}$ respectively, normalised by $D(v,v)$.
\end{thm}

 Here, by normalised Green's function of a random walk (with at least one absorbing state), we mean
 \[
 G(u,v) = \frac 1{ D(v,v)}\E_u\Big( \sum_{k=0}^\infty \mathbf 1_{\{ Z_k =  v \}}\Big),
 \] where $Z$ is the corresponding random walk. We now specify a few properties of the random walks ${\Zev}$ and $Z_{\text{odd}}$ which may be interesting to the reader already, even though the exact definition is postponed until Section \ref{S:inverseKfinite}. Both ${\Zev}$ and $Z_{\text{odd}}$ behave like simple random walk away (at distance more than $2$) from the boundary vertices, but with jumps of size $\pm 2$, so the parity of the walk does not change. Both have nontrivial boundary conditions, including some reflecting and absorbing boundary arcs along the non-monomer part of the boundary $\partial \cG\setminus \partial_{\textnormal{free}} \cG$. Furthermore, both walks are allowed to make additional jumps along their bottommost rows of vertices ($V_0$ for $Z_{\text{ev}}$ and $V_1$ for $Z_{\text{odd}}$). These jumps are symmetric, bounded in the even case but not in the odd case (although they do have exponentially decaying tail). Hence in the scaling limit, these walks would converge to Brownian motion in the upper half plane $\H$ with reflection on the real axis and with whatever boundary conditions are inherited from the Neumann/Dirichlet parts of the other boundary arcs.

 An important property of these random walks that highlights the difference with the setup of~\cite{Kenyon_ci}, is that they can change {colour} of the vertex (in a bipartite coloring of $\H \cap \Z^2$).
However, this can happen only when the walker visits the real line. This in turn means that the entries of the inverse Kasteleyn matrix indexed by two vertices of the same {colour} (which automatically vanish in Kenyon's work) have a natural interpretation in terms of walks that go through the real line (the free boundary). This is a clear analogy with the construction of reflected random walks via the reflection principle for a walk in a reflected domain.
Remarkably, this exact correspondence with reflected random walks is present already at the discrete level of the dimer model with free boundary conditions, and is the reason why the reflected Brownian motion appears in the correlation kernel of the scaling limit of the height function.

\medskip To illustrate this we explain here briefly a simple computation using Kasteleyn theory (for more details see Section~\ref{S:KT}) where this phenomenon is apparent. Let $ e= \{w,b\}$ and $e' = \{w', b'\}$ be two edges of $\mathbb Z^2\cap \mathbb H$ with $w,w'$ white and $b,b'$ black vertices in a fixed chessboard coloring of the lattice.
  Then, writing $\mathcal {M}$ for a random boundary monomer-dimer cover, and using Kasteleyn theory for the dimer representation described in Section~\ref{S:bij}, we have
  \begin{align*}
  \mathbf P( e , e' \in \mathcal{M}) &
  =a \text{Pf}
  \left(
  \begin{array}{cccc}
  0 & K^{-1} (w,b) & K^{-1} (w, w') & K^{-1} (w, b') \\
  &    0         & K^{-1} ( b, w') & K^{-1} ( b, b') \\
  &           & 0              &  K^{-1} (w', b') \\
 &             &               &   0
  \end{array}
  \right)
  \\
  & =a(K^{-1} (w, b)K^{-1} (w', b')  + K^{-1} (b, w') K^{-1} (w, b') - K^{-1} ( w, w') K^{-1} (b, b')),
  \end{align*}
  where the matrix is antisymmetric and $a=K(w,b)K(w',b')$. We also have $\mathbf P( e \in \mathcal{M})=K(w,b)K^{-1} (w, b)$ and $\mathbf P( e' \in \mathcal{M})=K(w',b')K^{-1} (w', b')$,
  which leads to
 \[
  \cov (\mathbf 1_{e \in \mathcal M},\mathbf 1_{e' \in \mathcal{M}}) =a( K^{-1} (b, w') K^{-1} (w, b') - K^{-1} ( w, w') K^{-1} (b, b')).
  \]
  Here, the second term is new compared to Kenyon's computation in \cite{Kenyon_ci}. Furthermore, using our random walk representation, $K^{-1} ( w, w')$ and $K^{-1} (b, b')$ can be interpreted as a derivative of the Green's function of the appropriate walks $\Zev$ and $\Zodd$ evaluated at pairs of vertices of different colors. Then by construction the walks which contribute to these Green's functions must visit the boundary.

 This intuition is what guides us to the next result, which however requires us to first take an infinite volume (thermodynamic) limit where an increasing sequence of graphs eventually covers $\H \cap \Z^2$. We first show that the monomer-dimer model converges in such a limit. For this we need to specify a topology: we view a monomer-dimer configuration on $\H \cap \Z^2$ as an element of $\{0,1\}^{E(\H)}$ where $E(\H)$ is the edge set of $\Z^2 \cap \H$, and equip this space with the product topology (so convergence in this space corresponds to convergence of local observables).

To state the result we will fix a sequence $\cG_n$ of graphs such that $\cG_n$ satisfies the assumptions of Section \ref{S:boundary}, and moreover $\cG_n \uparrow \Z^2 \cap \H$.
For simplicity of the arguments and ease of presentation, we have chosen $\cG_n$ to be a concrete approximation of rectangles, although the result is in fact true much more generally; we have not tried to find the most general setting in which this applies.

 \begin{thm}[Infinite volume limit]\label{P:infinite_vol_intro}
 Let $\cG_n$ be rectangles of diverging \emph{odd} sidelengths (number of vertices on a side) whose ratio remains bounded away from zero and infinity as $n\to \infty$, and such that in the top row the right-hand side half of the vertices is removed.
Let $\mu_n$ denote the law of the free boundary dimer model on $\cG_n$ with monomer weight $z>0$ except at the monomer-corners where the weight is $z'$, as in \eqref{zprime}. Then $\mu_n$ converges weakly as $n \to \infty$ to a law $\mu$ which describes a.s.\ a random {boundary} monomer-dimer configuration on $\Z^2 \cap \H$.
 \end{thm}

We note that the particular type of domains chosen in this statement guarantees that both the odd and even walks mentioned above are killed on a macroscopic part of the upper rows of $\cG_n$ (the odd walk is killed on the left-hand side half and the even walk on the right-hand side half of its uppermost row).
We stress the fact that the limiting law $\mu$ depends on the monomer weight $z>0$.
As mentioned before, we can associate to the monomer-dimer configuration in the infinite half-plane a {height function} which is defined on the faces of $\H \cap \Z^2$, up to a global additive constant. The last main result of this paper shows that in the scaling limit, this height function converges to a \textbf{Gaussian free field} with \textbf{Neumann} (or free) \textbf{boundary conditions}, denoted by $\hg^{\text{Neu}}$. We will not define this in complete generality here (see \cite{BPnotes} for a comprehensive treatment). We will simply point out what is concretely relevant for the theorem below to make sense. Given a simply connected domain $\Omega$ with a smooth boundary, $\hg^{\textnormal{Neu}}_{\Omega}$ may be viewed as a stochastic process indexed by the space $\cD_0(\Omega)$ of smooth test functions $f: \Omega \to \R$ with compact support and with zero average (meaning $\int_\H f(z) dz = 0$). The latter requirement corresponds to the fact that $\hg$ is only defined modulo a global additive constant. The law of this stochastic process is characterised by a requirement of linearity (i.e. $(\hg^{\textnormal{Neu}}_{\Omega}, a f + b g) = a (\hg^{\textnormal{Neu}}_{\Omega}, f) + b (\hg^{\textnormal{Neu}}_{\Omega}, g)$ a.s. for any $f, g \in \cD_0(\Omega)$ and $a, b \in \R$), and moreover $(\hg^{\textnormal{Neu}}_{\Omega}, f)$,  $(\hg^{\textnormal{Neu}}_{\Omega}, g)$ follow centered Gaussian distributions with covariance
$$
\cov((\hg^{\textnormal{Neu}}_{\Omega}, f),(\hg^{\textnormal{Neu}}_{\Omega}, g)) = \iint_{\Omega^2} f(x) g(y) G^{\textnormal{Neu}}_\Omega(x,y) dx dy,
$$
where $G^{\textnormal{Neu}}_\Omega(x,y)$ is a Green's function in $\Omega$ with Neumann boundary conditions. (Note that by contrast to the Dirichlet case, such Green's functions are not unique and are defined only up to a constant.) In the case of the upper-half plane $\Omega= \H$, the Green's function is given explicitly by
$$
G^{\textnormal{Neu}}_\H(x,y) = - \log |x-y| - \log | x - \bar y|.
$$

Informally, pointwise differences $\hg^{\textnormal{Neu}}_{\H}(a) - \hg^{\textnormal{Neu}}_{\H} (b)$ for $a,b\in \H$ (which do not depend on the choice of the global additive constant) are centered Gaussian random variables with covariances
\begin{equation}\label{GFF_differences_Green}
\E[ (\hg^{\textnormal{Neu}}_{\H}(a_i) - \hg^{\textnormal{Neu}}_{\H}(b_i))(\hg^{\textnormal{Neu}}_{\H}(a_j) - \hg^{\textnormal{Neu}}_{\H}(b_j)) ] \!=\! - \log \left|\frac{(a_i - a_j)(b_i - b_j)(\bar a_i - a_j) (\bar b_i - b_j)}{ (a_i - b_j)(b_i - a_j)(\bar a_i - b_j)(\bar b_i - a_j)}\right|.
\end{equation}
Note that our Green's function is normalised so that it behaves like $1 \times \log (1/ | x-y| )$ as $y -x \to 0$. Naturally, \eqref{GFF_differences_Green} must be understood in an integrated way since pointwise differences are not actually defined.

We may now state the announced result. For $\delta>0$ (the mesh size), let $h^\delta$ denote the height function (defined up to a constant) of the free boundary dimer model $\mu$ with weight $z$ in the infinite
half-plane $\H \cap \delta\Z^2$ (rescaled by $\delta$). We identify $h^\delta$ with a function defined almost everywhere on $\H$ by taking the value of $h^\delta$ to be constant on each face, and view $h^\delta$ as a random distribution (also called a random generalized function) acting on smooth compactly supported functions $f$ on $\H$ with zero average, i.e., satisfying $\int_\H f(a)da =0$ (see Section~\ref{S:final} for details).

\begin{thm}[Scaling limit]
  \label{T:NGFF_intro}
  Let $f_{1}, \ldots, f_k \in \cD_0(\H)$ be arbitrary test functions. Then for all $z>0$, as $\delta \to 0$,
  $$
  (h^\delta, f_i)_{i=1}^k \to\Big (\frac1{\sqrt{2}\pi}\hg^{\textnormal{Neu}}_{\H}, f_i\Big)_{i=1}^k
  $$
  in distribution.
\end{thm}
Note that, maybe surprisingly, the scaling limit does not depend on the value of $z>0$ (we discuss this in more detail in Section~\ref{S:heur}).
We also wish to call the attention of the reader to the normalising factor $1/ (\sqrt 2\pi)$ in front of $\hg$ on the right-hand side of Theorem \ref{T:NGFF_intro}. It is equal to the one appearing in the usual dimer model in which the centered height function has zero (Dirichlet) boundary conditions.
We note that comparisons with other works such as \cite{KenyonGFF,BLR16} should be made carefully, since the normalisation of the Green's function and of the height function may not be the same: for instance, Kenyon takes the Green's function to be normalised so that $G(x,y) \sim 1/(2\pi) \log 1/|x-y|$ as $y \to x$, so his GFF is $1/\sqrt{2\pi}$ ours (ignoring different boundary conditions). Also, in Kenyon's work \cite{Kenyon_ci}, the height function is such that the total flow out of a vertex is 4 instead of 1 here (so his height function is 4 times ours), while it is $2\pi$ in \cite{BLR16} (so their height function is $2\pi$ times ours).
Adjusting for these differences, there is no discrepancy between the constant $1/(\sqrt{2}\pi)$ on the right-hand side of Theorem \ref{T:NGFF_intro} and the one in \cite{Kenyon_ci} and \cite{KenyonGFF}.

\subsection{Heuristics: reflection and even/odd decomposition}
\label{S:heur}

As noted before, Theorem \ref{T:NGFF_intro} may be surprising at first sight, when we consider the behaviour of the model in the two extreme cases $z=0$ and $z = \infty$. Indeed, when $z=0$, the
free boundary dimer model obviously reduces to the dimer model on $\H$, in which case the limit is a Dirichlet GFF. When $z= \infty$, all vertices of $V_0$ are monomers, so the model reduces to a dimer model on $(V_1 \cup V_2 \cup \ldots )\simeq \H \cap \Z^2$. Hence, the limit is also a Dirichlet GFF in this case. However, the result above says that for any $z$ strictly in between these two extremes, the limit is a Neumann GFF.

The result (and the reason for this {arguably} surprising behaviour) may be heuristically understood through the following \textbf{reflection argument}. Let $\cG$ be a large finite graph approximating $\H$ and satisfying the assumptions of Section \ref{S:boundary}. Let $\tilde \cG$ be a copy of $\cG$ shifted by $i/2$, so with a small abuse of notation, $\tilde \cG = \cG + i /2$, and let $\bar \cG$ be the same graph to which we add its conjugate (reflection through the real axis). We also add vertical edges crossing the real axis of the form $(k-i/2, k+i/2)$ for each $k \in V_0$. Given a monomer-dimer configuration on $\cG$, we can readily associate a monomer-dimer configuration on $\bar\cG$ by reflecting it in the same manner. In this way, a monomer in $k+ i/2$ necessarily sits across another monomer in $k-i/2$ for any $k \in V_0$. Such a pair of monomers can be interpreted as a dimer on the edge $(k-i/2, k+i/2)$ and once we have phrased it this way the resulting configuration is just an \emph{ordinary dimer configuration} on $\bar \cG$ (which however has the property that it is reflection symmetric). It follows that its height function (defined on the faces of $\bar \cG$) is even, i.e., $h(f) = h (\bar f)$ for every face $f$ (where $\bar f$ is the symmetric image of $f$ about the real axis). Moreover, a moment of thought shows that monomer-dimer configurations on $\cG$ are in bijection in this manner with the set $\cE \cD(\bar \cG)$ of even (symmetric) dimer configurations on $\bar \cG$, and that under this bijection the image of the law \eqref{lawMD} is given by
\begin{equation}\label{dim_ref}
\P(M) =\frac1{\bar{\mathcal Z}} z^{|\text{mon}(M)|}
\end{equation}
(where for a dimer configuration $M \in \cD(\bar \cG)$, $\text{mon} (M) $ is the set of vertical edges of $M$ crossing the real axis), conditioned on the event $\cE\cD(\bar \cG)$ of being even,
where $\bar{\mathcal Z}$ is the partition function of the dimer model on $\bar \cG$.

Now, suppose e.g.\ that $\cG$ is such that $\bar \cG$ is piecewise Temperleyan \cite{Russkikh} (meaning that $\bar \cG$ has two more white convex corners than white concave corners). This happens for instance if $\cG$ is a large rectangle with appropriate dimensions.
By a result of Russkikh \cite{Russkikh}, in this case and if $z = 1$, the (centered) height function associated with the dimer model \eqref{dim_ref} converges to a Gaussian free field with Dirichlet boundary condition in the scaling limit.

It is reasonable to believe that this convergence holds true even when $z \neq 1$. For instance, when the monomer weights alternate between $z$ and 1 every second vertex, then whatever the value of $z$, the dimer model has a Temperleyan representation (see \cite{KPWtemperley}, \cite{BLR_torus}).
Then by considerations related to the imaginary geometry approach (see \cite{BLR16}), this convergence to the Dirichlet GFF is universal provided that the underlying random walk converges to Brownian motion (this will be rigourously proved in the forthcoming work \cite{BerestyckiLaslierRusskikh}). In particular, given these results, we should get convergence to the Dirichlet GFF for the height function even when $z \neq 1$: indeed, when we modify the weight of all the edges crossing the real line, random walk will still converge to Brownian motion. So far, this discussion concerned the (unconditioned) dimer model on $\bar \cG$ defined in \eqref{dim_ref}. Once we start conditioning on $\cE \cD(\bar \cG)$ it might be natural to expect that the scaling limit should be a ``Dirichlet GFF conditioned to be even", though this is a highly degenerate conditioning.
Nevertheless, this conditioning makes sense in the continuum, and in fact its restriction to the upper half plane gives the Neumann GFF, as we are about to argue. Indeed, for a full plane GFF $\hg_{\C}$ restricted to $\mathbb H$, it is easy to check that one has the decomposition
\begin{align}\label{eq:even_odd}
\hg_{\C} = \tfrac1{\sqrt{2}} (\hg_\H^{\text{Neu}} + \hg_\H^{\text{Dir}})
\end{align}
where $\hg^{\text{Neu}}_\H , \hg^{\text{Dir}}_\H$ are independent fields on $\H$ with Neumann and Dirichlet boundary conditions on $\R$ respectively. This follows immediately from the fact that any test function can be written as the sum of an even and odd functions, and this decomposition is orthogonal for the Dirichlet inner product $(\cdot, \cdot)_{\nabla}$ on $\cD_0(\C)$. Therefore, conditioning $\hg_\C$ to be even amounts to conditioning on $\hg_\H^{\text{Dir}}$ to vanish everywhere, meaning that $\hg_\C$ (restricted to the upper half plane) is exactly equal to $\hg^{\text{Neu}}_\H/\sqrt{2}$. (See Exercise 1 of Chapter 5 in \cite{BPnotes} for details.)

\medskip \indent We note that while this argument correctly predicts the Neumann GFF as a scaling limit of the height function, it is however also somewhat misleading as it suggests that the limit of $h^\delta$ is not $(1/\sqrt{2}\pi) \hg^{\text{Neu}}_{\H}$ as in Theorem \ref{T:NGFF_intro}, but is smaller by a factor $1/\sqrt{2}$, i.e., $1/(2\pi) \hg^{\text{Neu}}_{\H}$.

\medskip \indent To understand this discrepancy, we now explain why the additional factor turns out to be an artifact of a Gaussian computation and does not arise in the discrete setup. A convincing one-dimensional parallel can be that of Gaussian and simple random walk bridges. Indeed, consider bridges of $2n$ steps starting and ending at $0$, with symmetric Bernoulli and Gaussian jump distributions with variance one.
Now condition the walks to be symmetric around time $n$, i.e. $X(n\pm k) = X(n\mp k)$. Again, the Gaussian conditioning is singular but can be easily made sense of using Gaussian integrals.
Restricted to the time interval $[0,n]$, the conditioned simple random walk bridge is just a simple random walk with the same step distribution as the original bridge. However, the conditioned Gaussian walk has step distribution
with variance $1/2$ as a result of the conditioning. In particular, in the diffusive scaling limit, the former walk converges to standard Brownian motion whereas the latter to $1/\sqrt{2}$ times the standard Brownian motion. The framework of the current paper is more similar to the simple random walk case as discrete height functions are its ``two-dimensional-time" analogs. This concludes the discussion giving the heuristics for Theorem \ref{T:NGFF_intro}.

\subsection{A conjecture on the boundary-touching level lines}\label{subsec:conjecture}

In the study of the dimer model, a well known conjecture of Kenyon concerns the superposition of two independent dimer configurations. It is easy to check that such a superposition results in a collection of loops (including double edges) covering every vertex. This observation is attributed to Percus \cite{Percus}. These loops are the level lines of the difference of the two corresponding dimer height functions. Kenyon's conjecture (stated somewhat informally in \cite{KenyonWilson} for instance) is that the loops converge in the scaling limit to CLE$_4$, the conformal loop ensemble with parameter $\kappa = 4$ (defined in \cite{Sheffield_exploration}, see also \cite{SheffieldWerner}). This is strongly supported by the fact that in the continuum, CLE$_4$ can be viewed as the level lines of a (Dirichlet) GFF with a specified variance (a consequence of a well known coupling between the GFF and CLE$_4$ of Miller and Sheffield, see \cite{BTLS} for a complete statement and proof). Major progress has been made recently on this conjecture through the work of Kenyon \cite{Kenyon_CI_DD}, Dub\'edat \cite{DubedatCLE} and Basok and Chelkak \cite{BasokChelkak}, and the only remaining ingredient of the full proof is to show precompactness
of the family of loops in a suitable metric space.

\begin{figure}
\begin{center}
\includegraphics[height=5.cm]{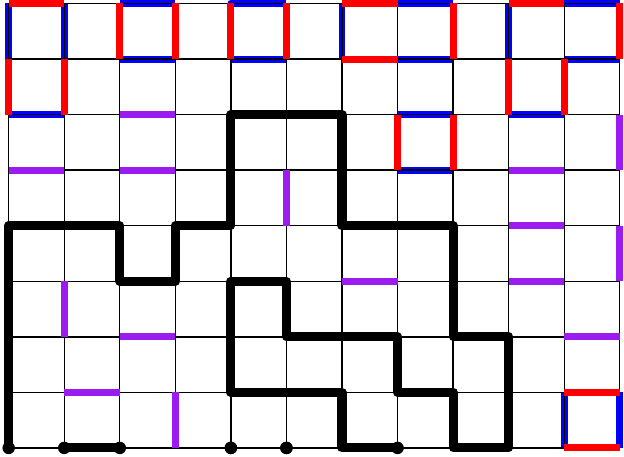}
\includegraphics[height=5.cm]{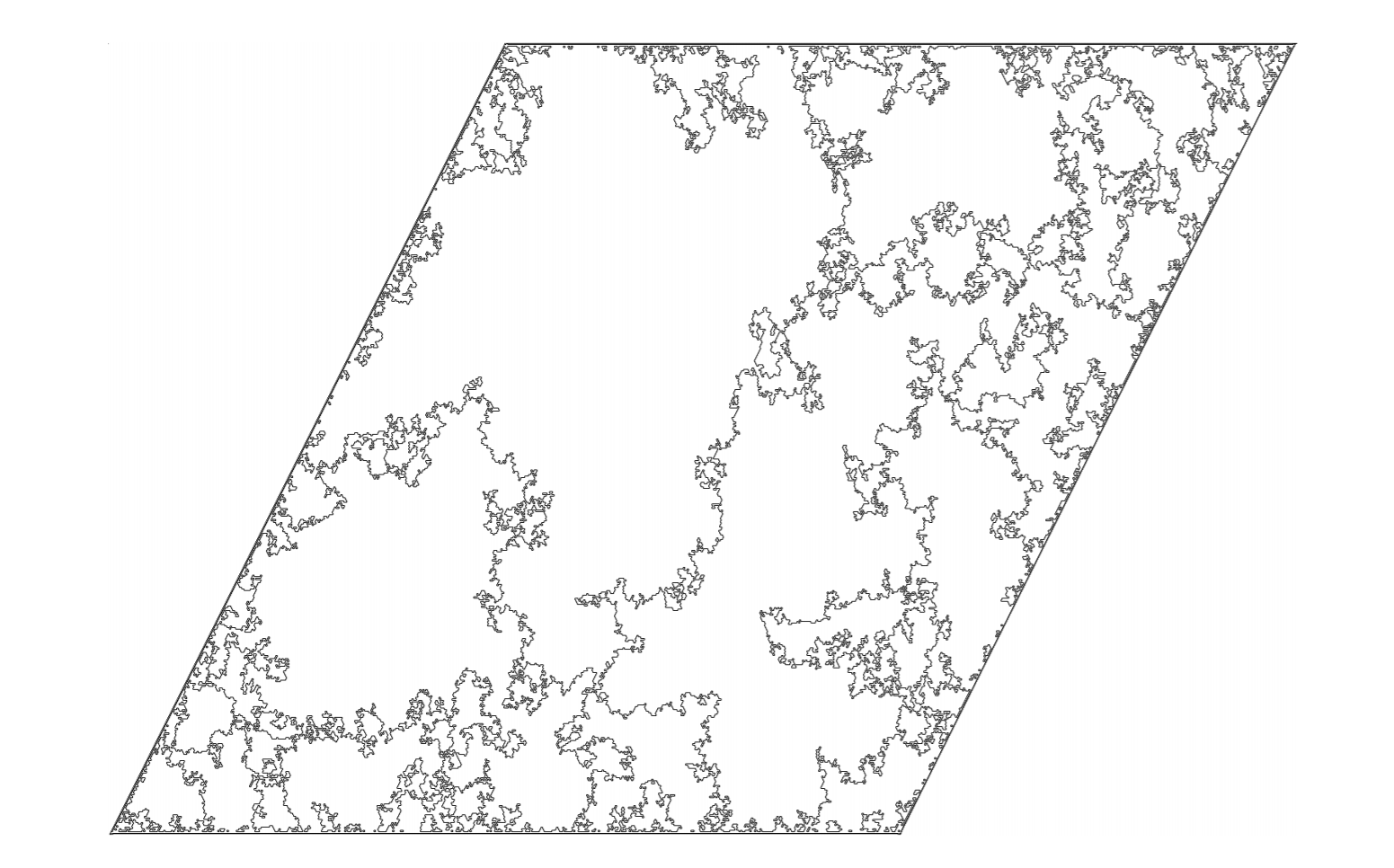}
\caption{Left: A superposition of two monomer-dimer configurations, respectively blue and red. Double edges are in purple. The boundary-touching level lines of the height-function is the collection of arcs joining monomers to monomers marked in bald black. Right: A simulation of ALE by B. Werness.}
\label{F:double}
\end{center}
\end{figure}

It is natural to ask if any similar phenomenon occurs when we superpose two independent monomer-dimer configurations sampled according to the free boundary dimer model, say in the upper half-plane. For topological reasons, this gives rise to a gas of loops as above but also a collection of curves connecting monomers to monomers (and hence the real line to the real line). See Figure \ref{F:double} for an example. An obvious question is to describe the law of this collection of curves in the scaling limit. By analogy with the above, and in view of our result (Theorem \ref{T:NGFF_intro}), it is natural to expect that these curves converge in the scaling limit to the level lines of a GFF with Neumann boundary conditions on the upper-half plane. The law of these curves was determined by Qian and Werner \cite{QianWerner_Neumann} to be the {ALE} process (ALE stands for Arc Loop Ensemble. It is a collection of arcs that can be connected into loops, but here we will not be interested in this aspect and will only see them as arcs.). ALE is one possible name for this set, but more precisely it is equal to the branching SLE$_4(-1,-1)$ exploration tree targeting all boundary points, and is also equal to the (gasket of) BCLE$_4(-1)$ in \cite{MR3708206} and $\mathbb{A}_{-\lambda, \lambda}$ in  \cite{BTLS}.

This leads us to the following conjecture:

\begin{conj} For any $z>0$, in the scaling limit, the collection of boundary-touching curves resulting from superimposing two independent free boundary dimer models converges to the \textbf{Arc Loop Ensemble} ALE in the upper half-plane.
\end{conj}

\subsection{Folding the dimer model onto itself}

The discussion in Sections \ref{S:heur} and \ref{subsec:conjecture} lead naturally to another conjecture which we now spell out. In Section \ref{subsec:conjecture} we explained a conjecture pertaining to the superposition of two \emph{independent} monomer-dimer configurations sampled according to the free boundary dimer model. But there is at least one other natural way to superpose two such configurations that are not independent: namely, when they come from the \emph{same} full plane dimer model. In fact, there are two ways to do the folding, depending on whether we shift by $i/2$ or not.

Let us explain this more precisely. Let us define the graph $\hat \cG$ which is obtained by adding to $\cG$  its reflection with respect to the real axis. The vertices of $\cG$ on the real axis (i.e., $V_0$) are not reflected: we only keep one copy of them in $\hat \cG$. (By contrast, in the graph $\bar \cG$, where $\cG$ is shifted by $i/2$ prior to reflection, these vertices are duplicated).

Now, consider an (infinite volume) dimer cover $M$ on $\hat \cG$, viewed as a subset of edges where every vertex has degree 1, and consider the superposition $\hat \Sigma$ obtained by superposing $M$ with itself via a reflection through the real line: thus,
 $$\hat \Sigma = M|_{\H} \cup (-M)|_{- \H}.$$
  Then $\hat \Sigma$ is a subgraph of degree two (including double edges), except for vertices on $V_0 \subset \R$ which in $M$ are connected to a vertical edge. Thus $\hat \Sigma$ is exactly of the same nature as the graph in Figure~\ref{F:double}. It is not hard to see that the ``height function'' $h_{\hat \Sigma}$ (really defined only up to a global additive constant) naturally associated with $\hat \Sigma$ converges in the fine mesh size limit to $(1/\pi) \hg^{\text{Neu}}_\H$: this is because at the discrete level, the corresponding height function $h_{\hat \Sigma}(f)$ at a face $f \subset \H$ can be viewed as $h_M(f) + h_M( \bar f)$ (where $h_M$ is the height function associated with $M$), and $h_M$ is known to converge to $(1/\sqrt{2} \pi)\hg_\C$~\cite{BdT_quadritilings}. These considerations lead us to the following conjecture:

\begin{conj} \label{C:refl}In the scaling limit, the collection of boundary-touching curves in $\hat \Sigma$ converges to the \textbf{Arc Loop Ensemble} ALE in the upper half-plane.
\end{conj}

We remark that it is also meaningful to fold a dimer configuration on $\bar \cG$ (rather than $\hat \cG$ above) onto itself via reflection through the real line. In that case, one must erase the vertical edges straddling the real line and view the corresponding dimers as pairs of monomers. The resulting superposition $\bar \Sigma$ is a subgraph of degree two, including multiple edges and double points (on $V_0 \subset \R + i /2$). In particular there are no boundary arcs in $\bar \Sigma$, except for degenerate lines connecting every monomer to itself. For the same reason as above, the height function $h_{\bar \Sigma}$ associated to $\bar \Sigma$ may be viewed as $h_M(f) - h_M(\bar f)$ and so converges in the scaling limit towards $(1/\pi) \hg^{\text{Dir}}_\H$. Analogously to Conjecture \ref{C:refl}, we conjecture that the loops of $\bar \Sigma$ converge to CLE$_4$.

\subsection{Connection with isoradial random walk with critical weights}

The following remark was suggested by an anonymous referee. There is a special value of the fuagcity parameter $z$, namely
\begin{equation}\label{E:z}
z^2 = \tan (\pi/8)
\end{equation}
such that the \emph{even} walk $\Zev$ coincides (after a small change in the embedding) with the random walk on isoradial graphs with critical weights considered in the work of  Kenyon~\cite{Kenyon_iso}. To see this, one can notice that the even walk $\Zev$ is equivalent to a random walk on two upper-half planes (or more precisely, square lattices on these half planes) welded together via a row of triangles. See Figure \ref{fig:isoradial}. Such a graph has an isoradial embedding and the corresponding critical weights have weight 1 on the square lattice edges, and weight $z$ given by \eqref{E:z} on the remaining triangle edges, as follows from elementary calculations. In that case, convergence of the derivative of the potential kernel (i.e., part of the inverse Kasteleyn matrix) would follow from Theorem 4.3 in \cite{Kenyon_iso}.

\begin{figure}
\begin{center}
 \includegraphics[scale=0.4]{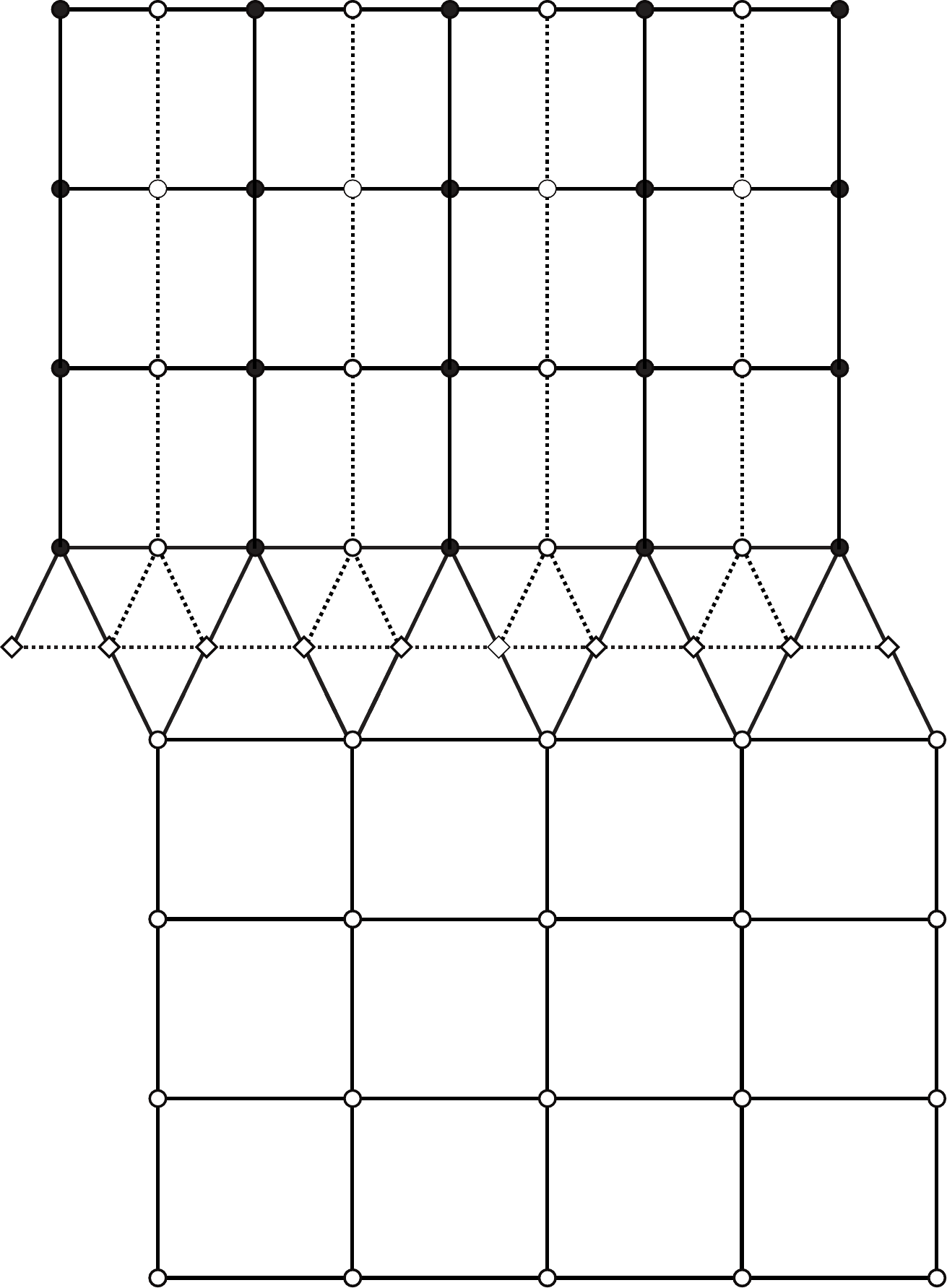}
\caption{
After reflecting the white sublattice of $V_{\textnormal{even}}$ (dashed lines), a graph composed of two square lattice half-planes glued together by a row of interlacing isosceles triangles is formed (solid lines).
If the angle between the legs of the triangle is equal to $\pi/4$, the graph is isoradial and the walk $Z_{\textnormal{even}}$ for $z^2 = \tan (\pi/8)$ is the same as the walk studied by Kenyon in~\cite{Kenyon_iso}.
}
 \label{fig:isoradial}
\end{center}
\end{figure}

\subsection{Outline of the paper and structure of proof}

This paper is organized as follows:
In Section~\ref{S:IKM} we first describe and slightly generalize the bijection from~\cite{GJL16} between the {free boundary dimer model} on $\cG$ and the standard dimer model on an
augmented (nonbipartite) graph $\cG^0$ (as in Figure~\ref{fig:Gtriangle}), which is the starting point of our paper.
We then define a Kasteleyn orientation on $\cG^0$, the associated Kasteleyn matrix $\tilde K$, and finally we choose a convenient complex-valued gauge changed Kasteleyn matrix $K$ (this gauge is closely related to the one of Kenyon~\cite{Kenyon_ci}
and allows one to interpret $K$ as a discrete Dirac operator).
Kasteleyn theory (which we recall later on in the paper) says that the correlations of the dimer model on $\cG^0$ (and hence also of the free boundary dimer model on $\cG$) can be computed from
the inverse Kasteleyn matrix $K^{-1}$.

\paragraph{Section \ref{S:IKM}}With an intention of developing its random walk representation,
we therefore begin analyzing the inverse Kasteleyn matrix when $\cG$ is a subgraph of the square lattice with appropriate boundary conditions described in Section~\ref{S:boundary}. To this end we look at the matrix $D =K^*K$, whose off-diagonal entries we interpret as (signed) transition weights.
These weights away from $\partial_{\textnormal{free}} \cG$ (which is a subset of the real line) are positive and hence define proper random walks as in~\cite{Kenyon_ci}.
However, the description of $D$ as a Laplacian matrix associated to a random walk breaks down completely for vertices on the three bottommost rows of $\cG^0$ (as in Figure~\ref{F:V01}).
We stress the fact that the level of complication is considerably higher for transitions between odd rows (that will lead to the definition of the walk $Z_{\textnormal{odd}}$).
Indeed, as mentioned in Figure~\ref{F:V01}, for even rows the arising walk $Z_{\textnormal{even}}$ can be relatively easily understood as a proper random walk reflected on the real line after taking into account
a \emph{global} sign factor appearing in $D$ (which leads to the formula in the second line of~\eqref{RWrep}).

Therefore the remainder of Section~\ref{S:IKM} is devoted to the random walk representation for $K^{-1}$, which is one of the main contributions of this paper.
The main idea is to ``forget'' the steps of the signed walk induced by $D$ taken along the row $V_{-1}$, or more precisely to only specify the trajectory of a path away from $V_{-1}$ and combine together all paths that agree with this choice.
The hope is that the resulting projected signed measure on trajectories contained in $V_{\textnormal{odd}}=V_1\cup V_3\cup \ldots$
with (unbounded) steps from $V_1$ to $V_1$ is actually a true probability measure. Remarkably (in our opinion),
we show that this is indeed the case; this phenomenon is what really lies behind the random walk representation of Theorem \ref{T:finitevol_intro}.
To achieve this, an additional (intermediate) limiting procedure is required.
To be precise, we first pretend that the rows $V_0$ and $V_{-1}$ of $\cG^0$ are infinite. This is done by defining graphs $\cG^N$, where $2N$ additional triangles are appended on both sides of $\cG^0$,
and then taking the limit $N\to \infty$. This allows us to perform exact computations for the transition weights from $V_1$ to $V_1$ by analysing the potential kernel of the auxiliary one-dimensional walk on $\Z$
defined in Section~\ref{sec:auxiliary}. The required positivity of the combined weights and the identity stating that these weights sum to one as we sum over all possible jump locations, stated in Lemma~\ref{T:eff}, is the result of an exact (and rather long) computation involving the potential kernel of this auxiliary walk.

This intermediary limit is also the technical reason for the introduction of the modified monomer weight $z'$, which arises as the limiting weight of the peripheral monomers on $\cG^0$.
Finally, in Section~\ref{S:inverseKfinite} we use the notion of  \textbf{Schur complement} of a matrix as a convenient tool to implement the idea of combining all the walks with given excursions away from (the now infinite) row $V_{-1}$.
All in all, at the end of Section~\ref{S:IKM} a random walk representation of $K^{-1}$ is developed and Theorem~\ref{T:finitevol_intro} is proved.

\paragraph{Section~\ref{S:IVL}.} The goal of this section is to prove Theorem~\ref{P:infinite_vol_intro}, i.e., to establish the infinite volume limit of the model when a sequence of graphs exhausts $\H\cap \Z^2$.
By Kasteleyn theory, it is enough that the inverse Kasteleyn matrix has a limit. This will be shown using the random walk representation established in Section~\ref{S:IKM}. Essentially, the main goal is to show that in the infinite volume limit, the difference of the Green function associated to the random walk $\Zev$ or $\Zodd$ at two fixed vertices $x,y$ converge to the {difference of the} \textbf{potential kernel} of the corresponding infinite volume walk. In fact, the very definition of this potential kernel is far from clear and occupies us for a sizeable part of this section. For the usual simple random walk on the square lattice, the definition of the potential kernel (see e.g. \cite{LawLim}) relies on precise estimates for the random walk coming from the exact computation of the Fourier transform of the law of random walk. Such an exact computation is clearly impossible here, since the effective walks cannot be viewed as a sum of i.i.d. random variables. We overcome this obstacle by developing a general method (which we think may be of independent interest) to define the potential kernel of a recurrent random walk and prove convergence of Green function differences towards it. The main idea is to proceed by \textbf{coupling}. We note that a similar idea has also been recently advocated by Popov (see Section 3.2 of \cite{popov}); but the approach in \cite{popov} also takes advantage of some properties and symmetries which are not available here. Instead, our starting point is the robust estimate of Nash (see e.g. \cite{Barlow}) characterising the heat kernel decay. With our approach, only a weak (polynomial of any order) bound for the probability of non coupling suffices to show the existence of the potential kernel. An immediate byproduct of our quantitative approach (which is crucial for us) is the proof of the desired convergence of Green function differences towards the {differences of the} potential kernel, obtained in Proposition \ref{P:greenPK_gen}.

\paragraph{Section \ref{S:PKscaling}} In Section \ref{S:PKscaling} we move on to describe the scaling limit (now in the limit of fine mesh size) for the potential kernel of the effective walks $\Zev$ and $\Zodd$. A key idea is to say that when such a walk hits the real line, it will hit it many times and therefore has a probability roughly 1/2 to end up at a vertex with even (resp.\ odd) horizontal coordinate once it is reasonably far away from the real line. This idea eventually leads us to asymptotic formulae for the potential kernel which depends on the parity of the horizontal coordinate of a point (see Theorem \ref{T:PKscaling}). To achieve this, we introduce an intermediary process which we call \textbf{coloured random walk}, which is a random walk on (twice) the usual square lattice, but which can also carry a colour (representing, roughly speaking, the actual parity of the effective walk). This colour may change only when the walk hits the real line, and then does so with a fixed probability $p$. The proof of Theorem \ref{T:PKscaling} relies on first comparing our effective walk to the coloured random walk (Proposition \ref{P:eff_col}) and then from the coloured walk to half of the potential kernel of the usual simple random walk (Proposition \ref{P:col_half}).

\paragraph{Section~\ref{S:SL}} We are now finally in a position to start the proof of Theorem~\ref{T:NGFF_intro}. From Theorem \ref{T:PKscaling} we obtain a scaling limit for the inverse Kasteleyn matrix of the (infinite volume) free boundary dimer model.
After recalling Kasteleyn theory in the nonbipartite setting, we then  compute the scaling limit of the pointwise moments of height function differences on $\H$ in Section~\ref{S:moments}.
The argument is based on Kenyon's original computation~\cite{Kenyon_ci} but with substantial modifications coming from the fact that we use Pfaffian formulas instead of the determinantal formulas for bipartite graphs. This leads to different expressions which fortunately simplify asymptotically (for reasons that are related but distinct from those in \cite{Kenyon_ci}). This leads to the formula in Proposition \ref{P:kpoint_dimer}, which is an asymptotic expression for the limiting joint moments of pointwise height differences, with an explicit quantification of the validity of the limiting formula (needed in the following).
To finish the proof of the result, we transfer this result in Section~\ref{S:final} into one about the scaling limit of the height function as a random distribution. This is essentially obtained by integrating the result of Proposition \ref{P:kpoint_dimer}, but extra arguments are needed for the case when some of the variables of integration are close to one another.

\begin{remark}
\label{rem:referee}
An alternative strategy for establishing the scaling limit of the inverse Kasteleyn matrix, suggested by an anonymous referee, would be the following. It would suffice to concentrate on one of the two types of walks (say $\Zev$, which is simpler to define than its counterpart $\Zodd$) and analyse its potential kernel in the manner indicated above in order to derive the scaling limit of $K^{-1} (u, \cdot)$ where $u \in \Ve$ is a given even vertex. Once this is done, discrete holomorphicity and antisymmetry can be invoked to obtain asymptotics of $K^{-1}$ on the remaining vertices, and with the same error bounds (using a discrete version of the Poisson formula for the derivative of harmonic functions).

We have chosen not to implement this strategy for the following reasons. On the one hand, the asymptotic analysis of $\Zev$ (and in particular its potential kernel) is as difficult as it is for $\Zodd$.  As this is probably the most challenging part of the analysis,  there would be no real simplification in considering $\Zev$ only. On the other hand, the exact random walk representation of $K^{-1}$ seems interesting in its own right, especially since it shows a connection with reflected random walks even at the discrete level.

\end{remark}

We end the introduction by mentioning the following problem. The dimer model on special families of bipartite planar graphs is famously related, through various measure preserving maps, to other classical models of statistical mechanics like spanning trees (see e.g.~\cite{dimer_tree}), the double Ising model~\cite{dubedat,bdt} or the closely related double random current model~\cite{DumLis}. This indicates the following direction of study.
\begin{problem}
 Analyse the boundary conditions in these classical lattice models induced by the presence of monomers in their dimer model representations.
\end{problem}


\paragraph{Acknowledgements.} N.B.'s work was supported by: EPSRC grant EP/L018896/1, University of Vienna start-up grant, and FWF grant P33083 on ``Scaling limits in random conformal geometry''. The authors are also grateful for an invitation to visit the D\'epartement de Math\'ematiques et Applications at Ecole Normale Sup\'erieure in April 2018, where part of this work took place, and during which N.B. was an invited professor. The hospitality of that department, and the stimulating atmosphere of the discussions (including with D. Chelkak, who helped us with the computation \eqref{gamma}) are gratefully acknowledged. {We also thank the anonymous referee who brought to our attention the alternative strategy outlined in Remark~\ref{rem:referee}, and the connection to isoradial walks in \eqref{E:z}.}

\section{(Inverse) Kasteleyn matrix} \label{S:IKM}

\subsection{Dimer representation}\label{S:bij}
In \cite{GJL16} a representation of the free boundary dimer model was given in terms of a dimer model on an augmented (nonbipartite) graph where a row of triangles is appended to $\partial_{\textnormal{free}}\cG$ (see Figure~\ref{fig:Gtriangle}).

\begin{figure}
\begin{center}
 \includegraphics[scale=0.9]{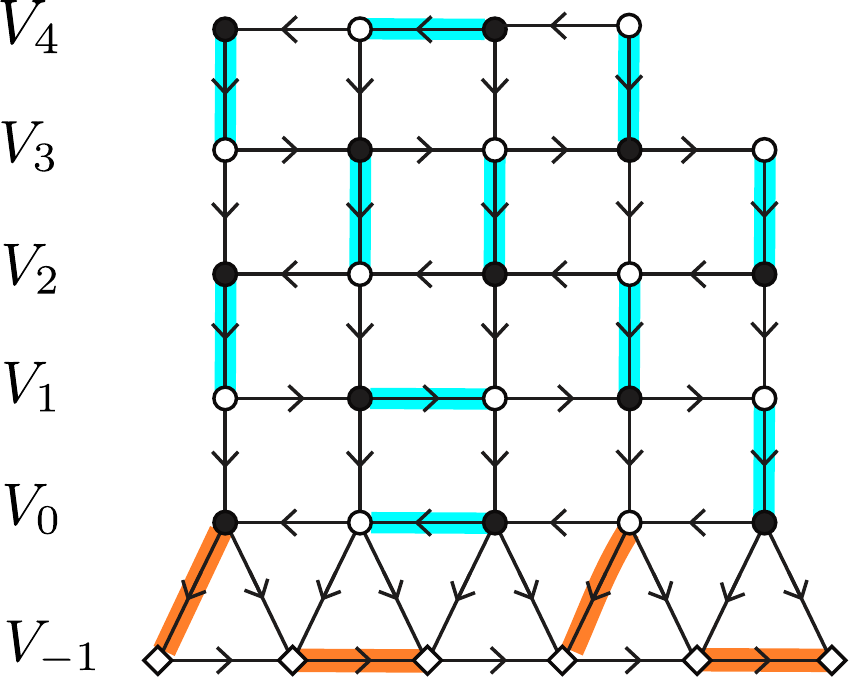}
\caption{An augmented non-bipartite graph $\cG^{0}$ and its Kasteleyn orientation. The graph is constructed from a piece of the square lattice $\cG$ with $\partial _m \cG=V_0$ by adding the bottom row of triangles.
In this case $\cG$ has two black monomer-corners, two black dimer-corners, and two white dimer-corners. The additional row of triangles (here $\cT_k$ with $k=9$) simulates the presence of monomers in the free boundary dimer model by means of a standard dimer model. This particular choice of $k$ forces and even number of monomers as $k-\lfloor k/2\rfloor +1$ is even.
This is expressed as a measure-preserving bijection between $\mathcal{D}(\cG^0)$ and $\mathcal{MD}(G)$ with a proper choice of weights}
 \label{fig:Gtriangle}
\end{center}
\end{figure}
We first slightly generalize the result contained in \cite{GJL16} in order to account for the case when $\partial_{\textnormal{free}} \cG$ is not the whole boundary of $\cG$.
To this end, for $k>0$, let $\cT_{k}'$ be the graph composed of $k$ triangles glued together in a manner like the bottom part of the graph in Figure~\ref{fig:Gtriangle} where we assume that
the left-most triangle is $\bigtriangledown$ for $k$ even and $\bigtriangleup$ for $k$ odd. Let $\cT_k$ be $\cT'_k$ with all top horizontal edges removed,
and let $\cT_0$ be a single edge interpreted as one non-horizontal side of a triangle.
Let $\partial \cT_{k}$ be the upper row of vertices in~$\cT_k$, and note that $| \partial \cT_{k}|=\lfloor k/2\rfloor +1$.
\begin{lemma}\label{L:mon_dim}
For every $k\geq0$ and every choice of $W \subseteq \partial \cT_k$ with $|W| = | \partial \cT_{k}|=\lfloor k/2\rfloor +1\ (\textnormal{mod}\ 2)$, there exists exactly one dimer cover of
$\cT_k\setminus W$, where with a slight abuse of notation, $\cT_k\setminus W$ is $\cT_k$ with the vertices in $W$ and all adjacent edges removed.
\end{lemma}

\begin{proof}
The statement for $k=0,1,2,3,4$ can be easily checked by hand.
Let $\{v_1,v_2\}$ be the first two vertices of $\partial \cT_{k}$.
We now consider cases for $\cT_{k+5}$:
\\
\textbf{case I.} $|W\cap \{v_1,v_2\}|\in \{0,2\}$. Then, there is exactly one choice of dimers on the first two triangles which corresponds to either situation.
One hence reduces the problem to the case of $\cT_{k+1}$.
\\
\textbf{case II.} $|W\cap \{v_1,v_2\}|=1$. Then, there is exactly one choice of dimers on the first triangle for $k$ even and on the first three triangles for $k$ odd which correspond to either situation.
One hence reduces the problem to the case $\cT_{k+2}$ for $k$ even and $\cT_{k}$ for $k$ odd.
\end{proof}

This lemma implies that by gluing the graph $\cT_k$ (with a proper choice of $k$) to $ \cG$ so that $\partial_{\textnormal{free}} \cG$ and $\partial \cT_k$ match, and by considering the dimer model on this extended graph $\cG^0$
(with dimer weight $z_v$ for the two non-horizontal edges of the triangle incident on a vertex $v\in \partial_{\textnormal{free}} \cG$) one can simulate the free boundary dimer model on $\cG$ with monomers on $\partial_{\textnormal{free}} \cG$ and with monomer weight $z_v$ for $v\in \partial_{\textnormal{free}}\cG$.
Indeed, it is enough to interpret the set $W$ from the lemma above as the set of vertices of $\cG$ that belong to a dimer, and $\cT_k\setminus W$ as the set of monomers.

In other words, there is a measure preserving bijection between $\mathcal{D}(\cG^0)$ and $\mathcal{MD}(\cG)$. Note that if $\cG$ has a dimer cover (which we assume in this article),
then one has to take $ k-\lfloor k/2\rfloor +1$ even, as there has to be an equal number of white and black monomers.

\subsection{Kasteleyn orientation, Kasteleyn matrix and gauge change}

A Kasteleyn orientation of a planar graph is an assignment of orientations to its edges such that for each face of the graph, as we traverse the edges surrounding this face one by one in a counterclockwise direction, we encounter an odd number of edges in the opposite direction (see e.g.\ \cite{toninelli_notes}). For graphs as defined in Section~\ref{S:boundary} we make the following choice (see Figure \ref{fig:Gtriangle}): every vertical line is oriented downwards (including the non-horizontal sides of triangular faces at the bottom). The orientation of horizontal edges alternates: in odd rows (starting at row $-1$): edges are oriented from left to right, whereas in even rows (starting at row 0) they are oriented from right to left.

Given a Kasteleyn orientation, the standard Kasteleyn matrix  $\tilde K(x,y)$ is taken to be the signed, weighted adjacency matrix: that is, $\tilde K (x,y) = \pm1_{x \sim y} w_{(x,y)}$ where the sign is $+$ if and only if the edge is oriented from $x$ to $y$, and the weight $w_{(x,y)}$ is 1 for horizontal and vertical edges (including on $V_{-1}$), and $z$ for the nonhorizontal sides of triangular faces.  However, it will be useful to perform a change of gauge, as follows. For every $k \ge 0$ even, and for every $x \in V_k$, we multiply by $i$ the weight of every edge adjacent to $x$. In particular, every horizontal edge in $V_k$ with $k$ even receives a factor of $i$ twice coming from both of its endpoints, whereas each vertical edge receives a factor of $i$ exactly once. We define the gauge-changed Kasteleyn matrix $K(x,y)$ to be the resulting matrix. Formally,
\begin{equation}\label{def:Kasteleyn}
K(x,y) = \tilde K(x,y) i^{1_{x \in \Ve} + 1_{y \in \Ve}}.
\end{equation}
For instance, if $x \in V_{0}$ is not on the boundary, then $x$ has five neighbours. Starting from the vertical edge and moving counterclockwise, the weights $K(x,y)$ are given by $-i, -1, iz, iz, 1$.

\subsection{Towards the inverse Kasteleyn matrix} \label{S:D-1}

Let $D=K^*K$. In this section we explain the key idea involved in computing $D^{-1}$, and thus ultimately $K^{-1}$. The matrix $D$
already played a crucial role in \cite{Kenyon_ci}, where Kenyon
observed that it reduced to the Laplacian on the four types of sublattices of the square grid.

We will follow a similar approach but, as we will see, the immediate interpretation of $D$ as a Laplacian breaks down in the rows $V_{-1}, V_0$ and $V_1$. Nevertheless, admitting the formal sum-over-all-paths identity~\eqref{eq:Greenformal}, we will be able to make a guess on the structure of $D^{-1}$. This will ultimately lead us to the identification of $D^{-1}$ as the Green's function of a certain \textbf{effective} random walk (or, in fact, a pair of effective random walks) which appear in the statement of Theorem \ref{T:finitevol_intro}.

Therefore, the purpose of this section is mostly to explain the heuristic principles guiding the proof, and to introduce the relevant objects and the notation. Once this framework is defined we will start with the actual proof in Section \ref{S:IL}. We will complete the rigorous computation of $D^{-1}$ (and therefore the proof of Theorem \ref{T:finitevol_intro}) in Section~\ref{S:inverseKfinite}.

\begin{figure}
\begin{center}
\includegraphics{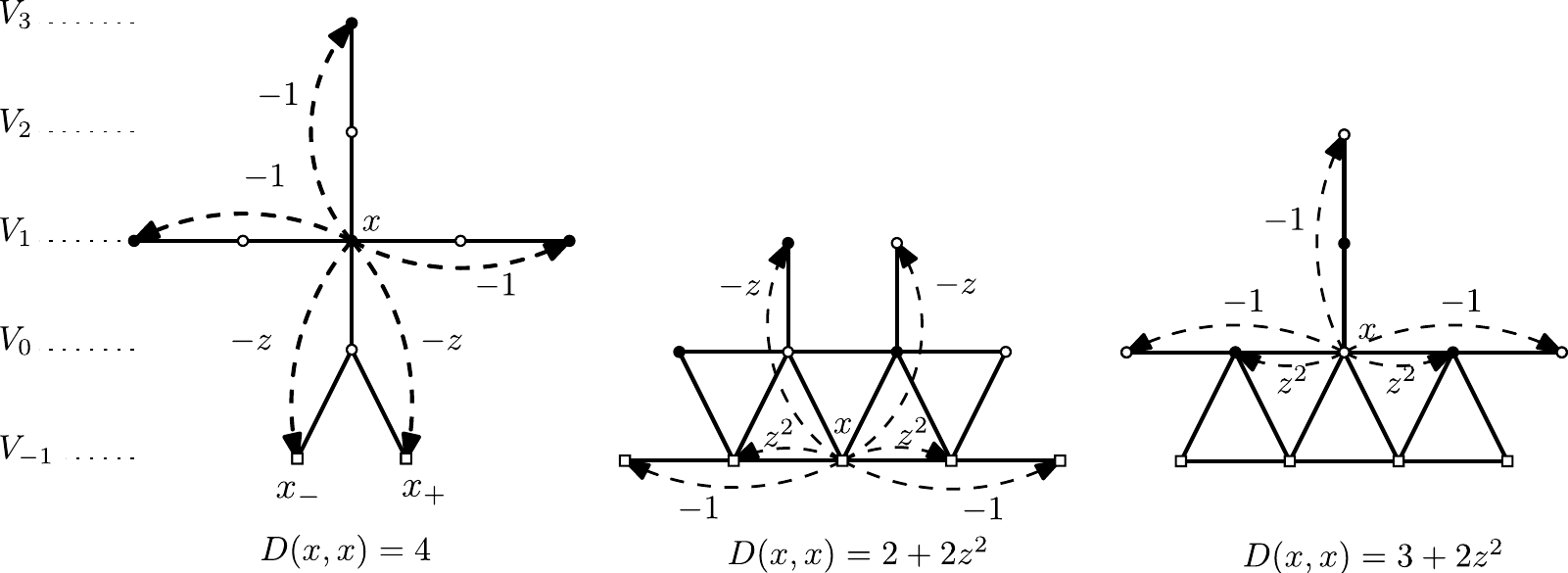}
\caption{Three types of vertices $x$ where the transition weights of $D=K^*K$ are signed. The arrows indicate the corresponding value of $D(x,y)$. Note the following crucial observations: First, in the rightmost case (when $x\in V_0$), the absolute values of the transition weights sum up to the diagonal term. Moreover, the transition weight is negative if and only if
the size of the step is odd (more precisely equal to one). 
A similar observation holds in the central picture (when $x\in V_{-1}$) if one ignores the transition weights that lead back to~$V_1$. This is the basis for the construction of Section~\ref{S:D-1} and
the definition of the auxiliary random walk on $\Z$ from \eqref{D:rw}. Our approach is to ``forget'' what the walk does when it stays in $V_{-1}$ and resum over all trajectories contained in $V_{-1}$ and
with the same endpoints in $V_1$. 
 }
\label{F:V01}
\end{center}
\end{figure}

We now fix a finite arbitrary graph $\cG$ that satisfies the conditions of Section \ref{S:boundary}.
We first compute $D$ explicitly. Note that if $x \in V_k$ with $k \ge 2$, the entries of $D$ are computed in a way identical to Kenyon \cite{Kenyon_ci}. Namely, the diagonal term is
$$D(x,x) = K^*K(x,x) = \sum_{y \sim x} K^*(x,y) K(y,x) = \sum_{y \sim x} | K (y,x)|^2 = \deg(x).$$
Moreover,
the off-diagonal terms are nonzero if and only if $y$ is at distance two from $x$, but not diagonally (the diagonal cancellation is a consequence of the Kasteleyn orientation), i.e., if $y$ is a neighbour of $x$ on one of the sublattices $2\Z \times 2 \Z$, $(2\Z +1) \times (2\Z + 1)$, $2\Z  \times (2\Z + 1)$ or $(2\Z +1) \times 2\Z $ in which case one can check as above that $D(x,y) =-1$.
Therefore away from the boundary $\partial_{\textnormal{free}} \cG$,
in the same way as in \cite{Kenyon_ci}, $D$ is the Laplace operator associated to a simple random walk on each of the sublattices, up to a multiplicative constant. The Temperleyan boundary conditions are then naturally associated with certain boundary conditions for $D$ on these sublattices.

Complications for such an interpretation arise when $x \in V_{-1} \cup V_{0} \cup V_1$. See Figure~\ref{F:V01} for the nonzero entries of $D$ in these cases. Notice that now it is not necessarily true that
the diagonal term $D(x,x)$ is (up to a sign) the same as the sum of the off-diagonal entries on the row corresponding to $x$, or in other words, the transition weights $d_{x,y}$ in \eqref{eq:dweight} do not sum up to $1$. Moreover, some of them are negative.
While this seems like a very serious obstacle for describing the behaviour of the operator $D^{-1}$ in the scaling limit, we nevertheless show in the next section how we can recover an effective random walk for which $D$ really is the Laplacian.

More precisely, $D^{-1}$ can be formally viewed as a sum of weights of paths of all possible lengths,
where the weight of a path is the product of (signed) transition weights of individual jumps.  That is, formally,
\begin{equation}\label{eq:Greenformal}
D^{-1} (u,v) = \frac1{D(v,v)} \sum_{\pi: u \to v} w(\pi),
\end{equation}
where
\begin{align} \label{eq:dweight}
w(\pi)=\prod_{(x,y) \in \pi} d_{x,y} \qquad \text{with} \qquad d_{x,y}=- \frac{D(x,y)}{D(x,x)}.
\end{align}
For $x$ in the bulk, $d_{x,y} =1/4$ for each $y$ which is neighbour of $x$ on the sublattice of twice larger mesh size containing $x$, and is $0$ otherwise, which is the same as the transition probability of a simple random walk on that sublattice.

Let us now point out that the transition weights between an even row and an odd row are always $0$. Compared to the odd rows, the construction for even rows is much simpler. As seen in Figure~\ref{F:V01},
for $x\in V_0$ which is not an extremity of the row $V_0$ or at distance one from the extremities,  $D(x,x)$ is in fact equal to the sum of $|D(x,y)|$ for all $y\not=x$.
We can therefore view $|d_{x,y}|$ for $x\in V_{0}$ as the transition weights of a random walk that is reflected on row $V_0$ (and can make jumps of size one and two on that row). When $x$ is one of the extremities of $V_0$ or is at distance one from the extremities, the values of $D(x,x)$ and $D(x,y)$ allow us to interpret it as  a killing or a reflection of the random walk at the boundary (see Section~\ref{S:1D} and in Figure~\ref{fig:GN} for more details).
When we take into account the signs of $d(x,y)$ in \eqref{eq:dweight}, this gives rise to a global sign factor which depends only on $u$ and $v$ can be seen in the second line of~\eqref{RWrep}.

The rest of this section is devoted to the more complicated task of giving a random walk representation to $D^{-1}$ restricted to the vertices in odd rows $V_{\textnormal{odd}}$.
We now describe the main idea.
We will manage to give a meaning to the right hand side of \eqref{eq:Greenformal} by fixing a specific order of summation. We will later on prove that this definition really does give us the inverse of $D$, and we will also find a random walk interpretation to this definition. We emphasise this because the signs are not constant, and hence the order of summation is a priori relevant to the value of the sum.
Essentially we will compute the sum in \eqref{eq:Greenformal} by ignoring the details of what the path does when it visits $V_{-1}$. That is, we will identify two paths if they enter $V_{-1}$ at the same place in $V_{1}$ and leave $V_{-1}$ at the same places in $V_1$ for each visit to $V_{-1}$, and we will be able to estimate contributions to \eqref{eq:Greenformal} coming from each such equivalence class.

An important observation (see Figure~\ref{F:V01}) here is that for each $x\in V_{-1}$ which is not the extremity of $V_{-1}$ or at distance one from the extremities, the diagonal term $D(x,x)$ is  equal to the sum of $|D(x,y)|$ for all $y\in V_{-1}$ not equal to $x$. Note that $D(x,y)$ is non zero for $y=x\pm1$ or $y=x \pm 2$. This allows us to express the weight of the paths which stay in $V_{-1}$ as the weight of a random walk with steps $\pm 1$ and $\pm 2$ on $V_{-1}$. For $x\in V_{-1}$ which is equal to the extremity of $V_{-1}$ or at distance one from the extremities, the values of $D(x,x)$ and $D(x,y)$  again allow us  to interpret it as  a killing or a reflection of the random walk at the boundary (see Section~\ref{S:1D} and in Figure~\ref{fig:GN} for more details).
One can therefore associate a Green's function $g(\cdot, \cdot)$ with the random walk on $V_{-1}$ with transition probabilities
\begin{align} \label{eq:pxy}
p_{x,y}=|d_{x,y}|.
\end{align}

For $x\in V_1=V_1(\cG^0)= V_1(\cG)$, let $x_-$ and $x_+$ be the left and right vertex in $V_{-1}=V_{-1}(\cG^0)$ two steps away from $x$.
We fix $u,v \in V_1$ and let $u_\bullet \in \{ u_-, u_+\}$ and $v_\bullet \in \{ v_-, v_+\}$. We define $\cP^{1}_{u_\bullet, v_\bullet} $ {to} be the set of paths from $u_\bullet$ to $v_\bullet$ which are contained in $V_{-1}$. Observe that if $\pi \in \cP^{1}_{\ub,\vb}$, then $\pi$ makes jumps of size $\pm 1$ or $\pm 2$, and that each odd jump contributes a negative weight to \eqref{eq:Greenformal} whereas each even jump contributes a positive weight.
Since $\pi$ goes from $u_{\bullet}$ to $v_{\bullet}$ the parity of the number of even and odd jumps is fixed and depends only on the distance between $u_{\bullet}$ and $v_{\bullet}$ in $V_{-1}$. Hence
$$
w(\pi)= (-1)^{\Re(\vb - \ub)}\prod_{(x,y) \in \pi} |d_{x,y}|,
$$
where $d_{x,y}$ is defined in \eqref{eq:dweight}.

Going further: if $\cP^{1}_{\ub, v}$ is the set of paths going from $\ub $ to $v$ and staying in $V_{-1}$ (except for the last step, which must be from $v_\pm$  to $v$), then
\begin{align} \label{eq:P1}
 \sum_{\pi \in \cP^1_{\ub, v}} w(\pi) &= (-1)^{\Re(v_+ - \ub)} (g(\ub, v_+)- g(\ub, v_-) ) \frac{z}{2+ 2z^2} 
\end{align}
where the last term accounts for the weight $- D(v_\pm, v) / D(v_\pm, v_{\pm})$ of the last step from $V_{-1}$ to $V_1$.
Finally, let $\cP^1_{u,v}$ be the set of paths from $u$ to $v$ which stay in $V_{-1}$ except for the first and last step (which necessarily are from $V_1$ to $V_{-1}$ and vice versa). Using \eqref{eq:P1} we have
\begin{align}
 \sum_{\pi \in \cP^1_{u, v}}w(\pi) =\frac {z^2} {8+8z^2} (-1)^{\Re(v-u)} (g(u_+,v_+) - g(u_+,v_-)-g(u_-,v_+)+g(u_-,v_-)) =:\tfrac14q_{u,v},
   \label{q}
\end{align}
where the additional term $\frac z4 $ compared to \eqref{eq:P1} accounts for the weight $- D(u, u_{\pm}) / D(u, u) $ of the first step from $V_1$ to $V_{-1}$.
The factor $\frac14$ in the definition of $q_{u,v}$ is included for later convenience.

Recall that our intention is to interpret the quantities $q_{u,v}$ as transition probabilities between vertices in $V_1$. In particular we would wish $q_{u,v}$ to be positive and sum up to (something less than) one
(since the other three transition weights induced by $D$ from a vertex in the bulk of $V_1$ to $V_1$ and $V_3$ are equal to $3/4$).
Unfortunately, in the setting described so far, we were unable to do so (this is because the exact analysis of the random walk on $V_{-1}(\cG^0)$ with its particular boundary conditions is not easy). However, a nice solution to this problem, which effectively gets rid of boundary conditions, is the following construction. We note that this construction is the reason for the appearance of the special
monomer weight $z'$ at the monomer-corners in the statement of our results.

\subsection{An intermediate limit} \label{S:IL}

To overcome the issue raised above, we introduce an intermediate limiting procedure in our model.
To this end, let $\cG^N$ be the graph $\cG^0$ to which we append $2N$ triangles on either side of $\cG$ along $V_{-1}$ and $V_0$ (see Figure~\ref{fig:GN} for an example). We assign weights $1$ to every edge except if it belongs to a triangle and is not horizontal, in which case we assign weight $z$. Since we assumed that $\cG$ has a dimer cover, it is easy to see that $\cG^N$ also has at least one dimer cover. We can hence talk about the dimer model on $\cG^N$ with the specified weights.

\begin{figure}
\begin{center}
\includegraphics[scale=1]{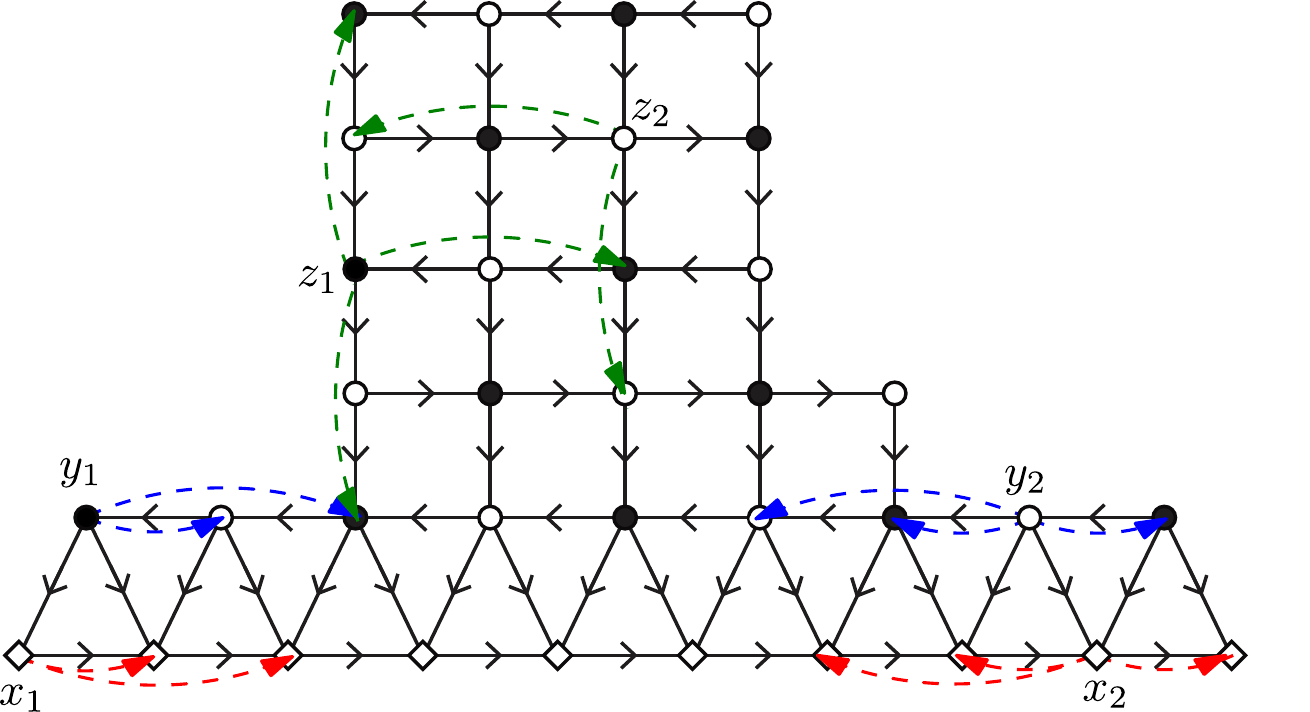}
\caption{A graph $\cG^2$ (there are $4$ additional triangles appended on each side of $\cG^0$) with different types of vertices and transitions marked. The transitions weights given by $D(\cdot,\cdot)=D_N(\cdot,\cdot)$ are $-1$ for long arrows and $z^2$ for short arrows. The diagonal terms
are $D(x_1,x_1)=1+z^2$, $D(y_1,y_1)=1+2z^2$, $D(x_2,x_2)=D(y_2,y_2)=2+2z^2$, $D(z_1,z_1)=3$, $D(z_2,z_2)=4$.
The black vertex $z_1$ has Neumann boundary conditions for the associated walk, since the total weight of outgoing transitions is also $3$.
The white vertex $z_2$ has Dirichlet boundary conditions since the total outgoing weight is $2<4$}
\label{fig:GN}
\end{center}
\end{figure}

Using Lemma \ref{L:mon_dim}, we can also rephrase this dimer model as a free boundary dimer model on $\cG$ to which we add a segment of $N$ edges to the left and right of $\partial_{\textnormal{free}} \cG$.
The first observation is that the monomer-dimer configuration on $\cG^N$ restricted to $\cG$, in the limit $N\to \infty$ has the law of the free boundary dimer model with
weight $z'$ from~\eqref{zprime} at the monomer-corners. This is an immediate consequence of the following elementary lemma.
\begin{lemma} Let $\cZ_N$ be the partition function of the monomer dimer model on a segment of $\mathbb{Z}$ of length $N$ with monomer weight $z$ and
edge weight $1$.
Then, as $N \to \infty$,
\[
\frac{\cZ_{N+1}}{\cZ_N} \to z', \quad  \text{where} \quad z' = \frac{z}2 + \sqrt{1+\frac{z^2}4}.
\]
\end{lemma}
\begin{proof}
It is enough to solve the recursion $\cZ_{N+1} = z \cZ_N +\cZ_{N-1}$ to get that
\[
\cZ_N= \Big (\frac12 -\frac{z}{4 \beta} \Big)\Big(\frac{z}2- \beta\Big)^N+\Big(\frac12 +\frac{z}{4 \beta}\Big)\Big(\frac{z}2+ \beta\Big)^N,
\]
where $\beta = \sqrt{1+\frac{z^2}4}$.
\end{proof}

Let $K_N$ be the Kasteleyn matrix of $\cG^N$ and let $D_N=(K_N)^*K_N$. The statement above and Kasteleyn theory imply that the inverse Kasteleyn matrix $K_N^{-1}$ restricted to $\cG^0$
converges as $N\to \infty $ to the inverse Kasteleyn matrix $(K')^{-1}$ for the free boundary dimer model on $\cG^0$ with monomer weights $z'$ at the monomer-corners.

\subsection{An auxiliary walk on $\Z$} \label{sec:auxiliary}
It will be convenient to consider a random walk on $V_{-1}(\mathbb Z^2) \simeq \Z$ with transition probabilities given by
\begin{equation}\label{D:rw}
p^{\infty}_{x,x\pm1}= \frac{z^2}{2+2z^2}=:1/2-p, \qquad p^{\infty}_{x,x\pm2}= \frac{1}{2+2z^2}=:p.
\end{equation}
In other words, this is the infinite volume version of the walk from~\eqref{eq:pxy}.
Now, while the Green's function of this walk is infinite since the walk is recurrent, its differences makes sense in the form of the \textbf{potential kernel} (see \cite{LawLim}, Section 4.4.3) given by
\begin{align}\label{D:pk}
\alpha_k = \sum_{n=0}^{\infty} (p_n(0)-p_n(k)) = \lim_{N\to \infty}\Big ( \sum_{n=0}^N p_n(0) -  \sum_{n=0}^N p_n(k) \Big),
\end{align}
where $p_n(k)= \sum_{i=0}^n \mathbb{P}_0( X_i = k )$ with $X$ being the random walk with jump distribution \eqref{D:rw}.
Using the potential kernel, for $u,v \in V_1(\Z^2)\simeq \Z$, we can now define the infinite volume version of the transition weight $q_{u,v}$ from~\eqref{q} by
\begin{align} \label{qinfty}
q^\infty_{u,v}= \frac {z^2} {2+2z^2} (-1)^{k+1}  (  2\alpha_k - \alpha_{k+1} - \alpha_{k-1}),
\end{align}
where $k = \Re (v-u)$. Note that the sign is opposite to that in \eqref{q}. To is due to the $-p_n(k)$ term in the definition of the potential kernel.

The next result is one of the crucial observations in this work.
\begin{lemma}[Effective transition probabilities]
\label{T:eff} For all $z>0$, and any pair of vertices $u,v \in V_1(\Z^2)$, we have
\[
q^{\infty}_{u,v} \ge 0 \quad \text{ and } \quad \sum_{v \in V_1(\Z^2)} q^{\infty}_{u,v} = 1.
\]
Moreover, $q^{\infty}_{u,v} \to 0$ exponentially fast as $|u-v|\to \infty$.
\end{lemma}
Before we give the proof note that neither of these two facts is at all clear from the definitions of $q^{\infty}_{u,v}$. Together they imply that we can think of $q^\infty_{u,v}$ as the step distribution of some \textbf{effective random walk} on $V_1$. Later in Proposition~\ref{P:GreenPK1d}, we will prove that $q^{\infty}_{u,v}$ is the limit of $q_{u,v}$ from \eqref{q} on $\cG^N$ when $N\to \infty$. As mentioned before, the proof of Lemma~\ref{T:eff} is just an exact computation of the potential kernel $\alpha$, and a conceptual understanding of why it is true is the subject of Problem~\ref{prob:1}.

\begin{proof}[Proof of Lemma~\ref{T:eff}]
The proof is based on an exact formula for the potential kernel $\alpha$ of the walk on $\mathbb{Z}$ defined by \eqref{D:rw}. To start with, by Theorem 4.4.8 from~\cite{LawLim} we know that
\[
\alpha_k = \frac{|k|}{\sigma^2} + A + O(e^{-\beta | k|})
\]
for some constants $\beta>0$ and $A\in \R$, and where $\sigma^2=1+6p$ is the variance of the walk with $p$ as in~\eqref{D:rw}.
Moreover, $\alpha$ is harmonic (except at $k=0$) with respect to the Laplacian of the walk \eqref{D:rw}. This implies that the $O(e^{-\beta | k|})$ term is of the form $\sum_i B_i\gamma_i^{|k|}$ for some constants $B_i$
and $\gamma_i$ satisfying $|\gamma_i|<1$ and
\[
1=(1/2-p)(\gamma_i+\gamma_i^{-1})+p(\gamma_i^2+\gamma_i^{-2}).
\]
We solve
and get only one such $\gamma=\gamma_i$ equal to
\begin{align} \label{gamma}
\gamma= \sqrt{\Big( \frac12 +\frac1{4p}\Big)^2-1} -\frac12 -\frac1{4p} \in (-1,0)
\end{align}
(the second solution is $\gamma=1$ and does not satisfy $|\gamma|<1$).
We can therefore write
\[
\alpha_k = \frac{|k|}{\sigma^2} + A + B \gamma^{|k|}
\]
for some constants $A$ and $B$.
Using that $\alpha_0=0$ by definition, we get $A=-B$ and hence
\begin{align} \label{eq:PKformula}
\alpha_k= \frac{|k|}{1+6p} -B+B\gamma^{|k|}.
\end{align}
We still need to compute $B$ which is equivalent to computing $\alpha_1$. Let $X$ be the walk with transition probabilities~\eqref{D:rw}.
Let $\tau=\inf \{n>0: X_n >0 \}$, and
\[
q = \P_0(X_\tau=1) \qquad \text{and} \qquad 1-q= \P_0(X_\tau=2).
\]
Then, by considering the possible four different first steps ($+1,-1,+2,-2$) of $X$ and using translation invariance and the strong Markov property, we get that
\begin{align*}
q=(\tfrac12-p)+(\tfrac12-p)((1-q)+q^2)+ p(q(1-q)+q^3+(1-q)q),
\end{align*}
which simplifies to
\begin{align}\label{eq:pq}
p q^2 +\big (\frac12-p\big )(q -1)=0.
\end{align}
One can check that $q=\gamma+1$.
Moreover, using the symmetry of jumps of $X$ and the Markov property for the walk, we get the equation (again considering the
first four steps in the same order)
\begin{align}
\alpha_1&= 1 + (\tfrac12-p)(-\alpha_1)+(\tfrac12-p)[q\alpha_1+(1-q)(-\alpha_1)] \nonumber \\
&+p[(1-q)\alpha_1+q(-\alpha_1)]+p[(q^2+(1-q))\alpha_1+q(1-q)(-\alpha_1)].\label{rec_alpha1}
\end{align}
To justify \eqref{rec_alpha1}, one starts from the definition of $\alpha_1$ in \eqref{D:pk} as the limit as $N \to \infty$ of the expected difference of number of visits by time $N$ to the sites 0 and $1$. We first apply the simple Markov property at the first step, and depending on the outcome of the first step, apply the strong Markov property at the next time $\tau$ (after time 1) that the walk returns to 0 or 1, taking care of the contribution coming from the event $\{\tau>N\}$. We then let $N \to \infty$. There is no problem in doing so, first because the sequence $\{\alpha_n\}_{n \ge  0}$ is bounded, which lets us use the dominated convergence theorem, and second because the contribution coming from the event $\{\tau>N\}$ to the difference between the number of visits at 0 and 1 by time $N$ is bounded by 1. Details are left to the reader.

Together with~\eqref{eq:pq}, \eqref{rec_alpha1} gives
\begin{align}
\alpha_1=\frac{1}{(1 + 2 p (-1 + q)) (2 - q)} = \frac{1}{(1 + 2 p\gamma) (1 - \gamma)}.
\end{align}
and hence from~\eqref{eq:PKformula} we obtain
\begin{align}\label{B}
B=\frac{4p}{(\gamma-1)(6p+1)(2p\gamma+1)} \leq 0.
\end{align}

We can now define
\[
q_k = (-1)^{k+1} \frac{z^2}{2+2z^2}  \Delta \alpha_k =(-1)^{k+1} (\tfrac12 -p)  \Delta \alpha_k ,
\]
where $\Delta  \alpha_k = 2\alpha_k - \alpha_{k+1} - \alpha_{k-1} $ is the Laplacian of simple random walk.
Then $q_k=q^\infty_{u,v}$ whenever $|u-v|=k$.
Using \eqref{eq:PKformula}, we have
\begin{align} \label{eq:expdec}
(-1)^{k+1}  \Delta \alpha_k = \begin{cases} -B  |\gamma|^{|k|} (2-\gamma-\gamma^{-1}) \geq 0& \text{for } k\neq 0, \\
						   \frac2{1+6p}-2B(1-\gamma) \geq 0& \text{for } k =0.
 \end{cases}
\end{align}
and hence the total transition weight is
\begin{align}
\sum_{k\in \mathbb Z} q_{|k|}&= (\tfrac12 -p)\Big(-2B (2-\gamma-\gamma^{-1}) \frac{-\gamma}{1+\gamma}-2B(1-\gamma) +\frac2{1+6p} \Big) \nonumber\\
\label{l2}
&= (\tfrac12 -p)\Big( -4B\frac{1-\gamma}{1+\gamma}+\frac2{1+6p}\Big).
\end{align}
Using \eqref{gamma} and \eqref{B}, it can be checked that the last expression is equal to one for all $0<p<1/2$ (equivalently all $z>0$).
Exponential decay of $q_k$ is clear from~\eqref{eq:expdec}.
\end{proof}

\subsection{Random walk representation of $D^{-1}$}
\label{S:inverseKfinite}

Here we finally establish a rigorous version of \eqref{eq:Greenformal} using the ingredients from the previous sections.
Recall that $K_N$ is the Kasteleyn matrix of the graph $\cG^N$ and $D_N = (K_N)^* K_N $. We will be mostly interested in the restriction of $D_N^{-1}$ to the vertices of $\cG$.
Observe that $D_N$ can be written as a block-diagonal matrix if we consider vertices respectively in the odd or even rows. Hence to invert $D_N$ it will suffice to invert each of these blocks separately.
We call $D_N^{\textnormal{odd}}$ (resp. $D_N^{\textnormal{even}}$) the matrix $D_N$ {restricted to $\Vo(\cG^N) \cup V_{-1}(\cG^N)$} (resp. $\Ve(\cG^N)$).

We first focus on the odd case (the even case is much easier as explained before), and for now we will write $D_N$ for $D_N^{\textnormal{odd}}$. The key idea will be to use the \textbf{Schur complement formula}.
To be more precise, we observe that $D_N$ has the block structure
$$
D_N = \left(
\begin{array}{cc}
   A & B\\
   B^T &C
\end{array}
\right),
$$
where $A$ is indexed by the special row $V_{-1}$, and $C$ is indexed by all the other rows $\Vo $. Hence $B$ and $B^T$ can be thought of as a ``transition matrices'' between $V_{-1}$ and $V_{\textnormal{odd}}$. Note that these matrices depend on $N$ but we don't write this explicitly to lighten the notation.
We define the Schur complement of $A$ to be the matrix
\begin{equation}\label{Schurdef}
D_N/A : =  C - B^T A^{-1} B.
\end{equation}
With this definition, the restriction of $D_N^{-1}$ to $\Vo$ is simply given by
\begin{equation}\label{Schurcomplinverse}
D_N^{-1} |_{\Vo} = (D_N/A)^{-1}.
\end{equation}

We now outline how we proceed.
\begin{itemize}
  \item We first write $A^{-1}$ in terms of the Green's function for the random walk on $V_{-1}(\cG^N)$ with transition probabilities as in \eqref{eq:pxy}.

  \item This gives us a formula for the Schur complement $D_N/A$ via \eqref{Schurdef}. We then use that for $N$ sufficiently large, this Schur complement can be viewed as a (genuine) Laplacian for a random walk. The proof of this statement is postponed until Section~\ref{S:1D}.

  \item As a consequence of \eqref{Schurcomplinverse}, this gives a formula for the inverse of $D_N$ as a Green's function of a genuine random walk.
  \item Finally, as the number $N$ of triangles appended to $\cG^0$ tends to infinity, on the one hand, the above analysis shows that the inverse Kasteleyn matrix (restricted to $\Vo$) can be written in terms of the Green's function of a random walk with jumps along the boundary. On the other hand as mentioned before, the free boundary dimer model becomes equivalent to the same model on $\cG^0$ with modified monomer weights  $z' $ as in \eqref{zprime} at the monomer-corners.

  \item The results of this section are summarised below as Corollary \ref{C:finitevol}.
\end{itemize}

We start with the computation of $A^{-1}$. To this end let
  $$
  g^N(u,v) = \mathop{\sum_{\gamma: u \to v}}_{ \gamma \subseteq V_{-1}(\cG^N)} \prod_{e= (x,y) \in \gamma}p^N_{x,y}
  $$
  be the Green's function of the random walk on $V_{-1}(\cG^N)$ with transition probabilities $p^N_{x,y}$ defined for $\cG^N$ as in \eqref{eq:pxy}.
  Note that this is well defined since the walk is killed on both the left and right extremities of $V_{-1}(\cG^N)$ (see Figure~\ref{fig:GN} for the exact form of transition probabilities at the extremities $x_1,x_2$).

\begin{lemma}
  \label{L:Ainv} Let $u,v \in V_{-1} = V_{-1} (\cG^N)$. Then
 \[
  A^{-1} (u,v) =\frac{1}{A(v,v)} (-1)^{\Re(u-v)} g^N({u,v}).
  \]
\end{lemma}
\begin{proof}
This follows from the fact that $|A|$ is the Laplacian for the random walk described above, and moreover (as mentioned before) the sign of the transition weights induced by $A$
is negative if the step is of size $\pm 1$ and positive otherwise (step size $\pm2$). This follows from the definition of $D_N$ and the Kasteleyn matrix.
\end{proof}

We now explain how this yields an interpretation for the Schur complement $D_N/A$ as a (genuine) Laplacian for a random walk in the bulk $V_{\textnormal{odd}}(\cG)$ with jumps along the boundary
$V_1(\cG)$.
For $u, v \in V_1 = V_1(\cG^N) = V_1(\cG)$, we define
\begin{align} \label{eq:qN}
q^N_{u,v} = (B^T A^{-1} B) (u,v).
\end{align}
Recalling that $D_N(v,v) = A(v,v) = 2 + 2z^2$ for $v \in V_{-1}(\cG^0)$ and $N\geq 1$, a straightforward computation using Lemma \ref{L:Ainv} shows that
\begin{equation}\label{qNjumps}
q^N_{u,v} = \frac {z^2} {2+2z^2}(-1)^{\Re(v-u)}  \left ( (g^N(u_+, v_+) - g^N(u_+, v_-)) - (g^N(u_-, v_+) - g^N(u_-, v_-) ) \right),
\end{equation}
where again $u_\pm, v_\pm$ are the left and right vertices in $V_{-1}$ at distance two from $u$ and $v$ respectively.

Recall the definition of $q^\infty_{u,v}$ from \eqref{qinfty}, and let $q^N_{u,v}$ be the transition weights defined by~\eqref{q} for the graph $\cG^N$.
The next results, whose proof will be given in Proposition~\ref{P:GreenPK1d} of the next section, implies that for $N$ large enough, $q^N_{u,v}$ become actual transition probabilities.

\begin{lemma}
  \label{L:locallimitqN}
    Let $u, v \in V_1= V_1(\cG^N) = V_1(\cG^0)$. Then $q^N_{u,v}\to q^{\infty}_{u,v}$ as $N \to \infty$ pointwise. In particular, for $N$ sufficiently large,
    \begin{align} \label{eq:largeN}
    q^N_{u,v} > 0 \qquad \text{ and } \qquad \sum_{v \in V_1} q^N_{u,v} < 1.
    \end{align}
\end{lemma}
\begin{proof}
The convergence follows immediately from \eqref{qNjumps} and the convergence result in Proposition~\ref{P:GreenPK1d}.
Condition \eqref{eq:largeN} is a consequence of Lemma~\ref{T:eff}.
\end{proof}
Note that the second inequality is strict since the sum is taken over  $V_1(\cG)\subsetneq V_1(\Z^2 \cap \H)$.
 Now let $N$ be sufficiently large that~\eqref{eq:largeN} holds true, and consider a transition matrix between vertices in $u,v\in V_\textnormal{odd}$ given by
 \begin{equation}\label{BulkRW}
 R_N(u,v) =   I(u,v) - \frac{1 }{ C(u,u)}\left (C(u,v)-q^N_{u,v} \mathbf 1_{\{u,v \in V_1\}} \right),
 \end{equation}
 where $I$ is the identity.
 Note that
 \[
 R_N(u,v) \ge 0 \qquad \text{ and } \qquad  \sum_{v} R_N(u,v) \le 1
 \]
 so that $R_N$ is a substochastic matrix. Indeed, this follows from the definition of $C=D_N|_{V_{\textnormal{odd}}}$ and \eqref{eq:largeN}.
 Also note that this holds even when $u$ is one of the two corners, i.e., the left and right extremities of $V_1(\cG)$.
In other words, we may add a cemetery absorbing point $\partial $ to the state space and declare $R_N(x, \partial) = 1- \sum_y R_N(x, y) \ge 0$. This turns $R_N$ into {the transition matrix of} a proper random walk on the augmented state space $\Vo \cup \{\partial\}$, which is absorbed at $\partial$. We let $Z^N$ be the random walk on $\Vo \cup \{\partial\}$ whose transition probabilities are given by $R_N(x,y)$. We call this random walk the \textbf{effective (odd) bulk random walk}.

The interest of introducing the transition matrix $R_N$ of this effective bulk random walk is that its associated Laplacian gives us the Schur complement $D_N/A$: that is, for $u,v \in \Vo$, we have
\begin{equation}\label{E:effbulk_Schur}
(D_N/A) (u,v)  = C(u,u)( I(u,v) - R_N(u,v)),
\end{equation}
which follows from the definition of the Schur complement~\eqref{Schurdef}, \eqref{eq:qN} and the definition of $R_N$.

From this formula and the Schur complement formula \eqref{Schurcomplinverse}, it is immediate to deduce the following proposition, which says that the inverse of $D_N^{\textnormal{odd}}=D_N$ (i.e., the inverse of $(K_N)^*K_N$ restricted to bulk odd vertices) is given by the Green's function of the effective bulk random walk.
Recall that $C(v,v)=D_N(v,v)$.
  \begin{prop}
Let $u,v \in \Vo(\cG) $. Then for all $N$ sufficiently, large we have
 $$
(D_N^{\textnormal{odd}})^{-1} (u,v) = G_{\textnormal{odd}}^N(u,v),
 $$
 where $G_{\textnormal{odd}}^N$ is the (normalised) Green's function associated to $R_N$, i.e.,
 \begin{equation}\label{godd}
 G^N_{\textnormal{odd}}(u,v) = \frac1{D_N(v,v)} \E_u\Big( \sum_{t=0}^\infty \mathbf 1_{\{Z^N_t = v\}} \Big).
 \end{equation}
\end{prop}

We now address the even case, and write $D_N=D_N^{\textnormal{even}}$. We introduce a ``sign'' diagonal matrix $S(x,x) = (-1)^{\Re(x)}$.
Then, the matrix
$$
\tilde D_N: = S^{-1} D_N S
$$
is positive on the diagonal and negative off-diagonal. Moreover, we have
$$
\tilde D_N^{-1}(u,v) =  G_{\textnormal{even}}^N(u,v)
$$
where
$$
G_{\textnormal{even}}^N(u,v) = \frac{1}{D_N(v,v)} \E_u \Big( \sum_{t=0}^\infty \mathbf1_{\{ \tilde Z_{t} = v\}} \Big),
$$
where $\tilde Z$ is a random walk on $V_{\textnormal{even}}(\cG^N)$ with the transition probabilities:
\begin{equation}\label{Rev}
\tilde R^{N}(x,y) = \frac{| D_N(x,y)|}{D_N(x,x)} \mathbf 1_{x \neq y}.
\end{equation}
The fact {that} the even case is much simpler than the odd one can be seen here since $\tilde R^{N}(x,y)$ is actually a transition matrix of a true random walk on $V_{\textnormal{even}}(\cG^N)$.
Indeed, (see Figure~\ref{fig:GN} for an illustration)
\begin{itemize}
\item in the bulk of $V_{\textnormal{even}}(\cG^N) \setminus V_0(\cG^N) $, the walk jumps by $\pm 2$ in each direction with probability $1/4$ each,
\item On the boundary $\partial G \cap V_{\textnormal{even}}(\cG^N) $, the walk makes jumps according to the local boundary conditions which are either Dirichlet or Neumann,
\item On $V_0(\cG^N)\cap V_0(\cG)$ it may jump horizontally by $\pm 1$ with probability $z^2/(3 + 2z^2)$ or by $\pm 2$ with probability $1/(3 + 2z^2)$, and vertically by $+2$ also with probability $1/(3 + 2z^2)$.
This is consistent with the fact that $D(x,x) = 3 + 2z^2$ for $x \in V_0(\cG)$,
\item On $V_0(\cG^N)\setminus V_0(\cG)$ except at its endpoints, it may jump horizontally by $\pm 1$ with probability $z^2/(2 + 2z^2)$ or by $\pm 2$ with probability $1/(2 + 2z^2)$.
This is consistent with the fact that $D(x,x) = 2 + 2z^2$ for $x \in V_0(\cG^N)\setminus V_0(\cG)$,
\item At the the endpoints of $V_0(\cG^N)$, it has transition probabilities as the vertices $y_1,y_2$ in Figure~\ref{fig:GN}.
\end{itemize}

All in all we obtain that
\begin{equation}\label{gev}
 D_N^{-1}(u,v) = (-1)^{\Re(v-u)}G_{\textnormal{even}}^N(v,u).
\end{equation}
Now a moment of thought shows that there is no problem in letting $N \to \infty$ in this expression. This is because the random walk associated with $R_N$ is absorbed on some portion of the boundary $\partial \cG \setminus \partial_{\textnormal{free}} \cG$, as described in Section~\ref{S:boundary}.

Hence we deduce that
\begin{equation}
\lim_{N\to \infty}  D_N^{-1}(u,v) = (-1)^{\Re(u-v)} G_{\textnormal{even}} (u,v).
\end{equation}
where $G_{\textnormal{even}} (u,v)$ is the Green's function on $\cG^\infty$ (that is, the graph $\cG^0$ to which infinitely many triangles have been added on either side of $V_0$) associated with the random walk on $\cG^\infty$ whose transition probabilities are given by \eqref{Rev}.

At the same time, when $N \to \infty$, the free boundary dimer model on $\cG^N$, restricted to $\cG^0$, becomes equivalent to a free boundary dimer model on $\cG^0$ where the monomer weights on the extreme vertices (corners) of $V_0$ have been given the weight $z'>0$ as in~\eqref{zprime}.

We now summarise the results obtained in this section.

\begin{corollary}\label{C:finitevol}
  Consider the free boundary dimer model on $\cG^0$ where the monomer weight $z>0$ on $V_0(\cG)$ except at its monomer-corners where the monomer weight is $z' $ as in~\eqref{zprime}. Let $K$ be the associated Kasteleyn matrix, and $D = K^* K$. Then for all $u,v\in V(\cG) $, we have
  $$
  D^{-1}(u,v) =
  \begin{cases}
      G_{\textnormal{odd}} (u,v) & \text{ if } u,v\in \Vo (\cG) ,   \\
(-1)^{\Re (v-u)}G_{\textnormal{even}} (u,v)     & \text{ if } u,v\in \Ve(\cG),  \\
    0 & \text{ otherwise.}
  \end{cases}
  $$
  where $G_{\textnormal{odd}}, G_{\textnormal{even}}$ are the normalised Green's functions associated with the effective (odd and even) bulk random walks described in \eqref{BulkRW} and  \eqref{Rev} respectively, normalised by $D(v,v)$.

In particular, the inverse Kasteleyn matrix is given by $K^{-1} = D^{-1} K^*$.
\end{corollary}

This result implies Theorem~\ref{T:finitevol_intro} with the walks $Z_\textnormal{even}$ and $Z_\textnormal{odd}$ explicitly defined as above.

%
%


%

\subsection{Convergence to potential kernel of the auxiliary walk} \label{S:1D}
In this section we prove the convergence statement from Lemma~\ref{L:locallimitqN}.

To this end, let $N\ge 1$ and consider a random walk $(\tilde X_n, n \ge 0)$ on $[-N-1, \ldots, N+1] $ (we note that the role of $N$ is slightly different here compared to the definition of $q^N_{u,v}$, as for simplicity of notation we do not account for the length of $V_{0}(\cG)$) where the transition probabilities $\tilde p_{u,v}$ in $[-N+1, N-1]$ coincide with those of the random walk $X$ from \eqref{D:rw}. At $\pm N$ and $\pm (N+1)$ the walk has the following boundary conditions (see the vertices $x_1,x_2$ in Figure~\ref{fig:GN} ):
\begin{itemize}
\item
the chain is absorbed at $\pm (N+1)$
\item
at $\pm N$ the transitions are those of $X$ but reflected (e.g., at $u = N$, the only possible transitions are to $v = N-1$ and $v = N-2$ with weights given by twice those in \eqref{D:rw}). Note that the boundary conditions are completely symmetric, so that $\tilde p_{u, v}= \tilde p_{|u|, |v|}$ if $\sgn(u) \sgn(v) =1$.
\end{itemize}
Let
$$\tilde g_N(u, v) = \E_u (\tilde L_{T_\partial}( v)),$$
where $\tilde L_t (v) = \sum_{s=1}^t \mathbf 1_{\{\tilde X_s = v\}}$ is the local time of $\tilde X$ at $v$ by time $t$, and $T_\partial$ is the killing time of $\tilde X$ (first hitting time of $\pm (N+1)$). We will check here that Green's function differences converge to the potential kernel, in the following sense:

\begin{prop}\label{P:GreenPK1d}
As $N \to \infty$,
\begin{equation}\label{Conv1D}
\tilde g_N(u,v) - \tilde g_N(u', v) \to  -( \alpha(u, v) - \alpha(u',v))
\end{equation}
where $\alpha(u, v) $ is the potential kernel from \eqref{D:pk} (that is, $\alpha(u,v) = \alpha_{|u- v|}$).
\end{prop}

\begin{proof}
To begin, recall that if $(X_n, n \ge 0)$ is the walk on $\Z$ with transitions given by $p_{u,v}$ in \eqref{D:rw}, and if $v$ is fixed, then $\alpha (x, v)$ is harmonic (for $X$) in $x$ except at $x =v$. More precisely, $M_n = \alpha (X_n, v) - L_n(v)$ is a martingale.

Suppose first that $v = 0$, and consider the walk $\tilde X$ instead of $X$. We claim that by symmetry,
\begin{align} \label{eq:MN}
\tilde M_n = \alpha (\tilde X_n, 0) - \tilde L_n(0) + \tilde A_n; \ \ 0 \le n \le T_\partial
\end{align}
is a martingale, where for some constant $c \in \R$,
$$
\tilde A_n = c (\tilde L_n (-N) + \tilde L_n(N)) = c L^{|\tilde X|}_n (N)
$$
is the local time of $\tilde X$ at the reflecting part of the boundary, or equivalently the local time of the absolute value $|\tilde X|$ at $N$ by time $n$; note that by assumption $|\tilde X|$ is itself a Markov chain. To be more precise, $c$ can be computed as
$$c = \E_{N}(\alpha (\tilde X_1, 0)) - \alpha (N, 0).$$
Applying to \eqref{eq:MN} the optional stopping theorem at $T_\partial$ (which is allowed since this only involves a finite number of possible values for $\tilde X_n$), and since $\alpha ( -N-1, 0) = \alpha (N+1, 0) = \alpha_{N+1}$ by symmetry of $\alpha$,
we get
$$
\tilde \E_u(\tilde M_{T_\partial}) = \alpha_{N+1} - \tilde g_N(u, 0) + \tilde \E_{u} ( \tilde A_{T_\partial}).
$$
On the other hand, starting from $u$, $\tilde M_0 = \alpha (u, 0)$, hence
\begin{equation}
\label{onepoint}
\alpha (u, 0) = \alpha_{N+1}- \tilde g_N(u, 0) + c\tilde \E_{|u|} ( L^{|\tilde X|}_{T_\partial} (N)).
\end{equation}
Therefore, applying \eqref{onepoint} also at a different vertex $u'$ and taking the difference,
we get
\begin{equation}\label{twopoints}
\alpha (u, 0) - \alpha (u', 0) = -(\tilde g_N(u, 0) - \tilde g_N (u', 0)) + c \big[ \tilde \E_{|u|} ( L^{|\tilde X|}_{T_\partial} (N)) - \tilde \E_{|u'|} ( L^{|\tilde X|}_{T_\partial} (N))\big].
\end{equation}

We now aim to take $N \to \infty$, and show that the last term of \eqref{twopoints} vanishes, which would imply the result of Proposition \ref{P:GreenPK1d} in the case $v = 0$.
Crucially, by the Markov property of $|\tilde X|$, the last term on the right hand side of \eqref{onepoint} depends only on $u$ in so far as the probability to reach $N$ before $T_\partial$ depends on $u$. That is,
$$
\tilde \E_{|u|} ( L^{|\tilde X|}_{T_\partial} (N)) = \tilde \P_u ( T_N < T_\partial) \E_{N} ( L^{|\tilde X|}_{T_\partial} (N)).
$$
Furthermore, note that
\begin{itemize}
  \item Each time the walk is at $\pm N$, there is a fixed positive chance that the walk will be absorbed before returning to the boundary (and a vanishing chance that it will reach the other end of the boundary), hence $ \E_{N} ( L^{|\tilde X|}_{T_\partial} (N))$ converges to a fixed limit as $N \to \infty$.

  \item As $N \to \infty$, and $u, u'$ are fixed, then $\lim_{N \to \infty}\tilde \P_u ( T_N < T_\partial)$
      exists and does not depend on $u$. This can be seen e.g. from renewal theory.
\end{itemize}
Taken together, these two points imply that the limit as $N \to \infty$ of $\tilde \E_{|u|} ( L^{|\tilde X|}_{T_\partial} (N))$ exists and does not depend on $u$. This concludes the case $v = 0$.

In the general case where $v$ is arbitrary and fixed, we still get a martingale
$$
\tilde M_n = \alpha (\tilde X_n, v) - \tilde L_n(v) + \tilde A'_n, \ \quad 0 \le n \le T_\partial
$$
but the form of the error $\tilde A'_n$ needs to be slightly adjusted compared to $\tilde A_n$, since the values
\begin{equation}\label{c}
 \E_{x}(\alpha (\tilde X_1, v)) - \alpha (x, v)
\end{equation}
are no longer the same for $x = N$ and $x= -N$. To deal with this and later arguments, we remark that we can replace $v$ by $0$ in the following manner:
\begin{lemma}\label{L:sign}
  As $|x| \to \infty$, then
  $$\alpha (x, v) = \alpha (x, 0) - \sgn(x)\frac{|v|}{\sigma^2} + o(1).$$
\end{lemma}
\begin{proof}
  This is straightforward from \eqref{eq:PKformula}.
\end{proof}
In particular, using Lemma \ref{L:sign}, the limits of \eqref{c} for $x = -N$ and $x = N$ exist and coincide with $c$. Consequently, we deduce that the error term in the martingale $\tilde M_n$ has the form
\begin{equation}\label{AvsAp}\tilde A'_n = (c+ o(1)) (\tilde L_n (-N) + \tilde L_n(N)) = (c+ o(1)) L^{|\tilde X|}_n (N) = (1+ o(1)) \tilde A_n,
\end{equation}
where the $o(1)$ term tends to 0 as $N \to \infty$ (but is not random and does not depend on $n$).
Applying the optional stopping theorem at $T_\partial$ to the martingale $\tilde M_n$, we deduce (using \eqref{AvsAp} and Lemma~\ref{L:sign} one more time) that
\begin{align}
\alpha (u, v) & =
\E_u ( \alpha (\tilde X_{T_\partial}, v)) - \tilde g_N(u, v) + \tilde \E_{u} (\tilde A'_{T_\partial}) \nonumber \\
& = \E_u ( \alpha (\tilde X_{T_\partial}, 0) - \sgn(\tilde X_{T_\partial}) \frac{|v|}{\sigma^2} +o(1)) - \tilde g_N(u,v) + (1+ o(1)) \E_u(\tilde A_{T_\partial}) \nonumber\\
& = \alpha_{N+1} - \frac{|v|}{\sigma^2} \E_u ( \sgn(\tilde X_{T_\partial})) +o(1)  - \tilde g_N(u,v) + (1+ o(1)) \E_u(\tilde A_{T_\partial}). \label{onepoint_asym}
\end{align}
Now, it is clear that as $N \to \infty$,
$$
\E_u (  \sgn(\tilde X_{T_\partial})) \to 0
$$
since by recurrence of $X$ there is probability tending to one to hit zero before $T_\partial$, after which the sign is equally to be positive or negative by symmetry. We deduce from \eqref{onepoint_asym} that as $N \to \infty$,
\begin{equation}\label{onepoint_asym2}
  \alpha(u, v) = \alpha_{N+1} + o(1) - \tilde g_N(u,v) + (1+ o(1)) \E_u(\tilde A_{T_\partial}).
\end{equation}
Since we have already verified that the limit of $\E_u(\tilde A_{T_\partial})$ as $N \to \infty$ exists and does not depend on $u$, we conclude the proof of Proposition \ref{P:GreenPK1d} by taking the difference in \eqref{onepoint_asym2} for $u$ and $u'$ and letting $N \to \infty$.
\end{proof}

\section{Infinite volume limit} \label{S:IVL}

In the previous section we showed that $D_N^{-1}$ (and hence $K_N^{-1}$) has a limit as $N \to \infty$ which is given in terms of two Green's functions $\Godd$ and $\Gev$ associated to random walks on $V_{\mathrm{odd}}(\cG)$ and $V_{\mathrm{even}}(\cG)$ which may jump along $V_1(\cG)$ and $V_0(\cG)$, and with various boundary conditions (Dirichlet or mixed Neumann--Dirichlet) on $\partial \cG \setminus (V_0(\cG) \cup V_1(\cG))$. Let us also denote these Green's functions by $\Godd^{\cG}$ and $\Gev^{\cG}$ to emphasize their dependence on $\cG$.

The purpose of this section is to take an infinite volume limit as $\cG$ tends to the upper half-plane. 
In this limit the Green's functions $\Godd^{\cG}$ and $\Gev^{\cG}$ diverge (corresponding to the fact that the limiting bulk effective random walk is recurrent). However, we can still make sense of its potential kernel. Hence the inverse Kasteleyn matrix, which is obtained as a derivative of these Green's functions, has a well defined pointwise limit.

The argument for this convergence as $\cG$ increases to the upper half plane are essentially the same for both the odd and even walks. As will be clear from the proof below, the arguments rely only on the fact that (a) the two walks coincide with the usual simple random walk (with jumps of size 2) away from the real line, (b) they are reflected on the real line with some jump probabilities that decay exponentially fast with the jump size (in fact, in the even case the jumps are bounded), and (c) they can `switch colour' with positive probability along the real line. This terminology will be explained below. For these reasons, and in order to avoid unnecessarily cumbersome notation, we focus in this section solely on the \emph{odd walk} (the argument works literally in the same way for the even case, and can in fact be made a little easier).

\subsection{Construction of the potential kernel in the infinite volume setting}

We write $\Gamma$ for the weighted graph corresponding to {the odd effective random walk}. Thus, the vertex set $V$ of $\Gamma$ can be identified (after translation so that $V_1 \subset \R$) with $(\Z \times 2\Z) \cap \H$ and its edges $E$ are those of $(2\Z)^2$, plus those of $(2\Z+1) \times (2\Z)$, plus additional edges connecting these two lattices along the real lines. In reality, it will be easier to consider a symmetrised version of $\Gamma$ obtained by taking the vertex set to be $V \cup \bar V$ and the edges to be $E \cup \bar E$, where $\bar V$ and $\bar E$ are the complex conjugates of $V$ and $E$. We will still denote this graph by $\Gamma$. Throughout this and the next section the random walks we will consider will take values in this symmetrised graph. Note that $\Gamma$ is not locally finite: any vertex on the real line has infinite degree, but the total weight out of every vertex is finite (and is equal to 1). We recall that when away from the real line, the random walk on $\Gamma$ looks like simple random walk on the square lattice \emph{up to factor 2}: the transitions from a point $x\in \Z^2$ away from $\R$ are to the four points $x \pm 2e_1$ or $x\pm 2e_2$, where $(e_1, e_2)$ is the standard basis of $\Z^2$. On the real line, the effective random walk can make jumps of any size, but the jumps are symmetric and the transition probabilities have an exponential tail.
Note that the odd effective random walk only jumps between vertices of the same colour in the bulk, and can possibly change colour only on the real line. In the current section, we will also use the word \textbf{class} to denote the notion of colour.
Finally, we say that two vertices in $\Gamma$ have the same \textbf{parity} (or \textbf{periodicity}) if the differences of their vertical and horizontal coordinates are multiples of $4$.

Our first goal will be to show that differences of Green's functions evaluated at two different vertices of the same class for the walk killed when leaving a large box, converge (when the box tends to infinity) to differences of the \textbf{potential kernel} of the walk on the infinite graph $\Gamma$. Our first task will be to define this potential kernel. For the usual simple random walk on $\Z^2$ this is an easy task because the asymptotics of the transition probabilities are known with great precision. In turn this is because simple random walk can be written as a sum of i.i.d.\ random variables making it possible to use tools from Fourier analysis: see Chapter 4 of \cite{LawLim} for a thorough introduction. The walk on $\Gamma$ obviously does not have this structure, and in fact it seems that there are few general tools for the construction of the potential kernel for walks on a planar graph beyond the i.i.d.\ case. The coupling arguments we introduce below may therefore be of independent interest.

Let $P$ denote the transition matrix of simple random walk on $\Gamma$, and let $\tilde P = (I+P)/2$ be that of the associated lazy chain. The rationale for considering this version is that, on the one hand, it gets rid of periodicity issues, while on the other hand, it only modifies the Green's function by a constant factor: e.g., on a transient graph, $\tilde G(x,y) = 2 G(x,y)$ for any $x,y$, if $G$ and $\tilde G$ are the corresponding Green's functions (this is because the jump chains are the same, and the lazy chain stays on average twice as long at any vertex as the non-lazy chain).

The basic idea for the definition of the potential kernel will be the following.
Let $X$ and $X'$ denote (lazy) random walks started respectively from two vertices $x$ and $x'$ of the same class, and suppose that they are coupled in a certain way so that after a random time $T$ (which may be infinite), $X$ and $X'$ remain equal forever on the event that $T< \infty$: that is,
\begin{equation}\label{eq:couplingtime}
X_{T+s } = X'_{T+s}, \qquad s \ge 0.
\end{equation}
We will define a coupling (its precise definition will be given below) that depends on a time-parameter $t$ such that for this particular value of $t$,
\begin{equation}\label{eq:tailcoupling}
  \P( T > t) \lesssim (\log t)^a t^{-1/2}
\end{equation}
for some $a>0$ whose value will not be relevant. (Note that this inequality should not be understood as saying something about the tail of $T$, since $T$ depends on $t$; indeed $T$ might be infinite with positive probability). In fact, a much weaker control of the form $\P( T > t) \lesssim t^{-\eps}$, for some $\eps>0$, would be sufficient for the definition of the potential kernel alone, as will be apparent from the argument below. We however insist on \eqref{eq:tailcoupling} in order to get good a priori bound on the potential kernel (see Proposition \ref{P:aprioriPK}). As we will see, the goal of this coupling will be to compare $\tilde p_t(x, o)$ to $\tilde p_t(x',o)$ which is why $T$ is allowed to depend on $t$, and why we only require $T$ to be less than $t$ with high probability (but we do not care what happens on the event $\{T> t\}$).
Here and later on, $o$ denotes an arbitrary fixed vertex.

We first argue that we can get a good a priori control on the transition probabilities $\tilde p_t(x, o)$. Let $A \subset \Z \times 2 \Z$ be a finite set. By ignoring the long range edges which may leave $A$ through the real line, and using the standard discrete isoperimetric inequality on $\Z^2$ (Loomis-Whitney inequality, Theorem 6.22 in \cite{LyonsPeres}) it is clear that
$$
\sum_{x \in A, y \in A^c} w_{x,y} \gtrsim |A|^{1/2}
$$
where $w_{x,y}$ is the weight of the edge $(x,y)$ in $\Gamma$. This means that $\Gamma$ satisfies the two-dimensional isoperimetric inequality $(I_2)$ (we here use the notation of \cite{Barlow}). Consequently, by Theorem 3.7, Lemma 3.9 and Theorem 3.14 of \cite{Barlow}, $\Gamma$ satisfies the two-dimensional \textbf{Nash inequality}, $(N_2)$. Therefore, if $q_s^x(\cdot)$ denote the transition probabilities of the continuous time walk on $\Gamma$, normalised by its invariant measure, we have by Theorem 4.3 in \cite{Barlow} that
$$
q^x_s(x) \lesssim 1/s
$$
and since $q_s^x$ is maximised on the diagonal, we deduce that
\begin{equation}\label{aprioriHK}
\tilde p_s(x,o) \lesssim 1/s,
\end{equation}
where the implied constant is uniform in $x,o$ and $s \ge 1$.

Now suppose we have a coupling satisfying \eqref{eq:couplingtime} and \eqref{eq:tailcoupling}. We will explain why this implies that
\begin{equation}\label{eq:PKconvergence}
  \sum_{t=0}^\infty (\tilde p_t(x,o) - \tilde p_t(x',o))
\end{equation}
converges. We couple the walks starting from $x, x'$ according to \eqref{eq:couplingtime}. Obviously, on the event $\{T \le t/2\}$, $X_t = o$ if and only if $X'_t = o$, and thus
\begin{align}
  | \tilde p_t(x, o) - \tilde p_t(x',o) | & \le 2\P (T \ge t /2)\max_{y}\tilde p_{t/2}(y, o)\nonumber \\
  & \lesssim t^{-3/2} (\log t)^a \label{HKdiff}
\end{align}
which is summable, whence the series \eqref{eq:PKconvergence} converges.

\begin{defn}
  \label{D:potkerninf}
  We set
  \[
  \tilde a(x,o) - \tilde a(x', o)  = - \sum_{t=0}^\infty ( \tilde p_t(x,o) - \tilde p_t(x',o)).
  \]
  By convention we define $\tilde a(o,o) = 0$ and so this recipe may be used to define $\tilde a(x,o)$ provided that $x$ and $o$ are of the same class (by summing increments along a given path from $x$ to $o$).
   (As the choice of a path from $x $ to $o$ does not matter before the limit in the series is taken, this is well defined.) Since $x$ and $o$ are arbitrary vertices of the same class, this defines $\tilde a(\cdot, \cdot)$ everywhere on this class.\footnote{The arguments in this section rely on thinking of $\tilde a (\cdot, \cdot)$ as a function of the first variable while the second is frozen, which is why we prefer to use $x$ for the first variable and $o$ for the second. In the next section, both variables will start playing a more symmetric role and we will switch to $x$ and $y$.}

   If also \eqref{eq:tailcoupling}
   holds for \emph{one} pair $x, x'$ not of the same class, then this defines $\tilde a(\cdot ,\cdot)$ over the entire graph.
\end{defn}

Note also that due to the fact that $\pi(x) =1$ is a \emph{constant} reversible measure on $\Gamma$ (hence $\tilde p_k(x,y) = \tilde p_k(y,x)$), the potential kernel is symmetric: $\tilde a(x,y) = \tilde a(y,x)$ for any $x,y$. We will not however need this property in the following.

In the next subsection we describe a concrete coupling which will be used for the construction of the potential kernel. We call this the \textbf{coordinatewise mirror coupling}, which is a variation on a classical coupling for Brownian motion in $\R^d$. We will then use this coupling again to obtain \emph{a priori} estimates on the potential kernel.

Before describing this coupling and justifying \eqref{eq:tailcoupling}, we first state and prove a lemma which will be useful in many places in the the following and which gives a subdiffusive estimate on the walk. Let $\text{dist}$ denote the usual $\ell^1$ distance (graph distance) on $\Z^2$.

\begin{lemma}
  \label{L:subdiffusive} Let $x $ be a vertex of $\Gamma$ and let $T_R = \inf \{ n\ge 0: \textnormal{dist}(X_n, x) \ge R\}$. Then for every $c_1>0$ there exists $c_2>0$ such that for any $n\ge 1$, and for any $R \ge c_1 \sqrt{n} \log n$,
  $$
  \P(T_R \le n) \lesssim \exp ( - c_2 (\log n)^2).
  $$
\end{lemma}

\begin{proof}
One possibility would be to use a result of Folz \cite{Folz} (based on work of Grigor'yan \cite{Grigoryan} in the continuum) which shows that an on-diagonal bound on the heat kernel $p_t(x,x)$ and $p_t(y,y)$ implies a Gaussian upper bound on the off-diagonal term $p_t (x,y)$. However, it is more elementary to use the following martingale argument.
We may write $X_n = (u_n, v_n)$ in
coordinate form. Since $(v_n)$ is a lazy simple random walk on the integers, the proof is elementary in this case (and of course also follows from the more
complicated estimate below). We therefore concentrate on bounding  $\sum_{i=1}^n\P (|u_i| \ge R)$.
We bound $\P(|u_i| \ge R)$ for $1\le i \le n$ as follows: either there is one jump larger than say $K=(\log n)^2$ by time $n$ (this has probability at most $n \exp ( - c(\log  n)^2)$ by a union bound and exponential tail of the jumps) or if all the jumps are less than $K$, then $u$ coincides with a martingale $\bar u$ such that all its jumps are bounded by $K$ in absolute value: indeed, we simply replace every jump of $u$ greater than $K$ in absolute value by a jump of the same sign and of length $K$. Since the jump distribution \eqref{q} is symmetric, the resulting sum $\bar u_n$ is again a martingale. Furthermore, $\bar u_n$ is a martingale with bounded jumps. We may apply Freedman's inequality \cite[Proposition (2.1)]{Freedman} to it which implies (since the quadratic variation of $\bar u$ at time $1\le i\le n$ is bounded by $b \lesssim n$),
  \begin{equation}\label{mart_displ}
  \P( |\bar u_i| \gtrsim \sqrt{n} \log n) \lesssim \exp \left( - c \frac{n (\log n)^2}{ (\log n)^2 \sqrt{n} \log n + n} \right) \lesssim \exp ( - c (\log n)^2).
  \end{equation}
The result follows by summing over $1\le i \le n$.
\end{proof}

\subsection{Coordinatewise mirror coupling}

Let $x, x'$ be two vertices of the graph $\Gamma$ of the same class, and let $\tilde X,\tilde X'$ be two (lazy) effective random walks started from $x$ and $x'$ respectively. In the coupling we will describe below, it will be important to first fix the vertical coordinate (stage 1). The coupling ends when we also fix the horizontal coordinate (stage 4). In between, we have two short stages (possibly instantaneous), where we make sure the class is correct (stage 2) followed by a so-called ``burn-in'' phase where the walks get far away from the real line in parallel (stage 3). This depends on a parameter $r$, which is a free choice. (When we prove \eqref{eq:tailcoupling} we will choose $r$ to be slightly smaller by logarithmic factors than $\sqrt{t}$).

We need to do so while respecting the natural \textbf{parity} (i.e., periodicity) of the coordinates we are trying to match. We will use the laziness to our advantage in order to deal with the potential issues arising from the walks not being of the same parity.

Note the following important property of $\tilde P$. At each step, the walk moves with probability 1/2. Conditionally on moving, the horizontal coordinate moves with probability 1/2, and otherwise the vertical coordinate moves (and in that case it is equally likely to go up or down by two); since we symmetrised $\Gamma$ note also that $\tilde p(x, x+ y)$ and $\tilde p(x, x-y)$ are always equal, for all $x,y \in  \Z^2$ (i.e., the jump distribution is symmetric). We will need a fair coin $\textsf{C}$ to decide which of the two $\textsf{C}$oordinates moves (if moving), and another fair coin $\textsf{L}$ to decide whether the walk is $\textsf{L}$azy or moves in this step.

\paragraph{Stage 1: vertical coordinate.} Suppose that $\tilde X_t = (u_t,v_t), \tilde X'_t = (u'_t,v'_t)$ are given. We now describe one step of the coupling. If $v_t = v'_t$ move to stage 2. If $v_t \neq v'_t$ then we consider the following two cases. In any case, we start by tossing $\textsf{C}$.
   If heads, then we plan for both $\tilde X$ and $\tilde X'$ to move their horizontal coordinates, and if tails, for both their vertical coordinates.

\begin{enumerate}
  \item Case 1: $v_t - v'_t =2 \mod 4$.
   Suppose $\textsf{C}$ is tails so the parity of vertical coordinate has a chance to be improved. Then we toss $\textsf{L}$. Depending on the result, one stays put and the other moves, or vice versa (either way the vertical coordinates are of the same parity after, and will stay so forever after). If instead $\textsf{C}$ was heads, so horizontal coordinate moves for both walks, then they move simultaneously or stay put simultaneously, and move independently of one another if at all.

   \item Case 2: $v_t - v'_t = 0 \mod 4$. Suppose $\textsf{C}$ is tails, so the vertical coordinates have a chance to be improved or even matched. Then we toss $\mathrm{L}$ and according to the result they both move simultaneously or stay put simultaneously. If moving at all, we declare the change in $v_t$ and the change in $v'_t $ to be opposite one another: thus, $v_{t+1} = v_t \pm 2$ with equal probability, whence $v'_{t+1} = v'_t \mp 2$.
       If however $\textsf{C}$ is heads (so the horizontal coordinates move), then the walks move simultaneously or stay put simultaneously, and  move independently of one another if at all.

\end{enumerate}

We leave it to the reader to check that this is a valid coupling (all moves are balanced and according to the transition probabilities $P$ if moving, and altogether each walk moves or stays put with probability 1/2 as desired).
As mentioned, once the parity of the vertical coordinates of the walks is matched (meaning the difference in vertical coordinates is even), it will remain matched forever.

Note also that once the vertical parity is matched ($v_t - v'_t = 0 \mod 4$), conditionally on the vertical coordinate moving (which is then the case for both walks simultaneously), the direction of movements is opposite: in other words, the positions of the vertical coordinates $v_t$ and $v'_t$ throughout time and until they match are mirrors of one another, with a reflection axis which is a horizontal line $L_1$. This line can be described as having a vertical coordinate equal to the average of $v_t$ and $v'_t$ at the first time $t$ that the parity of $v_t$ and $v'_t$ matches (note that $L_1$ goes via $(2\Z)^2$). In particular, the two coordinates $v_t$ and $v'_t$ \emph{will match after the first hitting time $T_1$ of the line $L_1$.}  By the end of the first stage, the two walks sit on the same horizontal line. This will remain so forever.

\paragraph{Stage 2: setting class and/or periodicity.} We now aim to match the horizontal coordinate.
If also $u_t  = u'_t$ the coupling is over and we let $\tilde X'_{t+1} = \tilde X_{t+1}$ chosen according to $\tilde P(\tilde X_t, \cdot)$. However the two walks might not be in the same class at that point, even if they started in the same class at the beginning of stage 1 (their class might change if one hits the real line but not the other during that stage).
During stage 2, we will make sure the walks become of the same class if they were not at the beginning of that stage (amounting to $u_t - u'_t$ even), and we will also make sure that they become of the same ``parity'' or ``periodicity'', meaning $u_t - u'_t = 0 \mod 4$. If that is the case already at the beginning of this stage, we can immediately move on to the next stage.

Otherwise, as before, suppose that $\tilde X_t = (u_t,v_t), \tilde X'_t = (u'_t,v'_t)$ are given, and suppose that $v_t = v'_t$. (In particular, $v_t = 0$ if and only if $v'_t = 0$.) We proceed as follows. As before, in any case we start by tossing $\textsf{C}$.
   If heads, then we plan for both $\tilde X$ and $\tilde X'$ to move their horizontal coordinates, and if tails, for both their vertical coordinates. In the latter case, we will use the same moves for both $\tilde X$ and $\tilde X'$, so we only describe what happens if the move is horizontal.

\begin{enumerate}

\item If $u_t - u'_t$ is odd, and $v_t = v'_t \neq 0$, then the walks move simultaneously and in parallel.

\item In all other situations, one walk will stay put while the other moves, or vice-versa, depending on the outcome of $\textsf{L}$.

\end{enumerate}

We make a few comments. First, note that with every visit to the real line there is a fixed positive chance to have $u_t - u'_t = 0 \mod 4$ and hence to end this stage. Also, if $u_t - u'_t$ is even to begin with, then there is also a fixed positive chance to end the stage right away.


\paragraph{Stage 3: burn-in.} In stage 3 of the coupling, we let the walks evolve in parallel (i.e., with the same jumps) until they are at distance $r$ from the real line. We will later choose $r$ as a function of $t$ (see~\eqref{rparamchoice}), which explains our comment under~\eqref{eq:tailcoupling} that $T$ depends on $t$. This is a valid choice of coupling since they will hit the real line simultaneously.  At the end of stage 2, the walks are on the same horizontal line and of the same ``periodicity'' meaning that they are 0 mod 4 apart. This will remain so until the end of stage 3.

\paragraph{Stage 4: horizontal coordinate.} As before, suppose that $\tilde X_t = (u_t,v_t), \tilde X'_t = (u'_t,v'_t)$ are given, and suppose that $v_t = v'_t$. (In particular, $v_t = 0$ if and only if $v'_t = 0$.) If also $u_t  = u'_t$ we let $\tilde X'_{t+1} = \tilde X_{t+1}$ chosen according to $\tilde P(\tilde X_t, \cdot)$. Otherwise we proceed as follows; we only describe a way of coupling the walks until hitting the real line; if coupling has not occurred before then we say that $T = \infty$. As before, in any case we start by tossing $\textsf{C}$. If tails, we let both walk evolve vertically in parallel. Otherwise, the walks will move their horizontal coordinates or stay put simultaneously depending on the result of $\textsf{L}$.
If both walks move horizontally, then let $u_t$ and $u'_t$ move in opposite manners, i.e., $u_{t+1} - u_t = - (u'_{t+1} - u'_t)$. This is possible by symmetry of the jump distribution $P$ (even on the real line).


\medskip Again, we leave it to the reader to check that what we have described in stages 2,3 and 4 forms a valid coupling. We note that any movement in the vertical coordinate is replicated across both walks, whatever the cases, and so the match created in stage 1 is never destroyed. Note also that once the walks are 0 mod 4 apart, this remains the case until hitting the real line.
%
Therefore the movement of the horizontal coordinates of both walks in stage 2 of this coupling will also be mirror off one another, with the mirror being a vertical line $L_2$ whose horizontal coordinate is the average of $u_t$ and $u'_t$ at the end of stage 3. We call $T_1, \ldots, T_4$ the end of each four stage respectively (with $T_4 $ being infinity if the walks hit the real line first).

\subsection{Suitability of coupling (proof of \eqref{eq:tailcoupling})}
\label{S:suitability}

In order to use the above coupling to construct the potential kernel of the walk on $\Gamma$, we need to verify two points.
We will consider two cases: the main one is that $x $ and $x'$ are of the same class and $\dist (x, x') = 2$. The other case is if $x, x'$ are on the real line and $\dist(x, x') = 1$. By Definition~\ref{D:potkerninf}, these two cases allow us to define the potential kernel over the entire graph. We will focus on the first case since it is a bit more involved than the second (which can be checked in a similar manner). We will first need to verify \eqref{eq:tailcoupling}, which requires that the two walks coincide with high probability at time~$t$.

We will check that each stage lasts less than $t/4$ with overwhelming probability (meaning with error probability satisfying \eqref{eq:tailcoupling}).

\textbf{Stage 1.} We may assume without loss of generality that $v_0 - v'_0 =0 \mod 4$ since otherwise it takes a geometric number of attempts until that is the case. Note then that $\P( T_1> t/4)$ is bounded by the probability that the random walk avoids the (horizontal) reflection line $L_1$ of stage 1 for time $t/4$. As the vertical coordinate performs a lazy simple random walk on the integers (with laziness parameter $3/4$) this is bounded by the probability that a random walk on the integers starting from 1 (or more generally a random value with geometric tails, as discussed above) avoids 0 for at least $\gtrsim t$, which is bounded by $\lesssim 1/\sqrt{t}$ by gambler's ruin arguments (see e.g. Proposition 5.1.5 in \cite{LawLim}).

\textbf{Stage 2.} Let $k = \text{dist} (x, \R) (= |v|)$. If the walks remained of the same class during the first stage, then stage 2 is over in a time which has a geometric tail so \eqref{eq:tailcoupling} holds trivially. On the other hand, if they did change class during the first stage, it is necessary to hit the real line again (and then wait for an extra time with geometric tail).

At the end of stage 1, the walk is on the reflection line $L_1$ which has vertical coordinate $v + O(1)$ and so is again at distance $k + O(1)$ from the real line. Let $T_\R$ denote the hitting time of $\R$. Then by Proposition 5.1.5 in \cite{LawLim} again,
$$
\P (T_\R > t/4) \lesssim \tfrac{k}{\sqrt{t}}.
$$
On the other hand, the probability that $\tilde X_{T_1}$ and $\tilde X'_{T_1}$ changed class during the first phase is bounded by $\lesssim 1/k$ again by gambler's ruin (since it requires touching the real line before the reflection line $L_1$), and so
$$
\P( T_2 -T_1> t/4; \tilde X_{T_1} \not\sim \tilde X'_{T_1}) \lesssim \tfrac1k \times \tfrac{k}{\sqrt{t}} = \tfrac1{\sqrt{t}},
$$
where $\not\sim$ denotes being of different class.
This implies \eqref{eq:tailcoupling} for $T_2- T_1$.

\textbf{Stage 3.} Here we will need to choose the parameter $r$ appropriately. We will take it to be
\begin{equation}\label{rparamchoice}
r = \frac{\sqrt{t}}{(\log t)^b},
\end{equation}
where $b>0$ can be chosen as desired. Note that every $r^2$ units of time, if the walk starts in the strip $S$ of width $r$ around the real line, it has a positive probability, say $p$, of leaving $S$ (where $p$ does not depend on the starting point of the walk). Thus for $j \ge 1$,
$$
\P( T_3 - T_2 > jr^2) \le (1-p)^j.
$$
Hence
$$
\P( T_3 - T_2 > t/4) \le \exp ( - c t/r^2)
$$
so that if  $b >1$ is any number,  the right hand side above is $\lesssim 1/\sqrt{t}$, as desired.

\textbf{Stage 4.} To prove the corresponding bound in stage 4, we need the following lemma which shows (up to unimportant logarithmic terms) the horizontal displacement accumulated in the first stage has a Cauchy tail. This corresponds of course to the well known fact that the density of Brownian motion when it hits a fixed line has exactly a Cauchy distribution.

\begin{lemma}
  \label{L:hor_displacement_stage1}
  We have
    \begin{equation}\label{E:displ}
  \P ( \sup_{t \le T_1}|u_{t} - u_0 | \ge k ) \lesssim \frac{(\log k)^2}{k}.
  \end{equation}

\end{lemma}

(The factor of $(\log k)^2$ is not optimal in the right hand side of \eqref{E:displ} but is sufficient for our purposes.) We now use this to derive a bound for $\P( T_4 - T_3 > t/4)$. From the construction of the coupling and Lemma \ref{L:hor_displacement_stage1}, we see that at the beginning of stage 4, the walk is at a distance from the (vertical) reflection line $L_4$ which has the same tail as in Lemma~\ref{L:hor_displacement_stage1} (this is because the additional discrepancy accumulated during stage 2 is easily shown to have geometric tail). Let us condition on everything before time $T_3$, and call $k$ the distance of the walk $X$ at time $T_3$ to the reflection line. Let $T_\R$ denote the hitting time of the real line and let $T_L$ denote the hitting time of the reflection line $L_4$. Set $s = t/4$ for convenience. Then we can bound the tail of $T_L$ in terms of the usual simple random walk on the lattice (without extra jumps on the real line).
Indeed, until time $T_\R$, the walk coincides with the usual \textbf{lazy} simple random walk on the square lattice. Writing $\Q$ for the law of the latter random walk, we have
\begin{align*}
 \P( T_4 - T_3 > s \mid \cF_{T_3}) & \le\P ( T_L> s, T_\R > s \mid \cF_{T_3}) + \P ( T_L > T_\R, T_\R \le s \mid \cF_{T_3})\\
& \le \Q ( T_L >s, T_\R > s) + \Q ( T_\R < T_L)\\
& \lesssim \Q( T_L>s)  + \tfrac{k}{r}\\
& \lesssim k ( \tfrac1{\sqrt{t}} + \tfrac1{r}).
\end{align*}
To go from the second line from the third line, we used that the walk starts at distance $r$ from the real line and Proposition 5.1.5 in \cite{LawLim}, and to go the last line we also used that same result. Taking expectations (we only use the above bound if $k \le r$ so that the right hand side is less than one, and we use the trivial bound 1 for the probability on the left-hand side otherwise), we see that
$$
\P( T_4 - T_3 > t/4) \le \E(X \mathbf1\{ X \le r\}) ( \tfrac1{\sqrt{t}} + \tfrac1{r}) + \P( X >r)
$$
where $X$ has a tail bounded by Lemma \ref{L:hor_displacement_stage1}. By Fubini's theorem,
\begin{equation}\label{eq:finalstage}
\P( T_4- T_3 > t/4)  \lesssim \tfrac{(\log t)^{3+b}}{\sqrt{t}}
\end{equation}
so we get \eqref{eq:tailcoupling} with $a = 3+b$. Since $b>1$ is arbitrary, $a>4$ is arbitrary.

It therefore remains to give the proof of Lemma \ref{L:hor_displacement_stage1}.

\begin{proof}[Proof of Lemma \ref{L:hor_displacement_stage1}]
  Let $L= L_1$ be the (horizontal) reflection line. We wish to show that $\P ( \sup_{t \le T_L} |u_{t} - u_0 | \ge k) \lesssim (\log k)^2/k$. Without loss of generality we assume that $v_0 < v'_0$ so $\tilde X$ starts below $L$, and $u_0 = 0$.
  Let $L'$ be a line parallel to $L$ below $L$, at distance $A $ from it, where $A = \lfloor k / ( \log k)^2\rfloor $. Let $\cS$ denote the infinite strip in between these two lines. Let $T = T_L$ denote the hitting time of $L$ and let $T'$ denote the hitting time of $L'$, and let $\tau = T \wedge T'$ denote the time at which the walk leaves the inside of the strip $\cS$. Let $T_k$ denote the first time at which $|u_t | \ge k$.
  Then
  \begin{align*}
  \P ( \sup_{t \le T}|u_{t} | \ge k)  &\le \P ( T' < T , \sup_{t\le T}| u_t| \ge k) + \P ( T' > T, \sup_{t \le T}| u_t | \ge k)\\
  & \le \P( T'< T) + \P( T_k \le \tau).
  \end{align*}
  Now, the event $T'< T$ concerns only the vertical coordinate which (ignoring the times at which it doesn't move which are irrelevant here) is simple random walk on $\Z$. Hence $\P( T'<T) = 1/A \lesssim ( \log k)^2/k$ by the gambler's ruin estimate in one dimension for simple random walk.

  It remains to show that $\P( T_k \le \tau) = o ( (\log k)^2/k)$. We split the event into two events, and show both are overwhelmingly unlikely.
  We observe that for $T_k \le \tau$ to occur, one of the following two events must occur: either (i) $T_k \le n := k^2/(\log k)^2$, or (ii) $\tau > n$. Let $E_1$ be the first event and let $E_2$ be the second one.
  Then by Lemma \ref{L:subdiffusive},
  \begin{equation}
  \P( E_1) \lesssim \exp ( -  c(\log k)^2)
  \end{equation}
  for some constant $c>0$. As for the second event $E_2$, we note that every $A^2 $ units of time there is a positive chance to leave $\cS$ (this is a trivial consequence of the fact that the vertical coordinate is lazy random walk on $\Z$, with the laziness parameter equal to $1/2 + 1/4 = 3/4$), hence
  $$
  \P(E_2) \le \exp ( - c \tfrac{n}{A^2}) = \exp ( -  c (\log k)^2)
  $$
  since $A = k/ (\log k)^2$ and $n = k^2 /(\log k)^2$. Thus
  $$
  \P( T_k \le \tau)  \le \P(E_1) + \P(E_2) \lesssim \exp ( - c (\log k)^2) = o ( k^{-1}),
  $$
and \eqref{E:displ} follows.
\end{proof}

\subsection{A priori estimate on the gradient of the potential kernel}

The purpose of this section is to show the following estimate. This will be useful both for proving that Green's functions differences converge to differences of the potential kernel in the limit of large box $\cG_n$, $n\to \infty$, but also as an input to the proof of the scaling limit result for the height function, where such an a priori estimate is needed for the inverse Kasteleyn matrix.

\begin{prop}
  \label{P:aprioriPK}
  Let $o,x, x'$ be any vertices of $\Gamma$ such that $\text{dist}(o, x) = R$ and such that $x,x'$ are of the same class with $\dist(x, x')  =2$. Then for $R \ge 2$,
  $$
  |\tilde a(x, o) - \tilde a(x',o) | \lesssim \frac{(\log R)^{c}}{R}
  $$
  where $c = a+2 > 6$, and $a>4$ is as in \eqref{eq:tailcoupling}.
\end{prop}

\begin{proof}
  Let $ t = R^2/(\log R)^2$. We note that for $s \le t$,
  \begin{align*}
  \tilde p_s(x, o)  &\le \P_o( d (X_s,o) \ge R) \\
  & \le \P_x( | u_s- u| \ge R/2) + \P ( | v_s - v | \ge R/2).
  \end{align*}
  Both terms are easily estimated. The first term is estimated by Lemma \ref{L:subdiffusive} which shows it is bounded $\exp( - (\log t)^2)$.
 The same estimate holds (and is of course easier) for the vertical coordinate, since this is simply lazy simple random walk (with laziness parameter $3/4$). Naturally this argument also holds with $x'$ in place of $x$. Thus
 \begin{equation}
 \label{PKsum1}
\Big| \sum_{s=0}^t \tilde p_s(x, o) - \tilde p_s(x',o) \Big | \lesssim \exp( - (\log R)^2).
 \end{equation}

 On the other hand, for $s \ge t$, we recall that
$$
|\tilde p_s(x,o) - \tilde p_s(x',o)| \le s^{-3/2} (\log s)^a,
$$
by \eqref{eq:tailcoupling}. Summing over $s  \ge  t$,
 \begin{equation}
 \label{PKsum2}
\Big| \sum_{s=t}^\infty \tilde p_s(x, o) - \tilde p_s(x',o) \Big | \lesssim
t^{-1/2} (\log t)^a \lesssim R^{-1} ( \log R)^{a+2}.
 \end{equation}
Combining with \eqref{PKsum1} this finishes the proof with $c = a+2$ as desired.
\end{proof}

\subsection{Convergence of Green's function differences to gradient of potential kernel}

Let $B_R = B(0, R)$. Define the unnormalised Green's function
\[
\tilde G_R (x,o) = \E_x\Big ( \sum_{n =0}^\infty \mathbf1{\{ \tilde X_n = o ,  \tau_R > n \}} \Big),
\]
where $\tau_R$ is the first time that the (lazy) walk $\tilde X$ leaves $B_R$.
We will prove the following proposition:

\begin{prop}
  \label{P:greenPK}
  As $R \to \infty$, for any fixed $x,x'$ of the same class, and any fixed $ y$,
  $$
  \tilde G_R (x, o) - \tilde G_R (x', o) \to  -( \tilde a(x,o) - \tilde a(x',o)).
  $$
  As a consequence, the same convergence is true also for the nonlazy walk $X$ instead of $\tilde X$.
\end{prop}

The proof is based on ideas similar to Proposition 4.6.3 in \cite{LawLim}. We first recall the following lemma which (in the case of finite range irreducible symmetric random walk would be  Proposition 4.6.2 in \cite{LawLim}):

\begin{lemma}\label{L:LL} For any $x, o$, we have
  $$
  \tilde G_R (x, o) =\E_x ( \tilde a ( \tilde X_{\tau_R},o ) ) -\tilde a(x,o).
  $$
\end{lemma}

\begin{proof}
The proof is simply an application of the optional stopping theorem for the martingale $M_n = \tilde a(\tilde X_n, o ) - L^{\tilde X}_n (o)$, where $L_n^{\tilde X} (o) = \sum_{m=0}^n 1_{\{\tilde  X_m = o \}}$ denote the local time of $\tilde X$ at $o$ by time $n$. The application of the optional stopping is first done at time $\tau_R \wedge n$ which is bounded. The limit when $n \to \infty$ can be taken by dominated convergence for the first term and monotone convergence for the second. In fact, for the application of the dominated convergence theorem, one must be a little more careful than with simple random walk, since when leaving $B_R$, there is an unbounded set of possibilities for $\tilde X_{\tau_R}$. However the jump probabilities decay exponentially and $\tilde a(x, o)$ grows at most like $(\log | x - o |)^{c}$ as $x \to \infty$ by Proposition \ref{P:aprioriPK} (note here that $x,x'$ and $o$ is fixed while $R \to \infty$). This makes the application of the dominated convergence justified.
We give full details of this argument for the sake of completeness.

 To this end, note that $|\tilde X_{n \wedge \tau_R}| \leq 2|\tilde X_{\tau_R} |$ almost surely. Let $B'_R\subseteq B_R$ be the set of vertices connected by an edge to the outside of $B_R$, and let $\tau'_R$ be the first hitting time of $B'_R$. By the strong Markov property we get
\begin{align*}
\E_x(|\tilde X_{\tau_R}|) &=\sum_{z\in B'_R} \P_x(\tilde X_{\tau'_R}=z)  \E_z(|\tilde X_{\tau_R}| )
\end{align*}
and
\begin{align*}
 \E_z(|\tilde X_{\tau_R}| )\leq  \P_z(\tau_R=1) \E_z(|\tilde X_{\tau_R}| \mid \tau_R =1 ) + \P_z(\tau_R>1) \max_{w\in B_R} \E_w(|\tilde X_{\tau_R}|).
\end{align*}
Plugging the latter into the former and taking the maximum over $z\in B'_R$, we obtain
for all $x\in B_R$,
\begin{align*}
\E_x(|\tilde X_{\tau_R}|) &\leq \max_{z\in B'_R}  \E_z(|\tilde X_{\tau_R}| \mid \tau_R =1 )  +\max_{z\in B'_R}  \P_z(\tau_R>1) \max_{w\in B_R} \E_w(|\tilde X_{\tau_R}|).
\end{align*}
Finally, taking maximum over $x\in B_R$, we arrive at
\begin{align*}
\E_x(|\tilde X_{\tau_R}|) \leq \frac{1}{1-\max_{z\in B'_R} \P_z(\tau_R>1)}\max_{z\in B'_R } \E_z(|\tilde X_{\tau_R}| \mid \tau_R =1 ).
\end{align*}
The quantities on the right hand side are clearly finite due to exponentially decaying probabilities for the jumps of $\tilde X$ and the fact that $z\in B'_R$. Moreover the maximums are taken over a finite set.
This together with the fact that  $\tilde a(x, o) \lesssim (\log | x - o |)^{c} \lesssim |x|+|o|$ as $x\to \infty$ completes the proof.
\end{proof}

\begin{proof}[Proof of Proposition \ref{P:greenPK}]
  By Lemma \ref{L:LL}, we have
  $$
  \tilde G_R(x, o) - \tilde G_R(x',o) = - (\tilde a (x,o) - \tilde a (x',o)) + \E_x ( \tilde a(\tilde X_{\tau_R},o)) - \E_{x'} (\tilde a (\tilde X_{\tau_R},o )),
  $$
  so it suffices to prove
  \begin{equation}\label{goalPK}
  \E_x ( \tilde a(\tilde X_{\tau_R},o)) - \E_{x'} (\tilde a (\tilde X_{\tau_R},o )) \to 0
  \end{equation}
  as $R \to \infty$. This will follow rather simply from our coupling arguments, where we will choose the parameter $r$ in the stage 3 of the coupling to be $R / (\log R)^2$.

Reasoning as in \eqref{eq:finalstage}, we see that
  \begin{equation}\label{couplingball}
  \P_x( \tau_R < T) \lesssim \frac{( \log R)^d}{R}
  \end{equation}
  for some $d>0$
  as $R \to \infty$ while $x, x'$ are fixed of the same class.
   Since we already know from Proposition \ref{P:aprioriPK} that $a(x, o)$ grows at most like $\log (|x-o|)^{c}$, \eqref{couplingball} implies that the difference of expectations in the left hand side of \eqref{goalPK} is at most $O(( \log R)^{c+d}/R )$ and so tends to zero as $R \to \infty$.
\end{proof}

We will now consider random walks which are killed on a portion of the boundary of a large box $\Lambda_R$ but may have different (e.g., reflecting) boundary conditions on other portions of the boundary. We will show that the same result as Proposition \ref{P:greenPK} holds provided that the Dirichlet boundary conditions are, roughly speaking, macroscopic. More precisely, let $\Lambda_R \subset \Z^2$ be such that $B(0, R) \subset \Lambda_R $. Let $\partial \Lambda_R$ denote its (inner) vertex boundary,
and let $\partial_{D}\Lambda_R$ denote a subset of $\partial \Lambda_R$.
Suppose that $\tilde X^\Lambda$ is a (lazy) random walk with transitions given by $\tilde p(x,y)$ if $x, y \in \Lambda_R$ and suppose that the walk is
absorbed on $\partial_{D}\Lambda_R$. We suppose that $\partial_{D}\Lambda_R$ is such that from every vertex in $\Lambda_R$, $\partial_{D}\Lambda_R$
contains a straight line segment of length $\alpha R$, and at distance at most $\alpha^{-1} R$ from $x$, where $\alpha>0$ is a (small) positive constant. Note that these assumptions are satisfied for the domains $\cG_n$ we consider (consider blue/yellow vertices separately) in~Theorem~\ref{P:infinite_vol_intro}. Indeed, the (approximate rectangles) $\cG_n$ are constructed in such a way the both the the odd and even effective bulk random walks are killed on half of the upper side of $\cG_n$.

We do not specify the transition probabilities for $\tilde X^\Lambda$ when it is on $\partial \Lambda_R \setminus \partial_{D} \Lambda_R$. Let $\tilde G^{\Lambda_R} (x,o) = \E_x (\sum_{n = 0}^\infty \mathbf 1_{\{\tilde X^{\Lambda_R}_n = o \}})$ denote the corresponding unnormalised Green's function.

\begin{prop}
  \label{P:greenPK_gen}
  As $R \to \infty$, for any fixed $x,x'$ of the same class and any fixed  $o$,
  $$
  \tilde G^{\Lambda_R} (x, o) - \tilde G^{\Lambda_R} (x', o) \to -(\tilde a(x,o) - \tilde a(x',o)).
  $$
  As a consequence, the same convergence is true also for the nonlazy walk $X$ instead of $\tilde X$.
\end{prop}

\begin{proof}
For this proof we will need the following lemma, which says that from any point there is a good chance to hit the boundary without returning to the point, whence the expected number of visits to that point before hitting the boundary is small.

\begin{lemma}\label{L:hitline}
  There exists a constant such that the following holds for all $k\ge 2$ and vertex $o$ of $\Gamma$. Let $L$ be a lattice line at distance $k$ from $o$ and of same class as $o$. Then
  \begin{equation}\label{E:hitline}
  \P_o(T_L < T^+_o) \gtrsim (\log k)^{-1},
  \end{equation}
  where $T_L$ is the hitting time of $L$, $T^+_o$ is the return time to $o$.
  \end{lemma}

\begin{proof}[Proof of Lemma \ref{L:hitline}]
We start by noticing that, up to a factor equal to the total conductance at $o$, the probability on the left-hand side is equal to the effective conductance (or inverse of the effective resistance $\Reff (o ; L)$) between $o$ and $L$. Since the total conductance at $o$ is bounded away from 0 and $\infty$, it suffices to show that
$$
\Reff (o; L) \lesssim  \log k.
$$
This can either be proved directly or by comparison with the analogous estimate on $\Z^2$ through Rayleigh's monotonicity principle (see Chapter II of \cite{LyonsPeres}). A direct proof is to construct a unit flow $\theta$ from $o$ to $L$ and estimating its Dirichlet energy $\cE (\theta) = \sum_e \theta(e)^2 \text{res}(e)$, where $\text{res}(e)$ denotes the resistance of $e$. Such a unit flow can be constructed by the method of random paths, as discussed in (2.17) of \cite{LyonsPeres}: we consider a cone of fixed aperture whose apex is at $o$ and intersects $L$, then choose a line at random in that cone starting at $o$ and whose angle is uniformly selected among the set of possibilities. We get a directed lattice path $\pi$ from $o$ to $L$ by selecting a lattice path staying as close as possible to this random line (with ties broken in some arbitrary way), staying on the same sublattice as $o$ and $L$. Note that this path never uses long range edge along the real line, and in fact jumps only by $\pm 2e_i, i = 1, 2,$ at any given steps. A unit flow $\theta$ from $o$ to $L$ is obtained by setting $\theta (e) = \P ( e \in \pi) - \P ( -e \in \pi)$ (where $-e$ denotes the reverse of the edge $e$). Then if $e$ is at distance $j$ from $o$, $$
|\theta (e)| \le \P(e \in \pi) + \P ( - e \in \pi)  \lesssim \frac1j$$
since there are $O(j)$ edges at distance $j$. Hence
$$
\cE( \theta) \le \sum_{j=1}^k O(j) \frac1{j^2} \lesssim \log k.
$$
Since the effective resistance is smaller than the energy of any flow from $o$ to $L$, we get the desired bound.
\end{proof}

Now let us return to the proof of Proposition \ref{P:greenPK_gen}. We apply the full plane coordinatewise mirror coupling of Proposition \ref{P:greenPK} (that is, with the parameter $r$ chosen to be $R/(\log R)^2$), until the time $S_R$ one of the walks leaves the ball $B_R = B_R(0)$. If they have not coupled before $S_R$, we consider this a failure and will not try to couple them after: we let them evolve independently.

Then note that (we write $\tilde X$ for $\tilde X^{\Lambda_R}$ for simplicity)
\begin{align}
   \tilde G^{\Lambda_R} (x, o) - \tilde G^{\Lambda_R} (x', o)& = \E_x( L^{\tilde X}_{S_R} (o) )  - \E_{x'}( L^{\tilde X'}_{S_R} (o) ) \label{PKbigbox} \\
   & \ \ + \E \big( L^{\tilde X}_{(S_R, \infty)} (o)   -  L^{\tilde X'}_{(S_R, \infty)} (o) \big)\label{PKnotbigbox}
\end{align}
Note that the term in \eqref{PKbigbox} converges to $-(\tilde a(x,o) - \tilde a(x', o))$ by Proposition \ref{P:greenPK}. So it suffices to show that the term in \eqref{PKnotbigbox} converges to zero. However, this is an easy consequence of the following facts:

\begin{itemize}
\item If the coupling was successful before $S_R$, then the random variable in the expectation of \eqref{PKnotbigbox} is zero.

\item The probability that the coupling has failed (i.e., that the walks did not meet before leaving $B_R$) is  $\lesssim( \log R)^{4}/R$, by \eqref{couplingball}.

\item Conditionally on not having coupled by time $S_R$, the expected number of visits to $y$ after that time is  $\lesssim \log R$ by Lemma \ref{L:hitline} and by assumption on the Dirichlet part $\partial_{D}\Lambda_R$. (In fact, the lemma is stated for hitting an infinite line, but it is easy checked that the argument shows it is a segment of macroscopic size that is being hit with the stated probability).
\end{itemize}
This completes the proof.
\end{proof}

We apply this to the bulk effective random walk of Section \ref{S:inverseKfinite}.
This yields the following corollary which also concludes the proof of Theorem~\ref{P:infinite_vol_intro}.

\begin{corollary}
  \label{C:infinitevol} Let $D_n$ be an increasing sequence of domains such that $\cup_n D_n = \Z^2\cap \H$ as in Theorem~\ref{P:infinite_vol_intro}. Consider the free boundary dimer model on $D_n$ with weights as described in Corollary~\ref{C:finitevol}. Then, the inverse Kasteleyn matrix converges pointwise as $n \to \infty$ to a matrix indexed by the vertices of $ \Z^2\cap \H$, called the \textbf{coupling function}, and given in matrix notation by
    $$
  C = -AK^*,
   $$
  where
  \[
  A(u,v)=\frac1{2D(v,v)}\tilde a(u,v)=\frac1{D(v,v)} a(u,v)
  \]
is the normalised potential kernel associated with the effective (odd and even) bulk (nonlazy) random walks.

 In particular, $\mu_n$ converges weakly as $n \to \infty$ to a law $\mu$ which describes a.s.\ a random monomer dimer configuration on $\Z^2 \cap \H$.
\end{corollary}
\begin{proof}
The first part of the statement follows from the random walk representation of $K^{-1}$ in finite volume from Corollary~\ref{C:finitevol}, the interpretation of $K^*$ as a difference operator,
and the convergence of differences of Green's functions of the bulk effective walk from Proposition~\ref{P:greenPK_gen}.

The convergence in law is a standard application of Kasteleyn theory.
Indeed, this follows from the fact the local statistics of $\mu_n$ are described by local functions of the inverse Kasteleyn matrix (which we will for instance recall in Theorem~\ref{thm:Kasteleyn}).
It is also clear that $\mu$ is supported on monomer-dimer configurations on $\Z^2\cap \H$.
\end{proof}

\section{Scaling limit of discrete derivative of potential kernel}
\label{S:PKscaling}
Let $x, y \in \bar \cG_\delta := (\delta \Z)^2 $. The purpose of this section will be to prove a scaling limit for the discrete derivatives of the potential kernel $\tilde a(x,y) - \tilde a(x',y)$ associated to the lazy (odd) effective random walk. (Contrary to the previous section, the second variable will more typically be called $y$ than $o$ in this section). As mentioned at the beginning of Section \ref{S:IVL}, the same result holds for both the even and odd walk, but for convenience (and also because this is a slightly more complicated case) we write our proofs in the odd case.

\begin{thm} \label{T:PKscaling}
Let $x'= x \pm 2\delta e_i \in  \bar \cG_\delta$, $i =1, 2$, and $y \in \bar \cG_\delta $. Suppose $\Im (x) \Im (y) \ge 0$ and $\min( |\Im (x)|, | \Im (y) | )\ge \rho$ for some fixed $\rho>0$.
Then there exists $\eps>0$ such that as the mesh size $\delta \to 0$, uniformly over such points $x, y$,
\begin{equation}
  \tilde a(x', y) - \tilde a(x,y) =
  \begin{cases}
    \dfrac{2}{\pi} \Re\Big (\dfrac{x'-x}{x-\bar y}\Big) + o(\delta^{1+ \eps}) & \text{ if  $x,y$ are of different class}\\
\dfrac{4}{\pi} \Re \Big(\dfrac{x'-x}{x-y}\Big)
-  \dfrac{2}{\pi} \Re \Big(\dfrac{x'-x}{x-\bar y}\Big)  + o(\delta^{1+ \eps}) + O( \frac{\delta}{|x-y|})^{2}
  & \text{ if  $x,y$ are of the same class}
  \end{cases}\label{eq:PKscaling}
\end{equation}
\end{thm}

To prove this theorem, we will first show that the potential kernel can be compared to that of a \textbf{coloured random walk} on the lattice. {The coloured random walk is a lazy simple random walk on the lattice $(2\delta\Z)^2$ which carries a black or white colour (in addition to its position). Its position moves like simple random walk on the lattice. It changes colour with some fixed probability $p\in (0,1)$ each time it touches the real line independently of the rest, and otherwise remains constant. If $X$ is a coloured random walk, we will use $\sigma(X_s)$ to denote the colour of the coloured walk $X$ at time $s$ (and again, this is different from the colour of the vertex $X_s$): thus, we will write $\sigma(X_s)=\bullet$ if $X$ is black at time $s$, and $\sigma(X_s)=\circ$ if $X$ is white at time $s$. Although $X_s$ consists both of a position $x \in (2\delta \Z)^2$ and a colour, we will sometimes with an abuse of notation refer to $X_s$ as only a position.

\begin{remark}We warn the reader that this should \emph{not} be confused with the black/white colouring (which we call class precisely to avoid confusion) of the vertices of our graph $\bar \cG_\delta$: indeed, the position of the coloured walk is in $(2\delta \Z)^2$ and so its ``class'' in $\bar \cG_\delta$ remains constant.
\end{remark}

Note that $x$ and $x'$ are necessarily of the same class (hence the same colour). However, $y$ may be of a different colour. 
We will choose $p$ to correspond to the probability that the odd effective walk makes a jump of odd length when it touches the real line: thus,
\begin{equation}\label{paritychange}
p ={\tfrac 14 \sum_{k\in \mathbb Z} q^{\infty}_{0,(2k+1)e_1}}
\end{equation}
where $q^{\infty}$ is as in \eqref{qinfty}.

We will prove the following two results. Let $y, x \in (2\delta\Z)^2$ and choose a colour among $\{ \circ, \bullet\}$, say $\bullet$. Let $\tilde a^{\bullet} (x,y)$ denote the potential kernel of the coloured random walk, constructed as in Definition \ref{D:pk} but only counting visits to $y$ with the predetermined colour $\bullet$: that is,
$$
\tilde a^\bullet(x,y) = \sum_{s=1}^\infty \P ( {X_s = y; \sigma (X_s) = \bullet})
$$
where $X$ is a coloured walk starting from $x$ with initial  colour $\bullet$.
The fact that the series defining $\tilde a^\bullet$ converges is an immediate consequence of the arguments in Section \ref{S:suitability}, which apply much more directly here.}

The first result below shows that the potential kernel of the lazy effective walk and of the coloured walk are quite close to one another, in the sense that the difference in their discrete derivatives are of lower order than $\delta$, our target for Theorem \ref{T:PKscaling}. In the next statement we write $y\not \sim x$ to denote that $x$ and $y$ are of different class.

\begin{prop}\label{P:eff_col}
Fix $\rho>0$. Let
$x, y \in \cG_\delta$, and let $z = x + \delta \mathbf 1_{y \not \sim x }$ (resp $z' = x' + \delta \mathbf 1_{y \not \sim x}$), so that $z$ and $z'$ are of the same class as $y$. Let us write $\nabla_x f(x) $ for $f(x') - f(x)$ (resp. $\nabla_z f(z) = f(z' ) - f(z)$).
Then there exists $\eps>0$ such that as $\delta \to 0$,
$$
  | \nabla_x \tilde a(x, \bar y) - \nabla_z \tilde a^\bullet (z, \bar y) | \lesssim \delta^{1+ \eps},
$$
uniformly over $x,y$, with $\min (\Im (x), \Im (y)) \ge \rho$.
\end{prop}

The next proposition says that the potential kernel of the coloured walk is close to $1/2$ that of the regular lazy simple random walk: this is because when the walk touches the real line, it does so many times in a row typically, and so is roughly equally likely to end up with the colour $\bullet$ or $\circ$. Moreover in the above setting the walk is forced to touch the real line in order to go from $x$ to $\bar y$. Let $\tilde b(x, y) = \tilde b(x-y)$ denote the potential kernel of lazy simple random walk on $(2\delta\Z)^2$.

\begin{prop}\label{P:col_half} In the same setting as Proposition \ref{P:eff_col},
$$
  | \nabla_x \tilde a^\bullet (z, \bar y) - \frac12\nabla_z \tilde b(z, \bar y)   | \lesssim \delta^{1+ \eps},
$$
for some $\eps>0$.
\end{prop}

\begin{proof}[Proof of Theorem \ref{T:PKscaling} given Proposition \ref{P:eff_col} and Proposition \ref{P:col_half}]
It is enough to combine Propositions \ref{P:eff_col} and \ref{P:col_half} as well as known estimates on the two-dimensional simple random walk potential kernel.

Let us give a few details. Suppose we are in the first case where $x, y$ are of different class. This means only walks going through the boundary have the possibility to contribute to the potential kernel. By the reflection symmetry, the walks from $x$ to $y$ going through the boundary have the same weight as the walks from $x$ to $\bar y$.
 In the full plane for simple random walk,
(see e.g. Theorem 4.4.4. in \cite{LawLim}), the potential kernel has the form
$$
b(z, 0) = \frac{2}{\pi} \log |z | + C + o( |z|^{-1})
$$
for some constant $C>0$, as $z \to \infty$. Let us rescale the lattice so that it becomes $\delta \Z^2$, and let us adopt complex notation, so $\log |x| = \Re ( \log x)$, and let $h = x'- x = \pm 2\delta e_i$. Then
\begin{align}
b(x, \bar y) - b ( x',\bar y )  &= \frac2{\pi} \Re ( \log (x- \bar y + h ) - \log (x -\bar y)) + o( \delta)\nonumber \\
& = \frac{2}{\pi}\Re \left ( \frac{h}{x-\bar y}  \right) + o( \delta).\label{boundary}
\end{align}
Now, multiplying by 2 to account for laziness, and by $1/2$ to account for the loss at the boundary (the real line) (Proposition \ref{P:col_half}) and we get the first line in \eqref{eq:PKscaling}.

To get the second line, we observe that if $x$ and $y$ are of the same class, there are two types of effective random walks to consider: the effective random walks going from $x$ to $y$ in the full plane without touching the boundary (type I), and those which do touch the boundary (type II). The effective random walks of type I can be written as all simple random walks going from $x$ to $y$
in the plane (type III) minus simple random walks going from $x$ to $y$ through the boundary (type IV).
By Propositions \ref{P:eff_col} and \ref{P:col_half}, the walks of type IV contribute roughly twice as much as those of type II. So we have to count walks of type III minus those of type II. Those of type III contribute $\tfrac{4}{\pi} \Re ( \tfrac{x'- x}{x-y} ) + O( \tfrac{\delta}{|x-y|})^{2}$ to the gradient of the potential kernel (the factor in front is twice that of \eqref{boundary} due to laziness, the error term comes from Corollary 4.4.5 in \cite{LawLim}). The contribution of type II on the other hand is exactly counted by the first line of \eqref{eq:PKscaling}. This proves Theorem \ref{T:PKscaling}.
\end{proof}

Now we derive the version which is useful for later, which includes folding the plane onto itself so that the walk is reflected on the real line, and is not lazy.

\begin{corollary} \label{C:PKscaling}
Let us assume that $x'= x \pm 2\delta e_i \in  \delta \Z^2 \cap \H$, $i =1, 2$. Let $y \in \delta \Z^2 \cap \H $.
Then there exists $\eps>0$ such that as the mesh size $\delta \to 0$, uniformly over points $x,y$ such that $\min (\Im (x), \Im (y) ) \ge \rho>0$,
\begin{align}
   a(x', y) -  a(x,y) =
  \begin{cases}
    \dfrac{2}{\pi} \Re \Big(\dfrac{x'-x}{x-\bar y}\Big) + o(\delta^{1+ \eps}) & \text{ if  $x,y$ are of different class,}\\
\dfrac{2}{\pi} \Re \Big (\dfrac{x'-x}{x-y}\Big)+ o(\delta^{1+ \eps}) +
O( \frac{\delta}{|x-y|})^{2}
  & \text{ if  $x,y$ are of the same class.}
  \end{cases}\label{eq:PKscaling2}
\end{align}
\end{corollary}

\begin{proof}[Proof of Corollary \ref{C:PKscaling} given Theorem \ref{T:PKscaling}.] As before the first case (when $x, y$ are of different classes) is easiest to compute. Since the walk is now nonlazy, we need to multiply the values of the potential kernel by $1/2$, but also add the walks from $x$ to $\bar y$; both are counted by the same formula in the first line of \eqref{eq:PKscaling}, and so the factor remains $2/\pi$ overall.

In the second case when $x,y$ are of the same class, we note that the number of lazy walks from $x $ to $y$ that don't touch the boundary are, as observed above, given by $\tfrac{4}{\pi} \Re ( \tfrac{x'- x}{x-y} )$ (type I). On the other hand, when we do the folding, we must add the walks that touch boundary and go from $x$ to $y$, to those going from $x $ to $\bar y$. This gives us one extra group of walks of type II and so these cancel. Multiplying by $1/2$ to account for non-laziness gives us the second line of \eqref{eq:PKscaling2}.
\end{proof}

Thus it remains to prove the two propositions \ref{P:eff_col} and \ref{P:col_half}. We do so in the following two subsections.

\subsection{Proof of Proposition \ref{P:eff_col}.}

We will prove this by coupling.
We will need to compare $\nabla_x \tilde p_t(x, o)$ and $\nabla_z \tilde p^\bullet_t (z, o)$, where
$\tilde p^\bullet_t (z, o) = \P_z( X_t = o, \sigma(X_t) = \bullet)$ for the coloured walk, where we take $o=\bar y$, and $z$ is a vertex chosen as in Proposition~\ref{P:eff_col}. We will see that by coupling our effective walks with coloured walks we will gain an order of magnitude compared with \eqref{eq:tailcoupling}: that is, we will show that
\begin{equation}\label{eq:effcolHK}
\big| \nabla_x \tilde p_t(x, o)- \nabla_z \tilde p^\bullet_t (z, o)\big| \le t^{-3/2 - \eps}; \ \ \  t \le \delta^{-2- \eps},
\end{equation}
for some $\eps>0$. Given \eqref{eq:effcolHK}, reasoning as in the proof of Proposition \ref{P:aprioriPK} (with $R = \delta^{-1}$, and using the improved \eqref{eq:effcolHK} instead of \eqref{PKsum2} in the range up to $t = \delta^{-2 - \eps}$), we immediately deduce Proposition \ref{P:eff_col}.

We will couple the effective walk $X$ and a coloured walk $Z$ as follows; as in the previous section we work with lazy versions. The coupling will be similar to the one in Section \ref{S:suitability}, but it is simpler since we are allowed to choose the starting point of $Z$. We will choose $z$ so that $X$ and $Z$ start immediately from the same horizontal line; as in the previous coupling this property will be preserved forever under the coupling, (so essentially only the last stage, stage 4, needs to be described). More precisely, we set $z=x$ if $x$ and $\bar y$ are of the same class, and $z = x+ \delta $ otherwise. In any case $Z$ will always be of the same class as $o$.
Until hitting the real line, we take $X$ and $Z$ to evolve in parallel, with equal jumps. After hitting the real line, we may arrange the coupling so that they are always on the same horizontal line by always first tossing the {\sf C}oordinate coin, so that any movement in the vertical coordinate is replicated for both walks no matter what. Beyond the {\sf C}oordinate and {\sf L}aziness coins, we will need a third coin which we use to indicate changes in the sublattice (for $X$) and in colour (for $Z$). This coin is only used when the walks are on the real line and a horizontal movement is to take place. We call this coin {\sf P}arity. Unlike the other two coins, {\sf P}arity comes up heads with the fixed probability $p \in (0,1)$ from \eqref{paritychange} which in general is not $1/2$.

It remains to specify what to do if the {\sf C}oordinate coin indicates a horizontal movement. To describe this, we need to introduce the following stopping times. Let $\sigma_0 = \inf\{ t \ge 0: X_t \in \R\}$ denote the hitting of $\R$ by $X$ (or equivalently by $Z$), and let $\tau_0 = \inf \{ t \ge \sigma_0: \Im (X_t) \le  \Im(\bar y)/3\}$ be the hitting time of the line $$\Delta = \{  z\in \C: \Im (z) = \lfloor \Im (\bar y) / 3\rfloor \}$$
by $X$ (or equivalently $Z$, since $X$ and $Z$ are always on the same horizontal line). Then define $\sigma_n, \tau_n$ inductively as follows: $$\sigma_n = \inf \{ t \ge \tau_{n-1}: X_t \in \R\};  \ \ \tau_n = \inf \{t \ge \sigma_n: X_t \in \Delta\}.$$

Write $X_t = (u_t, v_t)$ and $Z_t = (u'_t, v'_t)$ with $v_t = v'_t$ as explained above.


\begin{itemize}

\item If $X_t, Z_t \in \R$. Toss the {\sf P}arity coin: if it comes heads, let $X_t$ take a jump from its conditional distribution given that it is odd, and let $Z_t$ change colour and make an independent jump. If it is tails, let $X_t$ take a jump from its conditional distribution given that it is even, and let $Z_t$ keep its current colour and make an independent jump.

  \item Now suppose $X_t, Z_t \notin \R$. If $X_t$, $\bar y$ are of a different class, then let $X_t$ and $Z_t$ evolve in parallel (with equal jumps). This will remain so until hitting again the real line, where there will be a chance to change class again.

  \item $X_t$, $\bar y$ are of the same class, and thus also of the same class as $Z_t$. In that case, the evolution depends on whether $t \in [\sigma_n, \tau_n]$ for some $n\ge 0$ or $t \in (\tau_n, \sigma_{n+1})$ for some $n$: If $t \in [\sigma_n, \tau_n]$ then the walks evolve in parallel. Otherwise,
      we use {\sf L}aziness to first ensure that $u_t - u'_t = 0 \mod 4\delta$ after a number of steps which has geometric tail. Once that is the case, we let $u_t$ and $u'_t$ evolve in mirror from one another, so $(u_{t+1} - u_t) = - (u'_{t+1} - u'_t)$.
\end{itemize}

In general the walks get further from each other during a phase of the form $[\sigma_n, \tau_n]$ but get closer together again during the phase $[\tau_n, \sigma_{n+1}]$. Note that a visit to $o$ necessarily occurs during such a phase. In fact we will see that typically the walks agree (if they are on the same sublattice) by the time they reach $2\Delta$ or return to $\R$. Furthermore, only a small number of phases need to be considered if $t \le \delta^{-2 - \eps}$ (of order at most $\delta^{-\eps}$). Let us say that a non coupled visit to $o$ occurs at time $t$ if $\{X_t = o\} \triangle \{ Z_t = o, \sigma(Z_t) = \bullet\}$ occurs (where $\triangle$ denotes symmetric difference).

The coupling between $X$ and $X'$ on the one hand, and between $X$ and $Z$ on the other hand, induce a coupling between four processes: $X,X'$ (effective walks starting from $x,x'$) and $Z,Z'$ (coloured walks started from $z,z'$). Here we take $z'- z = x' - x = \delta$, as in the statement of Proposition \ref{P:eff_col}. The difference between the gradient of the transition probabilities can be written as an expectation
\begin{equation}\label{eq:gradient_cancell}
\nabla_x \tilde p_t (x, o) - \nabla_z \tilde p^\bullet_t (z, o) = \E( \mathbf1_{\{X_t =o\}} - \mathbf1_{\{X'_t = o\}} - \mathbf1_{\{Z_t = o; \sigma(Z_t) = \bullet\}} + \mathbf1_{\{Z'_t = o; \sigma(Z'_t) = \bullet\}})
\end{equation}
To get a nonzero contribution it is necessary that $X$ did not couple with $X'$ by time $(T_\R \wedge t/2)$ or that $Z$ did not couple with $Z'$ by time $t/2$. Both have a probability which is given by $(\log t)^a/t^{1/2}$ by a slight modification of \eqref{eq:tailcoupling} (in fact, since the walks start far from the real line, the proof is much simpler than what is given in Section \ref{S:suitability}, and follows directly from gambler's ruin).
Furthermore, given this, it is also necessary that a non coupled visit to $o$ occurs at time $t$ by $(X,Z)$ or by $(X',Z')$.

To estimate the latter conditional probability, we may condition on everything which happened until time $T_\R \wedge t/2$, and we will call $s$ the remaining amount of time until time $t$, i.e., $s = t - (T_\R \wedge t/2) \in [t/2, t ]$ so $s \asymp t$. Since at that time the walks have yet not touched the real line, the discrepancy between $X$ and $Z$ is therefore equal to the initial discrepancy $z - x \in \{0, \delta e_1\}$.

\begin{lemma}\label{L:backforth}
Suppose $s \le \delta^{-2-\eps}$.
Let $N_s = \max\{k: \tau_k \le s\}$. Then there exists some $c_1,c_2>0$ such that $\P( N_s \ge c_1\delta^{-\eps})\le \exp( - c_2 \delta^{-\eps} )$.
\end{lemma}

\begin{proof}
  Each journey between $\R$ and $\Delta$ and back may take more than $\delta^{-2}$ with fixed positive probability $p$, independently of one another. Hence the probability in the lemma is bounded by the probability that a Binomial random variable with parameters $c_1\delta^{-\eps}$ and $p$, is less than $\delta^{-\eps}$. Choosing $c_1$ such that $c_1  p>1$, the result follows from straightforward large deviations of binomial random variables.
\end{proof}

We will need to control the discrepancy between $X$ and $Z$ at the beginning of a coupling stage, of the form $\tau_k$ (for $0 \le k \le \delta^{-\eps}$), assuming that $\sigma(Z_{\tau_k}) = \bullet$ or equivalently that $X_{\tau_k} \sim \bar y$. Let us say that this coupling phase \textbf{succeeds} if by the time the walks next hit $\R$ or $2\Delta$, the discrepancy has been reduced to zero.

We note that the discrepancy between $X$ and $Z$ is typically accumulated when the two walks hit the real line; on the other hand they tend to be reduced to zero during a coupling phase, meaning a coupling phase is likely to be successful. However, we will not aim to control the discrepancy if at any point the coupling phase does not succeed.

The key argument will be to say that so long as there has been no unsuccessful coupling phase, the discrepancy at the beginning of any coupling phase is small. To this end, we introduce $\rho_n$ the first time that the real line has been visited more than $n$ times by either (both) walks. We let $\Delta_n$ the (horizontal) discrepancy accumulated by the walks at this $n$th visit: that is,
$$
\Delta_n =  \langle (X_{\rho_n +1} - X_{\rho_n})  - (Z_{\rho_n +1} - Z_{\rho_n}) ; e_1 \rangle
$$
Note that by construction of the coupling, $\Delta_n$ are i.i.d. and centered random variables with exponential moments (each of them of order the mesh size $\delta$). We then introduce the martingale
$$
M_n = \sum_{i=0}^n \Delta_i
$$
which counts the accumulated discrepancy at the $n$th visit to the real line. If $0\le u \le s$ is a time, let us call $n(u)$ the number of visits to $\R$ by time $u$. At the end of a successful coupling phase $\sigma_k$, the discrepancy is reduced to zero, so in fact in the future (until the beginning of the next coupling phase at time $\tau_k$), the discrepancy will be of the form $M_{n(u)} - M_{n(\sigma_k)}$.

\begin{lemma}\label{L:discrep}
  With probability at least $1 - s^{-2\eps}$, we have
  $$
  \max_{0 \le k \le N_s} | X_{\tau_k} - Z_{\tau_k}|\mathbf1_{\cG_k} \le \delta s^{1/4 + \eps},
  $$
  where $\cG_k$ is the good event that there was no unsuccessful coupling by time $\sigma_k$.
\end{lemma}

\begin{proof}
  Fix $0 \le k \le N_s$. Let $j = j (k) = \max\{ j \le k : \text{ the coupling starting at $\tau_j$ was successful} \}$. Suppose that the event $\cG_k$ holds otherwise there is nothing to prove. Then as observed above, the discrepancy at time $\tau_k$ is given by
  \begin{align*}
  |X_{\tau_k} - Z_{\tau_k}| & = |M_{n(\tau_k)} - M_{n(\tau_j)}| \le 2 \max_{n \le n (\tau_k)} |M_n|.
  \end{align*}
  By Chebyshev's inequality and Doob's maximal inequality,
  \begin{align*}
   \P \Big(\max_{0 \le k \le \delta^{- \eps}} | X_{\tau_k} - Z_{\tau_k}|\mathbf 1_{\cG_k} \ge \delta s^{1/4 + \eps}\Big)
  & \lesssim \frac1{\delta^2 s^{1/2 + 2 \eps}} \E \Big(  \max_{n \le n (\tau_{N_s})} |M_n|^2\Big)\\
  & \lesssim  \frac1{\delta^2 s^{1/2 + 2 \eps}} \E\Big(M_{n (\tau_{N_s})} ^2\Big).
  \end{align*}
  Now, $M_n^2 - c\delta^2 n$ is a martingale for some  constant $c>0$ corresponding to the (rescaled) variance of the increments of the martingale $M$, so (since $n( \tau_{N_s}) $ is trivially bounded by $s$),
  $$
  \E\left(M_{n (\tau_{N_s})} ^2\right) = c \delta^2 \E ( n (\tau_{N_s})) = c \delta^2 \E ( L_\R (s)),
  $$
  where $L_\R (s)$ denote the number of visits to $\R$ by both (either) walks by time $s$. Since the vertical coordinate performs a delayed simple random walk on the integers, this is less than the expected number of visits to 0 by time $s$ of a one-dimensional walk starting from zero, which is at most $\lesssim \sqrt{s}$. Hence
  $$
  \P \Big(\max_{0 \le k \le \delta^{- \eps}} | X_{\tau_k} - Z_{\tau_k}| \mathbf 1_{\cG_k} \le \delta s^{1/4 + \eps}\Big)\lesssim  \frac1{s^{2\eps}}
  $$
  as desired.
\end{proof}

We now deduce that all coupling phases are successful with high probability.
\begin{lemma}
  \label{L:successfulcoupling} We have that for $\eps$ small enough (fixed),
  $$
  \P \left( \cup_{k=0}^{N_s} \cG_k^c\right) \lesssim s^{-2\eps}.
  $$
\end{lemma}

\begin{proof}
  We may work on the event $\cN = \{ N_s \lesssim \delta^{-\eps}\}$ and the event $\cD$ of Lemma \ref{L:discrep}. On $\cN\cap \cD$ the probability of an unsuccessful coupling starting from time $\tau_k$ may be bounded as follows. Supposing that $\sigma(Z_{\tau_k}) = \bullet$ (or equivalently $X_{\tau_k} \sim \bar y$), the walks $X$ and $Z$ start a mirror coupling at time $\tau_k$ and they are initially spaced by no more than $\delta s^{1/4 + \eps}$, if $\cG_{k-1}$ holds. By the gambler's ruin estimate, the probability for $X$ to avoid the reflection line until hitting either $\R$ or $2 \Delta$ is then at most $\delta s^{1/4 + \eps}$. Hence
  $$
  \P ( \cG_k^c ; \cG_{k-1 } \cap \cN \cap\cD) \le \delta s^{1/4 + \eps} \lesssim \delta^{1/2 - 3 \eps}.
  $$
  Summing over $k \le \delta^{-\eps}$, we get
  $$
  \P \left( \cup_{k=0}^{N_s} \cG_k^c ; \cN \cap \cD\right) \lesssim \delta^{1/2 - 4 \eps}.
  $$
  We conclude by Lemma \ref{L:discrep} and Lemma \ref{L:backforth}.
\end{proof}

\begin{proof}[Proof of Proposition \ref{P:eff_col}.]
We estimate the right hand side of \eqref{eq:gradient_cancell}. For the random variable in the right hand side to be nonzero, it is necessary that:
\begin{itemize}
  \item $X$ and $X'$ did not couple prior to time $T_\R \wedge t/2$;
  \item one of the $\cG_k^c$ occurs for some $k \le N_s$;
  \item and still one of the four walks must visit $\bar y$ at exactly time $t$.
\end{itemize}
The first event has probability bounded by $\lesssim 1/\sqrt{t}$ by straightforward gambler's ruin. The second has probability at most $1/ t^{2\eps}$ by Lemma \ref{L:successfulcoupling} (since $ s\asymp t$). To bound the probability of the third event, we observe the following: if $w \in 2\Delta$, the maximum over all times $u$ of the probability to visit $\bar y$ at the specific time $u$ is small:

\begin{lemma}
  \label{L:visit}
  We have
  $$
  \sup_{w \in 2\Delta} \sup_{u \ge 0} \tilde p_u(w, \bar y) \le \delta^2 (\log 1/\delta)^c,
  $$
  for some $c>0$.
\end{lemma}

\begin{proof}
  This follows from the facts (already used before, so we will be brief) that if $u \le \delta^{-2}/(\log 1/\delta)^c$ then the probability to be at $\bar y$ at time $u$ is at most $\exp( - (\log 1/\delta)^2)$ by subdiffusivity, while for $u \ge \delta^{-2}/(\log 1/\delta)^c$ we have a bound of the form $1/ u$ thanks to \eqref{aprioriHK}.
\end{proof}
All in all, putting these three events together we find
$$
\big| \nabla_x \tilde p_t(x, o)- \nabla_z \tilde p^\bullet_t (z, o)\big| \lesssim
t^{-1/2} \times t^{-2\eps} \times \delta^2 (\log 1/\delta)^c
$$
Summing over $t \in [\delta^{-2}/\log (1/\delta)^c, \delta^{-2 - \eps}]$ we see that this is at most
$
(\log 1/\delta)^c \delta^{ 1+ 7 \eps/2},
$
which is sufficient.
\end{proof}

\subsection{Proof of Proposition \ref{P:col_half}}

At this point we may work exclusively with the simple random walk on $(2 \Z) \times (2\Z)$ or the coloured simple random walk on the same lattice.
Let us write $\P_{\br}$ for the law of a random walk bridge, i.e., the law of a (lazy) simple random walk on $(2\Z)^2$ conditioned to go from $ x $ to $\bar y$ in time $t$.

Let $\tilde q_t(x,y)$ denote the transition probability for (lazy) simple random walk on $(2\Z)^2$. Then note that
$$
\tilde p^\bullet_t (x,y) = \tilde q_t(x,y) \P_{\br} ( \sigma(X_t) = \bullet),
$$
where $\sigma (X_t)$ is the colour of the process which changes with probability $p$ every time this process touches the real line. Now, let $N$ denote the number of visits to $\R$ and observe that by conditioning on $N$,
$$
\P_{\br} ( \sigma(X_t) = \bullet | N = n )  = \frac12 \pm \frac12\lambda^n
$$
where $\lambda = 1- 2p$ is the eigenvalue of the $2$-state Markov chain which switches state with probability $p$ at each step, and the $\pm$ sign depends on the initial colour $\sigma(X_0)$. Therefore,
\begin{align*}
  \nabla_x \tilde p^\bullet_t (x,\bar y) & = \frac12 \nabla_x \tilde q_t(x,\bar y) \pm \frac12 \nabla_x \Big( \tilde q_t(x, \bar y) \E_{\br} (\lambda^N) \Big) .
\end{align*}
Since $\sum_{t=0}^\infty \frac12 \nabla_x \tilde q_t(x,\bar y)$ is by definition the potential kernel of the (lazy) simple random walk $\frac12\nabla_x \tilde b(x,\bar y)$, to prove Proposition \ref{P:col_half}, as we already observed before, it suffices to show that there exists $\epsilon'>0$ such that
\begin{equation}\label{eq:gradientexp}
   \Big|\tilde  q_t(x, \bar y) \E_{\br} [ \lambda^N ] - \tilde q_t (x', \bar y) \E_{\brp } [\lambda^N] \Big | \lesssim \frac1{t^{3/2+ \eps'}},
\end{equation}
for $t \in [\delta^{-2}/(\log \delta)^2, \delta^{-2  - \eps} ]$.  
We recall first that if $0 \le u \le t$ and $E \in \cF_u = \sigma(X_0, \ldots, X_u)$, then by the Markov property:
\begin{equation}\label{bridge}
\P_{\br} (E) = \E_x \Big( \mathbf 1_{E} \frac{\tilde q_{t-u} (X_u, \bar y)}{\tilde q_t(x, \bar y)}\Big).
\end{equation}
Let $T_L$ denote the hitting time of the reflection line bisecting $x$ and $x'$; and let $T_\R$ denote the hitting time of $\R$. We introduce the following bad events:
\begin{itemize}
\item $B_1 = \{T_\R > t - s\}$, where $s = [t / (\log  t)^2 ]\wedge [ \delta^{-2}/ (\log \delta)^2] $.
\item $B_2 = \{ T_\R \le t - s\} \cap \{ T_L > T_\R\} \cap \{ N_{t- s/2} \le (\log t)^2 \}$, where $N_u$ is the number of visits to $\R$ by time $u$.
\end{itemize}
We will first show that both events are highly unlikely. In words, $B_1$ is unlikely because it requires going to $\bar y$ in the remaining $s$ units of time starting from above $\R$, which means $\bar y$ is too far away compared to the time remaining. $B_2$ is unlikely because it requires avoiding the reflection line for a long time (until touching $\R$) \emph{and} thereafter making very few visits to $\R$.

\begin{lemma}
  \label{L:B1} For $t \in [\delta^{-2}/(\log \delta)^2, \delta^{-2  - \eps} ]$, we have
  $$
  \P_{\br} (B_1) \lesssim \exp ( - (\log t)^2) \tilde q_t(x, \bar y)^{-1}.
  $$
\end{lemma}
\begin{proof}
Note that by \eqref{bridge},
$$
\P_{\br} (B_1) \le \E_{\br}\Big (\mathbf1_{T_\R > t - s} \frac{\tilde q_s( X_{t-s}, \bar y)}{\tilde q_t (x, \bar y)} \Big ).
$$
Now, $\tilde q_t(x,\bar y)$ satisfies the Gaussian behaviour $\tilde q_t(x, \bar y ) \asymp  (1/t) \exp ( - \frac{| x- \bar y |^2}{2t} )$ in the range $t \geq \delta^{-2}/ (\log \delta)^2$ (see Theorem 2.3.11 in \cite{LawLim}). Since $ | X_{t - s } -\bar y | \gtrsim \delta^{-1}$ when $T_{\R} > t -s$, and since $|x - \bar y| \lesssim \delta^{-1}$, we deduce that for some constant $c>0$,
$$
 {\tilde q_s( X_{t-s}, \bar y)}\le \exp( -c  {\delta^{-2}}/{s} ) \le \exp ( - c(\log 1/\delta)^2)
$$
on the event $T_\R > t -s$, where we used that $s\leq \delta^{-2}/ (\log \delta)^2$. The desired inequality follows since $t\leq \delta^{-2  - \eps}$.
\end{proof}

\begin{lemma}
  \label{L:B2}
For  $t \in [\delta^{-2}/(\log \delta)^2, \delta^{-2  - \eps} ]$, we have
  $$
  \P_{\br} (B_2) \lesssim \frac1{t^{3/2+ \eps'}}(\log t)^6 \tilde q_t(x, \bar y)^{-1},
  $$
where $\eps'=\frac{1-\eps}{2+\eps}$.
\end{lemma}

\begin{proof}
  Using \eqref{bridge},
  \begin{align*}
    \P_{\br} (B_2) & \le \E_x\Big ( \mathbf1_{\{T_\R < t - s, T_\R  < T_L\}}\mathbf1_{\{ N_{t - s/2} \le (\log t)^2 \}} \frac{\tilde q_{s/2} (X_{t - s/2}, \bar y)}{\tilde q_t (x, \bar y)} \Big).
  \end{align*}
  We estimate the off-diagonal heat kernel term $\tilde q_{s/2}(X_{t - s/2}, \bar y)$ by its diagonal behaviour which is at most $\lesssim 1/s = [(\log t)^2/t]\vee [(\log \delta)^2 \delta^2 ]\lesssim (\log t)^2 t^{-\tfrac2{2+\epsilon}}$, where in the last bound we used that $t\leq \delta^{-2-\eps}$ and hence $\delta \leq t^{-\tfrac1{2+\eps}}$.
  Therefore
  \begin{equation}\label{boundB2}
  \P_{\br} (B_2) \lesssim t^{-\tfrac2{2+\epsilon}}\frac{(\log t)^2}{ \tilde q_t (x, \bar y)}\P_x( T_\R < t-s, T_\R< T_L, N_{t-s/2} \le (\log t)^2).
  \end{equation}
  We already know by gambler's ruin estimates that, since $x$ is at distance $O(\delta)$ from $L$ and at distance $\gtrsim 1$ from $\R$ that $\P_x(T_\R < T_L) \le O(\delta)$. Conditioning on everything up to time $T_\R$, and applying the strong Markov property at this time,
  $$
  \P_x( T_\R < t-s, T_\R< T_L, N_{t-s/2} \le (\log t)^2) \lesssim \delta \sup_{z \in \R} \P_z( N_{s/2} \le (\log t)^2).
  $$
  Let $T_i$ denote the length of the intervals between successive visits to the real line. Thus $T_i$ are i.i.d. and $\P( T_i \ge r) \asymp 1/\sqrt{r}$ when $r \to \infty$ by elementary one-dimensional random walk arguments. Fix $z \in \R$. Then by a union bound,
  \begin{align*}
    \P_z( N_{s/2} \le (\log t)^2) & = \P_z \Big ( \sum_{i=1}^{(\log t)^2} T_i \ge s/2 \Big) \\
    & \le \P( T_i \ge s/(2 (\log t)^2) \text{ for some $1\le i\le (\log t)^2$})\\
    & \lesssim (\log t)^3/\sqrt{s} = (\log t)^4/\sqrt{t}.
  \end{align*}
  Therefore, plugging this into \eqref{boundB2}, we find
  $$
  \P_{\br} (B_2)  \lesssim t^{-1/2-\tfrac2{2+\epsilon}} \frac{\delta(\log t)^6}{\tilde q_t (x, \bar y)} \leq    t^{-1/2-\tfrac3{2+\epsilon}} \frac{\delta(\log t)^6}{\tilde q_t (x, \bar y)}= t^{-3/2-\eps'}\frac{(\log t)^6 }{ \tilde q_t(x, \bar y)},
  $$
where we again used that $\delta \leq t^{-\tfrac1{2+\eps}}$.
\end{proof}

Finally, we turn to the remaining contribution. Together with Lemma \ref{L:B1} and Lemma \ref{L:B2}, this shows that \eqref{eq:gradientexp} holds with any $0<\eps <1/2$.
\begin{lemma}
  \label{L:B3}
  \begin{equation}\label{eq:B3}
    \big| \nabla_x \big( \tilde q_t(x , \bar y) \E_{\br} ( \lambda^N ; (B_1 \cup B_2)^c) \big)\big | \le |\lambda|^{(\log t)^2} .
  \end{equation}
where $\lambda = (1- 2p)<1$.
\end{lemma}

\begin{proof} On $(B_1 \cup B_2)^c$, we see that $T_\R < t-s$, and either $T_L < T_\R$ or $N_{t-s/2} \ge (\log t)^2$. In the latter case, $|\lambda^{N}| = |\lambda|^{N_t} \le |\lambda|^{(\log t)^2}$, so this event contributes at most the right hand side of \eqref{eq:B3} to the expectation. To conclude, it therefore suffices to show
\begin{equation}\label{eq:vanishgr}
\tilde q_t(x, \bar y) \E_{\br} (\lambda^N ; T_L < T_\R) = \tilde q_t(x', \bar y)\E_{\brp} ( \lambda^N ; T_L < T_\R)
\end{equation}
so that the contribution of this event to the left hand side of \eqref{eq:B3} vanishes exactly. To see this, let us rewrite the left hand side of \eqref{eq:vanishgr} as an expectation involving random walk rather than bridge, and observe that when $T_L < T_\R$ the walks from $x$ and $x'$ are coupled \emph{before} hitting the real line, so that the overall number of visits to the real line is the same for both walks. Hence
  \begin{align*}
  \tilde q_t(x, \bar y) \E_{\br} ( \lambda^N ; T_L < T_\R)&= \E_x (\lambda^N1_{\{T_L < T_\R\}}1_{\{X_t = \bar y\}})\\
  & = \E_{x'} (\lambda^N1_{\{T_L < T_\R\}}1_{\{X_t = \bar y\}})\\
  & = \tilde q_t(x', \bar y) \E_{\brp} ( \lambda^N ; T_L < T_\R) ,
  \end{align*}
  as desired.
\end{proof}

As explained above, this concludes the proof of Proposition \ref{P:col_half}, and thus also of Theorem \ref{T:PKscaling}.

\section{Convergence to the Neumann Gaussian Free field} \label{S:SL}

From now on we work in the upper-half plane $\H$ with the local (infinite volume) limit $\mu$ (depending on $z$) of the free boundary dimer model from Theorem~\ref{P:infinite_vol_intro}.
We will write $\mu$ to denote both the probability and expectation with respect to $\mu$.

\subsection{Infinite volume coupling function and its scaling limit}

\begin{figure}
\begin{center}
\includegraphics[width=0.45\textwidth]{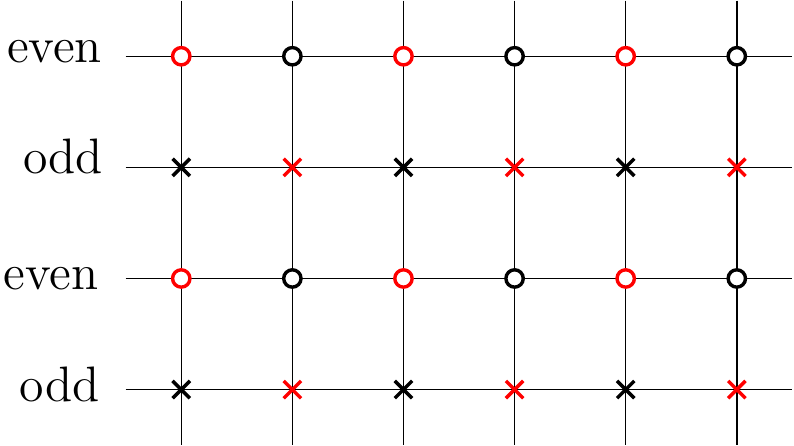}
\caption{Different types of vertices. The black vertices are drawn in red, and the white vertices are drawn in black.}
\label{fig:vertex_type}
\end{center}
\end{figure}

Let $C$ be the \textbf{coupling function}, as defined in Corollary~\ref{C:infinitevol}, i.e., the pointwise limit, as $n\to \infty$, of the inverse Kasteleyn matrix on $\cG_n$ given in matrix notation by
\begin{equation}\label{coupling_master}
  C = -A K^*,
\end{equation}
where $A(x,y)=\frac1{D(y,y)}a(x,y)$ is the normalised potential kernel of the infinite volume bulk effective (nonlazy) walk.
We write $A_{\text{even}}$ and $A_{\text{odd}}$ for the restriction of $A$ to the even and odd rows respectively.
When we unpack \eqref{coupling_master} we find that its meaning is different depending on the respective type of the pair of vertices. We denote the black and white vertices in the even and odd rows by the symbols $\BE$, $\WE$, $\BO$, $\WO$ respectively as illustrated in Figure~\ref{fig:vertex_type}.
Fix $v_1, v_2$ two vertices in $\H \cap \Z^2$. Suppose for instance that $v_1 \in \WE$ and $v_2 \in \BE$. Then \eqref{coupling_master} says
\begin{align*}
C(v_1, v_2) &=
a(v_1,v_2+1)-a (v_1 , v_2-1) \\
& : =  \frac{\delta}{\delta x_2} A_{\text{even}} (v_1, v_2),
\end{align*}
where the second identity follows from
In the following, $\tfrac{\delta }{ \delta x_2}$ (resp.\ $\tfrac{\delta}{\delta y_2}$) will denote the discrete derivative in the $x$ (resp. $y$) direction of the second coordinate of the Green's function. Note that
\[
\frac{\delta}{\delta x_2} \sim 2\delta \frac{\partial}{\partial x_2} \quad \text{as} \quad \delta\to 0.
\]
Likewise, if instead we have $v_1 \in \WE$ and $v_2 \in \BO$, then
$$
C(v_1, v_2) =  i \frac{\delta}{\delta y_2}A_{\text{even}}  (v_1, v_2).
$$
We summarise these computations in a table:
$$
\begin{array}{c|c|c|c|c}
v_1 \backslash v_2 &  v_2 \in \WO & v_2 \in \WE & v_2 \in \BE & v_2 \in \BO \\
\hline
v_1 \in \WO   &  \tfrac{\delta}{\delta x_2} \Aodd  &  i \tfrac{\delta}{\delta y_2} \Aodd &
 i \tfrac{\delta}{\delta y_2} \Aodd&
 \tfrac{\delta}{\delta x_2} \Aodd  \\
\hline
v_1 \in   \WE &  i \tfrac{\delta}{\delta y_2} \Aev &   \tfrac{\delta}{\delta x_2} \Aev &  \tfrac{\delta}{\delta x_2}  \Aev   &   i \tfrac{\delta}{\delta y_2} \Aev  \\
\end{array}
$$
Furthermore, when $v_1$ is in the black lattice ($v_1 \in \BE$ or $v_1 \in \BO$) we obtain the corresponding table simply by translation invariance:
$$
\begin{array}{c|c|c|c|c}
v_1 \backslash v_2 &  v_2 \in \WO & v_2 \in \WE & v_2 \in \BE & v_2 \in \BO \\
\hline
v_1 \in \BO & \tfrac{\delta}{\delta x_2} \Aodd & i \tfrac{\delta}{\delta y_2} \Aodd &    i\tfrac{\delta}{\delta y_2} \Aodd & \tfrac{\delta}{\delta x_2} \Aodd \\
\hline
v_1 \in \BE & i \tfrac{\delta}{\delta y_2} \Aev & \tfrac{\delta}{\delta x_2} \Aev   &   \tfrac{\delta}{\delta x_2} \Aev &i \tfrac{\delta}{\delta y_2}  \Aev \\
\end{array}
$$

\begin{remark}
It is useful to point out that the terms involving mixed colours and those involving matching colours behave very differently: indeed, if both $v_1$ and $v_2$ are of the same colour, then the arguments of the corresponding potential kernel are of different colours. This corresponds to only considering walks that go through the boundary in the definition of the potential kernel (see Corollary~\ref{C:PKscaling}).
\end{remark}

There is a convenient algebraic rewriting of these different values. Suppose $v_1$ and $v_2$ are two {arbitrary} vertices (of any colour), and let
$$
s(v) = (-1)^{\text{row $\#$ of $v$}}
$$
be the signed parity of the row of $v$.
Then we have
\begin{align}
C(v_1, v_2) & = \frac14 \left[ \left( {1+ s(v_1) s(v_2)} \right) \frac{\delta}{\delta x_2} + \left( {1 - s(v_1) s(v_2)} \right) i \frac{\delta}{\delta y_2}\right] \nonumber \\
&  \times \left[ \left ( {1 - s(v_1)}\right) \Aodd(v_1,v_2) + \left( {1+ s(v_1)} \right) \Aev(v_1,v_2) \right]\label{masterCgreen}.
\end{align}

We will now combine \eqref{masterCgreen} with Corollary~\ref{C:PKscaling} to obtain the scaling limit of the inverse Kasteleyn matrix in the upper half-plane.
\begin{thm}\label{T:couplingscaling}
Let $z$ and $w$ be two vertices on $\delta \Z^2 \cap \H$, and fix $\rho>0$. Then there exists $\eps>0$ such that uniformly over $z \neq w$ with $ \min (\Im (z), \Im (w) ) \ge \rho$, as the mesh size $\delta \to 0$,
   $$
  C(z,w)
 = \begin{cases}
   -\dfrac{\delta}{2 \pi} \left( s(z) s(w)  \dfrac1{z-w} +  \dfrac{1}{\bar z - \bar w} \right)  + o(\delta^{1+ \eps}) + O( \frac{\delta}{|z-w|})^{2}& \text{ if $z,w$ are of different class}\\
   \dfrac{\delta}{2\pi}  \left(s(z) \dfrac1{z -\bar w} +  s(w) \dfrac1{ \bar z - w} \right) + o(\delta^{1+ \eps}) & \text{ if $z,w$ are of the same class}.
 \end{cases}
  $$
\end{thm}
\begin{proof}
We could use the master formula \eqref{masterCgreen} but in order to avoid making mistakes it is perhaps easier to consider all the possible cases for the types of vertices $z$ and $w$ using the tables above.
We start with the case when $z$ and $w$ are of different colour. We will use the symmetry of the potential kernel $a(z,w)=a(z,w)$ and Corollary~\ref{C:PKscaling}
(applied to the case when the arguments of the potential kernel are of the same colour).
\begin{itemize}
\item For $z, w$ with $s(z)=s(w)=-1$, we have
\begin{align*}
C(z, w) &= \frac{\delta}{\delta x_2} \Aodd(z,w) \\
&= \frac14(a(w+\delta ,z)- a(w-\delta ,z))
\\&=  \frac14\times \frac2\pi \Re\left ( \frac {2\delta} {w-z}\right) +o(\delta^{1+\eps})+ O\left(\frac{\delta}{|z-w|}\right)^2\\
 &= -\frac\delta\pi \Re\left ( \frac {1} {z-w}\right) +o(\delta^{1+\eps}) + O\left(\frac{\delta}{|z-w|}\right)^2.
\end{align*}
The factor $1/4$ comes from the fact that $\Aodd$ is normalised by the degree of $w$ which is equal to $4$ (see~\eqref{godd}).

\item For $z, w$ with $s(z)=-1$ and $s(w)=1$, we have
\begin{align*}
C(z,w)&=i\frac{\delta}{\delta y_2} \Aodd(z,w) \\
&=-i\frac\delta\pi \Re\left ( \frac {i} {z-w}\right) +o(\delta^{1+\eps}) + O\left(\frac{\delta}{|z-w|}\right)^2
\\& =i\frac\delta\pi \Im\left ( \frac {1} {z-w}\right) +o(\delta^{1+\eps}) + O\left(\frac{\delta}{|z-w|} \right)^2.
\end{align*}

\item
For $z, w$ with $s(z) =s(w)=1$, since $z, w$ are of different colors, $z$ and $w\pm\delta$ are of the same colour.
Note that $\Gev$ is a signed function such that $\Gev(z,w)<0$ for $z,w$ of different colors, and $\Gev(z,w)>0$ for $z,w$ of the same colour. Therefore we have
\begin{align*}
C(z,w)=\frac{\delta}{\delta x_2} \Aev(z,w) =-\frac\delta\pi \Re\left ( \frac {1} {z-w}\right) + o(\delta^{1+\eps}) + O\left(\frac{\delta}{|z-w|} \right)^2.
\end{align*}

\item For $z, w$ with $s(z) =1$ and $s(w)=-1$, $z$ and $w\pm\delta i$ are again of the same colour. We have
\begin{align*}
C(z,w) &= i\frac{\delta}{\delta y_2} \Aev(z,w) \\
& =-i \frac\delta\pi \Re\left ( \frac {i} {z-w}\right) + o(\delta^{1+\eps}) + O\left(\frac{\delta}{|z-w|} \right)^2 \\
&=i\frac\delta\pi \Im\left ( \frac {1} {z-w}\right) + o(\delta^{1+\eps}) + O\left(\frac{\delta}{|z-w|}\right)^2.
\end{align*}
\end{itemize}

Let us now consider the case where $z$ and $w$ are of different colors, by applying Corollary~\ref{C:PKscaling} (when the arguments are of different colors).
\begin{itemize}
\item For $z, w$ with $s(z)=s(w)=-1$, we have
\[
C(z,w) = \frac{\delta}{\delta x_2} \Aodd(z,w) = -\frac\delta\pi \Re\left ( \frac {1} {\bar z-w}\right) + o(\delta^{1+\eps}),
\]

\item For $z,w$ with $s(z)=-1$ and $s(w)=1$, we have
\begin{align*}
C(z,w) = \frac{\delta}{\delta y_2} \Aodd(z,w) =\frac\delta\pi \Im\left ( \frac {1} {\bar z-w}\right) + o(\delta^{1+\eps}).
\end{align*}

\item For $z, w$ with $s(z) =s(w)=1$, since $z, w$ are of the same colour, $z$ and $w\pm\delta$ are of different colors. We have
\[
C(z,w) = \frac{\delta}{\delta x_2} \Aev(z,w) = \frac\delta\pi \Re\left ( \frac {1} {\bar z-w}\right) + o(\delta^{1+\eps}).
\]
\item For $z, w$ with $s(z) =1$ and $s(w)=-1$, $z$ and $w\pm\delta i$ are again of different colors. We have
\begin{align*}
C(z,w)= i \frac{\delta}{\delta y_2} \Aodd(z,w) =-i\frac\delta\pi \Im\left ( \frac {1} {\bar z-w}\right) + o(\delta^{1+\eps}).
\end{align*}
\end{itemize}
Combined, we have proved the theorem.
\end{proof}

\subsection{Pfaffians and Kasteleyn theory} \label{S:KT}
In this section we recall basics of Kastelyn theory. In particular we will express local statistics of $\mu$ in terms of the coupling function $C$.

Let $A$ be $2k \times 2k$ matrix indexed by vertices $w_1, b_1, \ldots, w_k, b_k$ of $k$ edges $(w_1, b_1), \ldots, (w_k, b_k)$. Then a Pfaffian can be expressed as a sum over matchings of these $2k $ vertices, in a similar way as the determinant can be expressed as a sum over permutations. Let $M$ be such a perfect matching. We can write it as $(i_1, j_1) ,\ldots, (i_k, j_k)$ where:

\begin{itemize}
\item in each pair $(i,j)$ the $i$-vertex comes from an edge $(w_i, b_i)$ that is listed before the edge $(w_j, b_j)$ which contains the $j$-vertex. If the two vertices belong to the same edge,
then a white vertex comes before a black vertex,
\item we require that $i_1< \ldots< i_k$.
\end{itemize}
This defines a permutation
$$
\pi_M = \left[
\begin{array}{ccccccc}
1 & 2 & 3 &4 & \ldots & 2k-1 & 2k \\
i_1 & j_1 & i_2 & j_2 & \ldots & i_{k} & j_k
\end{array}
\right].
$$
Then we have
\begin{align} \label{eq:Pfmatching}
\Pf (A) = \sum_{M \text{ matching}} \text{sgn} (\pi_M) a_{i_1, j_1} \ldots a_{i_k, j_k}.
\end{align}
Based on Kasteleyn's theorem Kenyon derived the following description of local statistics for the dimer model~\cite{Kenyon97}.
Recall that $\mu$ denotes the probability measure of the infinite volume free boundary dimer model on $\H \cap \Z^2$. 
\begin{thm} \label{thm:Kasteleyn}
Let $E$ be a set of pairwise distinct edges $e_1 = (w_1, b_1) , \ldots, e_k = (w_k, b_k)$ with the convention that the white vertex comes first. Then
$$
\mu(e_1, \ldots, e_k \in \cM) = a_E \Pf (C),
$$
where $\cM$ is the random monomer-dimer configuration under $\mu$, and where $C=C(v_1,v_2)$ is the coupling function restricted to the vertices $v_1, v_2 \in \{w_1, \ldots, w_k\} \cup \{b_1, \ldots, b_k\}$ (implicit here is the fact that the vertices are ordered from black to white, and from $1$ to $k$), and where
$$
a_E = \prod_{i=1}^k K(w_i, b_i)
$$
is the product of the Kasteleyn weight of each edge, oriented from white to black.
\end{thm}

To compute the scaling limit of the height function, we will need to study the centered dimer-dimer {correlations}. When expanding the Pfaffian into matchings, this leads to a simplification which is the analogue of Lemma 21 in \cite{Kenyon_ci}:
\begin{lemma}\label{L:centeredPfaffian}
In the setting as above,  we have
  \begin{align*}
  \mu[ (\mathbf 1_{\{e_1 \in \cM\}} - \mu( e_1 \in \cM) )& \ldots (\mathbf 1_{\{e_k \in \cM\}} - \mu( e_k \in \cM) ) ] \\
  & = a_E \sum_{M \text{restricted matching}} \sgn(\pi_M) \mathop{\prod_{\{ u,v\} \in M}}_{u<v}  C(u,v)
  \end{align*}
  where a \textbf{restricted matching} is a matching $M$ such that $w_i$ cannot be matched to $b_i$ for any $1\le i \le k$,
  and where $u<v$ means that $u$ comes before $v$ in the fixed order on vertices.
\end{lemma}
\begin{proof}
Let $M$ be a matching of the vertices of $e_1, \ldots, e_k$.
We will call a pair of matched vertices an $M$-edge to distinguish it from the edges of the underlying graph.
We mark the vertices $w_1,b_1,\ldots,w_k,b_k$ (see Figure~\ref{fig:graphicalrep}) in the order from left to right on the real line.
For each $M$-edge, draw an arc (a simple continuous curve) in the upper half-plane connecting the vertices matched by this $M$-edge (see Figure\ \ref{fig:graphicalrep}). Moreover, draw the arcs in such a way that any two arcs cross at most once.
A standard result that can be checked by induction says that
\begin{align} \label{eq:signcrossing}
\sgn(\pi_M)=(-1)^{\#\text{ arc crossings}}.
\end{align}
Note that an arc connecting $b_i$ to $w_i$ does not cross any other arc.
Using this and Theorem~\ref{thm:Kasteleyn}, we can write
\begin{align*}
 \mu [ (1_{\{e_1 \in \cM\}} - \mu( e_1 \in \cM) )& \ldots ( 1_{\{e_k \in \cM\}} - \mu( e_k \in \cM) ) ] \\
 &= \sum_{E'\subseteq E} (-1)^{|E\setminus E'|}\mu(e\in \cM \text{ for all } e \in E' ) \prod_{e\in E\setminus E'} \mu(e \in \cM) \\
 & = a_E\sum_{E'\subseteq E} (-1)^{|E\setminus E'|} \sum_{M \in \Pi( E')} \sgn(\pi_M)\mathop{\prod_{\{ u,v\} \in M}}_{u<v}  C(u,v) \mathop{\prod_{\{ u,v\} \in E\setminus E'}}_{u<v}C(u,v)\\
  & = a_E\sum_{E'\subseteq E} (-1)^{|E\setminus E'|}\mathop{ \sum_{M \in \Pi( E)}}_{E\setminus E'\subset M} \sgn(\pi_M)\mathop{\prod_{\{ u,v\} \in M}}_{u<v}  C(u,v) \\
    & = a_E \sum_{M \in \Pi( E)} \sgn(\pi_M)\mathop{\prod_{\{ u,v\} \in M}}_{u<v}  C(u,v)\Big (\sum_{E'\subseteq E\cap M} (-1)^{|E'|}\Big)
  \end{align*}
where $\Pi( E')$ is the set of matchings of the vertices of $E'$. To finish the proof it is enough to notice that the sum of signs in the last expression is equal to one if $M\cap E=\emptyset$ and
it vanishes otherwise.
\end{proof}

\subsection{Matchings and permutations}
In this section we discuss the combinatorics of matchings and permutations which will be used in the computation of moments of the height function in Section~\ref{P:infinite_vol_intro}.

Let $M$ be a restricted matching of the vertices of $e_1, \ldots, e_k$ (recall that restricted means that the endpoints of an edge cannot be matched with each other).
We stress the fact that the objects paired by the matching are the vertices of the edges and not the edges themselves. This will be important in the following combinatorial considerations.
We will call a pair of matched vertices an $M$-edge to distinguish it from the edges of the underlying graph.

We can turn a matching $M$ into a \textbf{directed matching} {$\m$} by assigning to each $M$-edge a direction in such a way that each edge $e_i$ has exactly one outgoing and one incoming $M$-edge.
Let $\mathfrak{S}_k^*$ be the set of permutations on $k$ elements with no fixed points.
Observe that a directed restricted matching {$\m$} defines a permutation $\sigma \in \mathfrak{S}_k^*$ of the $k$ edges: indeed,
simply define $\sigma(e_i)=e_j$ where $e_j$ is the edge pointed to by the unique outgoing $M$-edge emanating from~$e_i$.
We will say that the directed matching is \textbf{compatible} with the permutation $\sigma$. Note that since the matching is restricted, $\sigma$ does not have fixed points.
Let $\mathsf{DM}_\sigma$ be the class of restricted directed matchings compatible with $\sigma \in \mathfrak{S}_k^*$. Note that if $\sigma$ has $n=n(\sigma)$ cycles and $\m\in \mathsf{DM}_\sigma$, then there are $2^n$ oriented matchings that correspond to the same unoriented matching as $\m$ (one can choose the orientation of each cycle independently of the choice for other cycles).

Fix $\sigma \in \mathfrak{S}_k^*$. We now describe how to encode a directed matching $\m$ compatible with $\sigma$  by a sequence of signs $\nu\in \{-1, 1\}^k$.
The sign $\nu_i$ denotes the choice of the vertex of $e_i$ from which the outgoing edge of $\m$ will emanate, i.e., if
$\nu_i =+1$ (resp.\ $-1$) then the outgoing edge of $e_i$ emanates from the black (resp.\ white) vertex of $e_i$.
This choice implies that the directed $M$-edge corresponding to the pair $(i,j)$ such that $\sigma(i)=j$ points to the white (resp.\ black) vertex of $e_j$ if $\nu_j=+1$ (resp.\ $\nu_j=-1$).
The resulting map
\begin{equation}\label{bij_DM}
\{-1, 1\}^k\to \mathsf{DM}_\sigma
\end{equation}
is clearly a bijection.

We can now rewrite the truncated correlation function from Lemma \ref{L:centeredPfaffian} as follows:
  \begin{align}
  & \mu \left[ (\mathbf 1_{\{e_1 \in \cM\}} - \mu( e_1 \in \cM) ) \ldots ( \mathbf 1_{\{e_k \in \cM\}} - \mu( e_k \in \cM) ) \right] \nonumber \\
  & = a_E \sum_{\sigma \in \mathfrak{S}_k^* }\sum_{M  \in \mathsf{DM}_\sigma} \sgn(\pi_M) \frac1{2^n}  \prod_{(u,v) \in M} C(u,v) (-1)^{\mathbf1_{{u >v}}}
  \label{correl_OM}
  \end{align}
where $n=n(\sigma)$ is, as above, the number of cycles of $\sigma$. To explain \eqref{correl_OM}, we simply recall that each undirected matching that can be oriented as to be compatible with $\sigma$ corresponds to $2^n$ directed matchings by choosing the orientation of each cycle arbitrarily.
The factor $(-1)^{\mathbf 1_{{u >v}}}$ comes from the fact that $C(u,v)$ is antisymmetric and
that we always have $i_l < j_l$ in the expansion of the Pfaffian as a sum over matchings \eqref{eq:Pfmatching}.
Here $u >v$ means that $u$ comes later than $v$ in the order defined by $w_1,b_1,\ldots,w_k,b_k$.

We will later need the following lemma. What is specifically interesting to us in the expression below is that the right hand side depends very little on the permutation $\sigma$, given the signs $\nu$.

\begin{lemma} \label{lem:signinversions}
Let {$\m$} be the restricted directed matching compatible with $\sigma\in \mathfrak{S}^*_k$ and encoded by $\nu \in \{-1,1 \}^k$ by the map \eqref{bij_DM}. We have
\begin{align} \label{eq:signcomparison}
 \prod_{(u,v) \in \m} (-1)^{\mathbf 1_{{u >v}}} \sgn(\pi_M) = (-1)^n \prod_{i=1}^k \nu_i,
\end{align}
 where $n=n(\sigma)$ is the (total) number of cycles in $\sigma$.
 \end{lemma}
\begin{proof}
Mark the vertices $w_1,b_1,\ldots,w_k,b_k$ in the order from left to right on the real line as in Lemma~\ref{L:centeredPfaffian} and recall formula~\eqref{eq:signcrossing}.
\begin{figure}
\begin{center}
\includegraphics{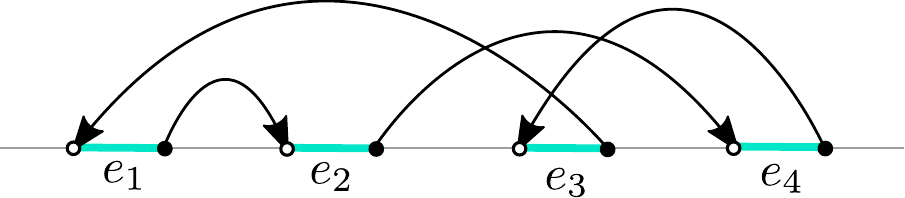}
\caption{Graphical representation of the directed matching $M=\{(b_1,w_2), (b_2,w_4), (b_3, w_1), (b_4,w_3) \}$ corresponding to the cyclic permutation $\sigma=1\to2\to4\to3\to1$ and signs $\nu=(1,1,1,1)$.
We have $\sgn(\pi_M)=-1$ since the number of arc crossings is odd}
\label{fig:graphicalrep}
\end{center}
\end{figure}
Note that if we flip exactly one sign $\nu_i$, then both sides of \eqref{eq:signcomparison} change sign since the parity of the number of
crossings between arcs changes (we either cross or uncross the arcs ending at~$e_i$ and we do not change the number of crossings for other pairs of arcs), and since the number of \emph{decreasing edges} $(u,v)$ of $\m$, i.e., satisfying $u>v$, does not change.
We can hence assume that $\nu_i = +1$ for all $i$.

\begin{figure}
\begin{center}
\includegraphics[scale=0.9]{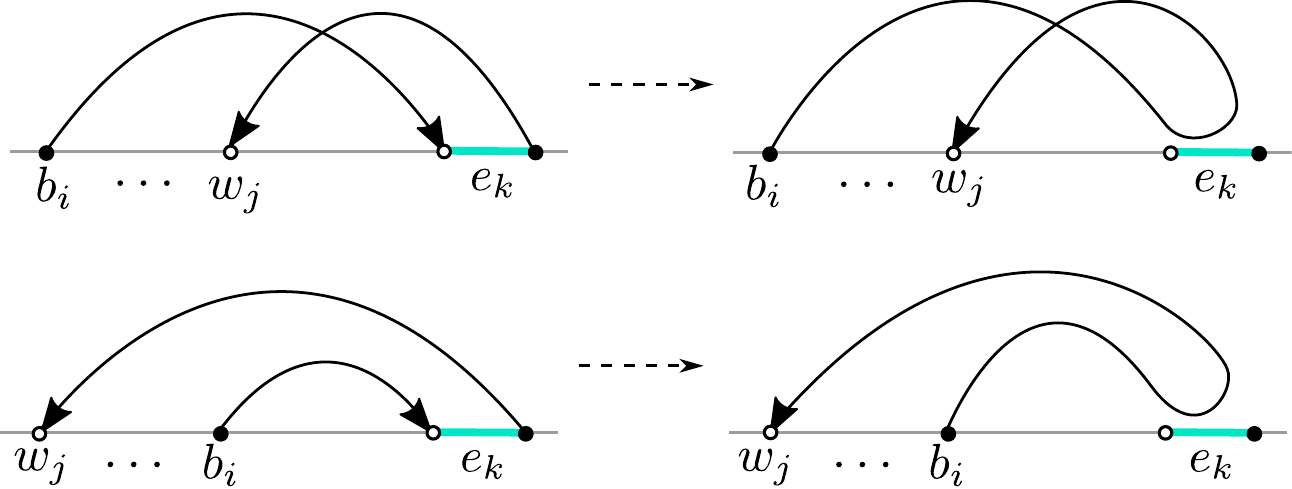}
\caption{The induction step from the proof of Lemma~\ref{lem:signinversions}}
\label{fig:crossingsign}
\end{center}
\end{figure}

Equipped with the graphical representation as in Lemma~\ref{L:centeredPfaffian} we proceed by induction on $k$. One can check that the statement is true for $k=2$.
We therefore assume that $k>2$.
Let $\m$ be a directed restricted matching on $e_1,\ldots,e_{k}$, and let $\sigma\in \mathfrak{S}_{k}^*$ be the permutation associated with $M$.
Let $i,j$ be such that $\sigma(i)=k$ and $\sigma(k)=j$. Consider a graphical representation of {$\m$}.
Imagine infinitesimally deforming the path composed of the arcs connecting $e_i$ to $e_{k}$ and $e_{k}$ to $e_j$
together with the line segment representing $e_{k}$ in such a way that the path is fully contained in $\H$.
This path hence becomes an arc (modulo a possible self-crossing) representing an {$\m'$}-edge $(b_i,w_j)$, where {$\m'$} is a directed restricted matching on $e_1,\ldots,e_{k-1}$.
Let $\sigma'\in \mathfrak{S}_{k-1}$ be the permutation associated to {$\m'$}. Note that $\sigma'$ has the same number of cycles as~$\sigma$.

In this transformation we replaced an increasing edge $(b_i,w_{k})$ and a decreasing edge $(b_{k},w_j)$ by the edge $(b_i,w_j)$.
For topological reasons, the deformed path representing the {$\m'$}-edge $(b_i,w_j)$ has a self-crossing if and only if $(b_i,w_j)$ is an increasing edge (see Figure~\ref{fig:crossingsign}).
To finish the proof we use \eqref{eq:signcrossing} to evaluate and compare \eqref{eq:signcomparison} for {$\m$} and {$\m'$}, and we use the induction assumption.
\end{proof}


\subsection{Moments of the height function} \label{S:moments}
In this section we compute the scaling limit of the pointwise moments of the height function on $\delta \Z^2 \cap \H$, which is the penultimate step in establishing its convergence as a random distribution.

We fix $k \ge 1$, and $2k$ faces $a_1, b_1, \ldots, a_k, b_k$ of $\delta \Z^2 \cap \H$. We consider disjoint paths $\gamma_i$ in the dual lattice $(\delta \Z^2 \cap \H)^*$ connecting $a_i$ to $b_i$ for $1\leq i\leq k$. The following is the analogue of Proposition 20 in \cite{Kenyon_ci}. Let $D$ denote the minimal distance in the complex plane between any pair of points within $\{a_i,b_i\}_{1\le i \le k}$.

\begin{prop} \label{P:kpoint_dimer} Let $k \ge 1$. Let $\rho >0$ be fixed and let $\beta>0$ be sufficiently small (possibly depending on $k$). As $\delta \to 0$,
  \begin{align}
  & \Big| \mu\left[ (h^{\delta}(a_1)- h^{\delta}(b_1))  \cdots(h^{\delta}(a_k)- h^{\delta}(b_k))  \right]  - \\
  &\sum_{\m \in \cM(1 , \ldots, k)} \prod_{(i,j) \in \m}     -\frac{1}{2\pi^2} \Re \log \frac{(a_i - a_j)(b_i - b_j)(\bar a_i - a_j) (\bar b_i - b_j)}{ (a_i - b_j)(b_i - a_j)(\bar a_i - b_j)(\bar b_i - a_j)} \Big| \to 0,
  \end{align}
   uniformly over the choice of $a_1, b_1, \ldots, a_k, b_k$ such that $ D \ge \delta^\beta$ and $\min_{1\le i \le k} ( \Im (a_i), \Im (b_i)) \ge \rho$.
\end{prop}

\begin{proof}
  As in Kenyon \cite{Kenyon_ci} we can assume without loss of generality that the paths $\gamma_i$ are piecewise parallel to the axes and that each straight portion is of even length. In this way, we can pair the edges of a straight portion of the path in groups of two consecutive edges. In order to distinguish between the two edges in a given pair it will be useful to have a notation which emphasises this difference, and following the notations of Kenyon we will call a generic pair of edges $\alpha $ and $\beta$ respectively; an $\alpha$-edge will have a black vertex on the right while a $\beta$-edge will have a black vertex on its left. The point is that considering their contributions together will lead to cancellations that are crucial in the computation. Also, in this way the contribution from a pair of edges does not depend anymore on the microscopic types of its vertices and has a scaling limit which depends only the macroscopic position.

  Let $\alpha^{i}_{t}$ (resp. $\beta_t^i$) be the indicator that the $t$-th $\alpha$-edge (resp. $\beta$-edge) in the path $\gamma_i$ is present in the dimer cover, minus its expectation. In this way due to the definition of the height function and the choice of reference flow,
  $$
  h(a_i) - h(b_i) =  \sum_t \alpha^i_t  - \beta_t^i.
  $$
  (Note we do not have a factor 4 as in Kenyon because our choice of reference is slightly different in order to deal directly with a centered height function: more precisely, the total flow out of a vertex is one instead of four in Kenyon's work \cite{Kenyon_ci}).
  We are ignoring here possibly one term on the boundary if the faces $a_i$ and $b_i$ do not have the correct parity; but in any case it is clear that the contribution of a single term in such a sum is of order $O(\delta)$ and so can be ignored in what follows.

  We therefore have
  \begin{equation}\label{E:alphabeta}
  \mu[ (h^{\delta}(a_1)- h^{\delta}(b_1)) \cdots(h^{\delta}(a_k)- h^{\delta}(b_k))  ] =  \sum_{t_1, \ldots, t_k}  \mu [ (\alpha^1_{t_1}  - \beta^1_{t_1}) \cdots (\alpha^k_{t_k} - \beta^k_{t_k}) ].
  \end{equation}
  We fix a choice of $t_i$s and analyse this product. We first expand this product into a sum of $2^k$ terms containing for each $i$ a term which is either $\alpha_{t_i}^i$ or $- \beta^i_{t_i}$. Consider for simplicity the term containing all of the $\alpha^i_{t_i}$. Write $w_i, b_i$ for the white and black vertices of the edge corresponding to $\alpha^i_{t_i}$. Let $E$ be the set of edges $(w_1,b_2),\ldots,(w_k,b_k)$ and let
$a_E = \prod_{e \in E} K(e)$. Then by \eqref{correl_OM} we have
  \begin{align*}
    \mu( \alpha^1_{t_1} \ldots \alpha^k_{t_k} )  & = a_E \sum_{\sigma \in \mathfrak{S}^*_k} \sum_{\m \in \mathsf{DM}_\sigma} \sgn(\m) \frac1{2^n} \prod_{(u,v) \in \m }C(u, v) (-1)^{1_{u>v}}.
  \end{align*}
  We rewrite the sum over directed matchings $\m \in \mathsf{DM}_\sigma$ as a sum over $(\nu_i)_{1\le i \le k}$ using \eqref{bij_DM}, and get (writing $\m$ for the unique directed matching determined by $\sigma$ and $\nu = (\nu_i)_{1\le i \le k} \in \{ - 1, 1\}^k$),
  \begin{align*}
    &  a_E \sum_{\nu } \sum_{\sigma \in \mathfrak{S}_k^*} \sgn(\m) \frac1{2^n} \prod_{(u,v) \in \m }C(u, v) (-1)^{1_{u>v}} = a_E \sum_{\nu} (\prod_{i=1}^k \nu_i) \sum_{\sigma \in \mathfrak{S}_k^*} (-1)^n \frac1{2^n} \prod_{(u,v) \in \m} C(u,v)
  \end{align*}
  using Lemma \ref{lem:signinversions}.

  Fix $\nu$ and $\sigma$ (i.e., we fix a directed matching $\m$) and use Theorem \ref{T:couplingscaling} to approximate $C(u,v)$. Let $(u,v) \in \m$ and let $\nu$ and $\nu'$ be the respective values of the variables $\nu_i$ associated with the two edges containing the vertices $u$ and $v$. Note that if $\nu \nu' = 1$ then $u$ and $v$ must be of different colours and so we fall in case 2 of the approximation given by Theorem \ref{T:couplingscaling}, while if $\nu \nu' = -1$ then $u$ and $v$ are of the same colour and so we fall in the first case of this approximation. Hence we get
    $$
  C(u,v)
 = \begin{cases}
   - \frac{\delta}{2\pi} \left[ s(u) s(v)  \frac1{u-v} + \frac{1}{\bar u - \bar v} \right]  + o(\delta^{1+ \eps}) + O( \frac{\delta}{|u-v|})^{2}& \text{ if } \nu \nu' = 1 \\
   \frac{\delta}{2\pi }\left[ s(u) \frac1{u - \bar v} + s(v) \frac1{\bar u - v} \right] + o(\delta^{1+ \eps}) & \text{ if } \nu \nu' = -1.
 \end{cases}
  $$
The terms $o(\delta^{1+\eps})$ here are uniform on $u,v$ (subject only to the imaginary parts being $\ge \rho$). Since $|u - v | \ge D \ge \delta^\beta$, we see that $O( \frac{\delta}{|u-v|})^{2} \le O(\delta^{2- 2 \beta}) = o ( \delta^{1+ \eps})$ if $\beta$ is sufficiently small. We can thus absorb the term $O( \frac{\delta}{|u-v|})^{2} $ into the term $o(\delta^{1+\eps})$ under our assumptions on $D$.

We expand $\prod_{(u, v) \in \m} C(u,v)$ using the above formula. This gives us another sum of $2^k$ terms, which we view as a polynomial in the variables $s(w_i), s(b_i)$. We group the terms by their monomials; and since $s(z)^2 = 1$ for any $z$, these monomials can only be of degree at most one in each variable. We now claim that any monomial such that $s(b_i)$ appears but not $s(w_i)$ for some $1\le i \le k$, or vice-versa, will contribute $o(\delta)$ when we take into account the equivalent term coming from the same expansion where $\alpha^i_{t_i}$ has been replaced by $- \beta^i_{t_i}$. Indeed, since $\sigma$ and $\nu$ have been fixed, consider what happens when $\alpha^i_{t_i}$ is replaced by $ - \beta^i_{t_i}$:
  \begin{itemize}
    \item There is a $-$ sign coming from the change $\alpha^i_{t_i} \to - \beta^i_{t_i}$.

    \item The sign of $a_E$ changes by $-1$ always (consider separately the cases of a horizontal or vertical edge to see this).

    \item Crucially both $s(b_i)$ \textbf{and} $s(w_i)$ change.

    \item Yet the coefficients accompanying $s(b_i)$ and $s(w_i)$ (both of which are terms of the form $\tfrac1{z - w} + o(\delta), \ldots $ or $ \tfrac1{\bar z - \bar w} + o(\delta)$) do not change in the scaling limit, since this term is determined only by the choice of $\nu$, which is fixed.
  \end{itemize}

As a consequence, as we sum over all choices of $\alpha$ and $\beta$ in the $2^k$ terms of \eqref{E:alphabeta}, and we expand in terms of monomials as described above, we only keep terms that contain for each $1 \le i \le k$ either, simultaneously $s(b_i)$ and $s(w_i)$, or neither of them.

As it turns out, given $\sigma $ and $\nu$, only very few terms do not cancel out. In fact, for each cycle of $\sigma$ there will be only two terms. For example, consider the case $k = 4$, $\sigma = (1234)$ a four-cycle, and $\nu = +-++$. This means we are expanding
$$
C(b_1, b_2) C(w_2, w_3) C(b_3, w_4) C(b_4, w_1).
$$
Letting $z_i$ be the point in the middle of the edge $(b_i, w_i)$, the expansion looks like
\begin{align*}
& \frac{\delta}{2\pi }\left [s(b_1) \frac1{z_1 - \bar z_2} + s(b_2) \frac1{\bar z_1 - z_2 } \right ] \\
\times & \frac{\delta}{2\pi } \left [s(w_2) \frac1{z_2 - \bar z_3} + s(w_3) \frac1{\bar z_2 - z_3} \right ]\\
\times & (- \frac{\delta}{2\pi}) \left [s(b_3) s(w_4) \frac1{z_3 - z_4} +\frac1{\bar z_3 - \bar z_4} \right ]\\
\times & (-\frac{\delta}{2\pi})\left [s(b_4) s(w_1) \frac1{z_4 - z_1} + \frac1{\bar z_4 - \bar z_1} \right]  + o( \delta^{4+ \eps}/D^4).
\end{align*}
The only terms that survive this expansion with the above requirements are the monomials corresponding to $s(b_1) s(w_1) s(b_3)s(w_3) s(b_4) s(w_4)$ and $s(b_2) s(w_2)$: indeed, choosing or not the term containing $s(b_1)$ in the first line imposes a choice on every other line, which is why just two terms survive this expansion.

Furthermore, crucially, in the corresponding coefficients of the surviving monomials, the variables $z_i$ or $\bar z_i$ occurs exactly twice, either twice in the type $z_i$ or twice in the type $\bar z_i$ (but never in a mixed fashion). For instance, in the above example, the coefficient will involve either $z_1, \bar z_2, z_3, z_4$ or the other way around: $\bar z_1, z_2, \bar z_3, \bar z_4$. Note that the dependence on $z_i$ or $\bar z_i$ is consistent with the choice of signs coming from $\nu$: more precisely, for $z \in \C $ and $\eps = \pm 1$, define $z^\eps $ to be $z$ if $\eps = +1$ and $\bar z$ if $\eps = -1$. Then for a cyclic permutation $\sigma$, the two monomials which survive the expansion have a coefficient proportional to
$$
\prod_{i=1}^k \frac1{z_i^{\nu_i} - z_{\sigma(i)}^{\nu_{\sigma(i)}}} \quad \quad \text{ and } \quad \quad
\prod_{i=1}^k \frac1{\overline{z_i^{ \nu_i}} - \overline{z_{\sigma(i)}^{ \nu_{\sigma(i)}}}}
$$
 and a similar property holds for a general permutation $\sigma$ by considering each of its cycles separately.

 Note furthermore that each such coefficient comes with a factor $\pm (\delta/2\pi)^k \times 2^k$: indeed, when a monomial survives it arises exactly once in each of the $2^k$ terms from the $\alpha-\beta$ expansion \eqref{E:alphabeta}. The sign itself is determined purely by the parity of the cycle of the permutation: indeed, for an even length cycle the number of times the colour changes as we follow the directed matching must be even; while it must be odd for an odd length cycle.

 Suppose $C = \{c_1, \ldots, c_n\}$ is the cycle structure of $\sigma$. We will use variables $(\eps_c)_{c \in C} \in \{-1, 1\}^{n}$ to denote which type of monomials we consider. Thus the right hand side of \eqref{E:alphabeta} (still for a fixed choice of $t_i$'s) becomes
 \begin{equation}\label{E:sumpermut}
 = a \delta^k \sum_{\nu} (\prod_{i=1}^k \nu_i)\sum_{\sigma \in \mathfrak{S}_k^*} (-1)^n \frac1{2^n} \sum_{\eps} \prod_{i=1}^k \frac1{\pi} \frac1{z_i^{\nu_i \eps_{c(i)}} - z_{\sigma(i)}^{\nu_{\sigma(i)} \eps_{c(\sigma(i))}} }[s(b_i) s(w_i)]^{(1 +\eps_{c(i)} \nu_i)/2 } + o( \delta^{k+\eps}/D^k)
 \end{equation}
 where $c(i)$ is the cycle containing $i$.

 We now claim that if $\sigma \in \mathfrak{S}_k^*$ has any cycle $c$ of length $|c| >2$ then it contributes zero to the sum. We start by considering odd cycles. Indeed consider the case where $k$ is odd and $\sigma$ is a cyclic permutation of length $k$. Then apply the bijection $\nu \to - \nu$ and $\eps \to - \eps$ to find that all the terms are unchanged except for a negative sign coming from $\prod_{i=1}^k \nu_i$. Hence this contribution must be equal to zero, and a similar argument can easily be made when $\sigma$ contains a cycle of odd length.

 In particular, $k$ itself must be even for the contribution to be nonzero. To get rid of permutations containing cycles of even length $>2$, we will rely on the following lemma.

 \begin{lemma}\label{L:k_cancel} Let $k>2$ be even and let $(x_i)_{1\le i \le k}$ be pairwise distinct complex numbers. Let $\mathfrak{C}_k$ be the set of cyclic permutation of length $k$. Then
   $$
   \sum_{\sigma \in \mathfrak{C}^k} \prod_{i=1}^k \frac1{x_i - x_{\sigma(i)}} = 0.
   $$
 \end{lemma}
Note in particular that it follows from Lemma \ref{L:k_cancel} that if $A$ is the matrix $A_{ij} = 1_{i \neq j} 1/ (x_i - x_j)$ then $\det(A)  $ can be written as a sum over matchings (which can be thought of as permutations with no fixed points and where each cycle has length $2$):
\begin{equation}\label{E:detmatching}
\det(A) = \sum_{\m} \prod_{(u,v) \in \m} \frac1{ ( x_u - x_v)^2}
\end{equation}
 which is Lemma 3.1 of Kenyon \cite{KenyonGFF}.

\begin{proof}[Proof of Lemma \ref{L:k_cancel}.]
First of all, the case $k =4$ must be true because of \eqref{E:detmatching} (note that the odd cycles clearly give a zero contribution to the determinant by an argument similar to the above).

Using again \eqref{E:detmatching} but for $k =6$ gives the desired identity for $k = 6$, since the terms corresponding to $\mathfrak{C}_4$ in the expansion of the determinant into permutations contribute zero by the case $k = 4$. Proceeding by induction, we deduce the result for every even $k \ge 4$.
\end{proof}

By Lemma \ref{L:k_cancel}, the number of cycles $n$ is necessarily $k/2$. Note also that in a two-cycle we get a term of the form $C(z,w)$ and another one of the form $C(w,z) = - C(z,w)$, which results in a term of the form $- C(z,w)^2$. Hence the moment \eqref{E:sumpermut} becomes
\begin{align}
  & a \delta^k \sum_\nu (\prod \nu_i) \sum_{\m  \in \cM(1, \ldots, k)} (-\tfrac12)^{k/2} \times \nonumber \\
  \times & \prod_{(i,j) \in \m} -\left[ \frac{(s(b_i) s(w_i))^{(1+\nu_i)/2 } (s(b_j) s(w_j))^{(1+\nu_j)/2 }}{\pi^2(z_i^{\nu_i} - z_j^{\nu_j})^2} + \frac{(s(b_i) s(w_i))^{(1- \nu_i)/2} (s(b_j) s(w_j))^{(1-\nu_j)/2 }}{\pi^2 (\bar z_i^{\nu_i} - \bar z_j^{\nu_j})^2}\right] + o( \delta^{k+\eps}/D^k)
   \label{E:summatch}
\end{align}
where $\cM(1, \ldots, k)$ are the matchings of $1, \ldots, k$.

Recall that in the above expression $s(b_i)$ and $s(b_j)$ refer to the sign (parity) of the white and black vertex respectively of the $\alpha$-edge in position $t_i$ of the path $\gamma_i$ (we have already accounted for the corresponding $\beta$-edge). We will now sum over $i$ and interpret the corresponding sums as discrete Riemann sums converging to integrals.
For a horizontal edge $(w_i, b_i)$, we have $s(b_i) s(w_i) = 1$ whereas it is $-1$ for a vertical edge: this is simply because $s$ measures the parity of the row. We claim that (as in Kenyon's proof of Proposition 20 in \cite{Kenyon_ci}, see the equation between (20) and (21)),
\begin{equation}
\label{dz}
2 \delta (s(b_i) s(w_i))^{(1 + \nu_i) /2} \nu_i K(w_i, b_i) = -  i \delta z_i^{\nu_i}.
\end{equation}
Indeed, suppose for instance that $\gamma_i$ moves horizontally from left to right in step $t_i$. Then the corresponding $\alpha$-edge is vertical, and has a black vertex at the bottom so $K(w_i, b_i) = +i$. Furthermore, $\delta z_i = \delta \bar z_i = 2 \delta$ (since one step of the path corresponds to two faces of length $\delta$ each). The vertical cases can be checked similarly (keeping in mind the corresponding values of $\delta z_i $ and $\delta \bar z_i$).


From \eqref{dz}, we can multiply by $\nu_i$ both sides of the equation and take the product over $i$. Then, recalling that $a = \prod_i K(w_i, b_i)$,
%
%
and observing also that the second term in each bracket of the right hand side of \eqref{E:summatch} is the same as the first term but with $\nu_i$ replaced by $- \nu_i$ and $\nu_j$ replaced by $- \nu_j$, \eqref{E:summatch} becomes
\begin{align}
  &  \sum_\nu \sum_{\m  \in \cM(1, \ldots, k)} (-1)^{k}\frac1{2^{3k/2}} 
  \times  \prod_{(i,j) \in \m} - \left[
  \frac{\delta z_i^{\nu_i} \delta z_j^{\nu_j}}
  {\pi^2(z_i^{\nu_i} - z_j^{\nu_j})^2} +   \frac{\delta z_i^{-\nu_i} \delta z_j^{-\nu_j}}
  {\pi^2(z_i^{-\nu_i} - z_j^{-\nu_j})^2}\right] + o( \delta^{k+\eps}/D^k)
   \label{E:summatch2}
\end{align}
(We have kept a term $(-1)^k$ even though $k$ is even to indicate that this comes from $(-1)^{k/2}$ at the top of \eqref{E:summatch} and a factor $-1$ in each of the $k/2$ terms of the product in the bottom of the same equation. On the other hand, the coefficient $-1$ in each of the $k/2$ of the product in \eqref{E:summatch2} above comes from the coefficient $-i$ (squared) in the right hand side of \eqref{dz}.) Fixing the matching and summing over $\nu$ (so exchanging order of summation) we get
\begin{equation}
 (-1)^{3k/2}\frac1{2^{3k/2}}  \sum_{\m \in \cM(1, \ldots, k)} \prod_{(i,j) \in \m} 2 \left[ \frac{\delta z_i \delta z_j}{\pi^2(z_i - z_j)^2}
  + \frac{\delta \bar z_i \delta \bar z_j}{\pi^2(\bar z_i - \bar z_j)^2}
  + \frac{\delta \bar z_i \delta z_j}{\pi^2(\bar z_i - z_j)^2}
  + \frac{\delta z_i \delta \bar z_j}{\pi^2(z_i - \bar z_j)^2} \right] + o( \delta^{k+\eps}/D^k)
\end{equation}
(The term $(-1)^{3k/2}$ in front comes from the previous $(-1)^{k}$ in \eqref{E:summatch2} and another factor $(-1)$ in each of the $k/2$ terms of the product of the same equation.) Summing over the choice of $t_i$ in \eqref{E:alphabeta}, and since $k$ is even (so $(-1)^{3k/2} = (-1)^{k/2}$), we obtain
\begin{align}
    \mu[ (h^{\delta}(a_1)- h^{\delta}(b_1)) &\cdots(h^{\delta}(a_k)- h^{\delta}(b_k))  ]  =   \nonumber \\
&\frac{(-1)^{k/2}}{2^{k}}\sum_{\m \in \cM(1 , \ldots, k)} \prod_{(i,j) \in \m}  \int_{\gamma_i} \int_{\gamma_j}
   \Big [
    \frac{\diff z_i \diff z_j}{\pi^2(z_i - z_j)^2}
  + \frac{\diff \bar z_i \diff \bar z_j}{\pi^2(\bar z_i - \bar z_j)^2} \nonumber\\& \qquad\qquad
  + \frac{\diff \bar z_i \diff z_j}{\pi^2(\bar z_i - z_j)^2}
  + \frac{\diff z_i \diff \bar z_j}{\pi^2(z_i - \bar z_j)^2}
  + O\big( \frac{\delta}{D^6}\big) \Big] + o\big(\frac{\delta^\eps}{D^k}\big) .
\end{align}
To understand the bound on the error above, the term outside of the brackets corresponds to summing the error in \eqref{E:summatch2} over $k$ paths (each of length at most $O(\delta^{-1})$); the term inside corresponds to approximating a Riemann sum by an integral. When we do so, for each sum/integral, we make an error of size at most $O(\delta/D) |\sup f'|$, since each path is at least of length $D/\delta$, and $f$ is the function being integrated, so that here $\sup |f'| = O ( D^{-3})$. Furthermore, as these are double integrals, we need to multiply this error for a single integral by the overall value the other integral which we bound crudely by $O(1/D^2)$.

Now observe that
$$
\int_{\gamma_i} \int_{\gamma_j}
    \frac{\diff z_i \diff z_j}{(z_i - z_j)^2} = \log \frac{( a_i - a_j)(b_i - b_j)}{(a_i - b_j)(b_i - a_j)}
$$
Noting that the four integrals give two pairs of conjugate complex numbers, and recalling that $x + \bar x = 2 \Re(x)$, we obtain
\begin{align}
     \mu[ (h^{\delta}(a_1)- h^{\delta}(b_1))& \cdots(h^{\delta}(a_k)- h^{\delta}(b_k))  ] \\
& =\sum_{\m \in \cM(1 , \ldots, k)} \prod_{(i,j) \in \m}  -\frac1{2\pi^2} \Re \log \frac{(a_i - a_j)(b_i - b_j)(\bar a_i - a_j) (\bar b_i - b_j)}{ (a_i - b_j)(b_i - a_j)(\bar a_i - b_j)(\bar b_i - a_j)} + \text{err.} \nonumber  \end{align}
where
$$
\text{err.}  = o\left (\frac{\delta^\eps}{D^k} \right) + O\left( \frac{\delta}{D^6}\right) O\left(\log D\right)^{k/2-1} = o( \delta^{\eps - k \beta})
$$
for $\beta$ sufficiently small, since $D \ge \delta^\beta$. In particular, if $\beta$ is sufficiently small (depending on $k$ but not on anything else) then this error goes to zero as $\delta \to 0$.
This concludes the proof of Proposition~\ref{P:kpoint_dimer}.
\end{proof}

\subsection{Convergence of the height function} \label{S:final}
In this section we finish the proof of the scaling limit result from Theorem~\ref{T:NGFF_intro}.

One can think of the height function on $\Z^2\cap\H$ (which is defined up to a constant) as a random distribution (generalised function) acting on bounded test functions $f$ with compact support and mean zero.
We follow~\cite{BLR16} and write the action as
\begin{align} \label{eq:action}
(h^\delta,f) = \int_\H \int_\H (h^\delta(a)-h^\delta(b))\frac{f^+(a)f^-(b)}{Z_f}dadb,
\end{align}
where $f^\pm = \max \{ \pm f ,0\}$ and $Z_f=\int_{\H}f^+(a)da=\int_{\H}f^-(a)da $. Note that this is well defined as the additive indeterminate constant in $h^\delta$ cancels out in this expression.
One can also check that this gives the same result as just integrating the height function against $f$.
Note that by Fubini's theorem $( h^\delta,f)$ is centered as $\mu (h^\delta(a) -h^\delta(b))=0$
for all $a,b\in \mathbb H$ by our choice of the reference flow from Section~\ref{S:MD}.

The result in Proposition \ref{P:kpoint_dimer} is the key step to prove the main result of the paper, which we rephrase below for convenience. In fact, in \cite{Kenyon_ci}, no further justification beyond the analogue of Proposition \ref{P:kpoint_dimer} is provided (this is also the case in \cite{Russkikh}). The fact that an argument is missing was already pointed out by de Tili\`ere in \cite{BdT_quadritilings} (see Lemma 20 in that paper).
Here, we follow an approach similar to the one used in \cite{las2013lozenge} and in Toninelli's lecture notes \cite{toninelli_notes} (see in particular Theorem 5.4 and the following discussion), but tailored to our setup since our a priori error estimates are somewhat different.

\begin{thm} \label{thm:fieldmoments}
Let $\hg^{\textnormal{Neu}}_{\H}$ be the Neumann Gaussian free field in $\H$, and let $f_1,\ldots f_k \in \cD_0(\H)$ (smooth test functions of compact support and mean zero). Then for $l_1,\ldots,l_k \in \mathbb{N}$,
\[
\mu \Big[ \prod_{i=1}^k (h^{\delta},f_i)^{l_i}\Big] \to \mathbf{E} \Big[ \prod_{i=1}^k  (\tfrac1{\sqrt{2} \pi} \hg^{\textnormal{Neu}}_{\H},f_i)^{l_i} \Big], \quad \text{as } \delta \to 0,
\]
where $\mathbf E$ is the expectation associated with $\hg^{\textnormal{Neu}}_{\H}$.
\end{thm}

\begin{proof}
For a function $g$, we write $g(a;b)=g(a)-g(b)$.
To simplify the exposition, we only treat the case of the second moment; the other cases are similar but with heavier notation. To start with, note that
\begin{align} \label{eq:secondmoment}
\mu [ (h^{\delta},f_1) (h^{\delta},f_2) ] = \int_{\H^4} \mu [h^\delta(a_1;b_1) h^\delta(a_2;b_2)]\frac{f^+_1(a_1)f^-_1(b_1)}{Z_{f_1}}\frac{f^+_2(a_2)f^-_2(b_2)}{Z_{f_2}}da_1db_1da_2db_2.
\end{align}
Let $\rho>0$ be such that $\Im (z) \ge \rho$ whenever $z \in \text{Supp} (f_1) \cup \text{Supp} (f_2)$, and let \[
H(a_1,b_1,a_2,b_2)= -\frac1{2\pi^2} \Re \log \frac{(a_1 - a_2)(b_1 - b_2)(\bar a_1- a_2) (\bar b_1 - b_2)}{ (a_1 - b_2)(b_1- a_2)(\bar a_1- b_2)(\bar b_1 - a_2)}.
\]
By Proposition~\ref{P:kpoint_dimer}, since all relevant points have imaginary parts greater than $\rho$, we have
\begin{align} \label{eq:uniform}
\mu [h^\delta(a_1;b_1) h^\delta(a_2;b_2)]=H(a_1,b_1,a_2,b_2)+o(1),
\end{align}
where the error $o(1)$ is \textbf{uniform} over
\[
\cD_\delta := \{ (a_1,b_1,a_2,b_2) \in \H^4: D \geq \delta^\beta\},
\]
where, as before, $D=D(a_1,b_1,a_2,b_2)$ denotes the minimal distance in the complex plane between any pair of points within
{$\{a_1,b_1, a_2, b_2\}$}.
We now split the integral in \eqref{eq:secondmoment} into the integral over $\cD_\delta$ and over $\cD^c_\delta$. Now the important observation is that since the
the error is uniform, the limit of the integral over $\cD_\delta$ is given by
\[
 \int_{\H^4} H(a_1, \ldots, b_2)\frac{f^+_1(a_1)f^-_1(b_1)}{Z_{f_1}}\frac{f^+_2(a_2)f^-_2(b_2)}{Z_{f_2}} da_1db_1da_2db_2=\frac1{2\pi^2}\mathbf{E} [ (\hg,f_1) (\hg,f_2) ].
\]

Therefore, we are left with proving that the contribution to the integral \eqref{eq:secondmoment} coming from $\cD_\delta^c$ is negligible. To do that, we proceed somewhat crudely, noting that by Cauchy--Schwarz,
\begin{equation}\label{CS}
\mu [h^\delta(a_1;b_1) h^\delta(a_2;b_2)] \le \left ( \var_\mu ( h^\delta(a_1; b_1)) \var_\mu (h^\delta (a_2; b_2)) \right)^{1/2}.
\end{equation}
 Since the volume of $\cD_\delta^c$ is polynomially small in $\delta$, it will therefore suffice to show that
\begin{equation}\label{variance}
 \var_{\mu} (h^\delta(a;b) )=  O ( \log \delta)^C
\end{equation}
for some $C>0$ and arbitrary points $a,b$ within some fixed compact of $\H$.
To prove this, we will go back to the definition of the height function as a sum of increments over a path, and we will use our \emph{a priori} bound on $K^{-1}$ coming from Proposition \ref{P:aprioriPK}, which gives $$
K^{-1} (u,v)  = O\Big( (\log \dist (u,v))^C\frac{ 1 }{\dist(u,v)}\Big).
$$
Fix two paths $\gamma_1, \gamma_2$ from $a$ to $b$. It will be advantageous to take these paths at positive macroscopic distance from one another except near the endpoints, where they must necessarily come together. We will explain more precisely below how we construct them.
By the triangle inequality and Theorem \ref{thm:Kasteleyn}, we have
\begin{align*}
    \var_\mu (h^\delta(a;b)) & \le \sum_ {e_1 \in \gamma_1, e_2 \in \gamma_2}  |\cov_{\mu} (\mathbf 1_{\{ e_1 \in  \cM\}}, \mathbf 1_{\{ e_2 \in \cM\}})|\\
   & \lesssim \sum_{e_1 \in \gamma_1, e_2 \in \gamma_2}  \Big( (\log \dist(e_1, e_2))^C\frac{ 1}{\dist(e_1, e_2)}\Big)^2\\
   & \lesssim (\log \delta)^{2C}  \sum_{e_1 \in \gamma_1, e_2 \in \gamma_2}  \Big( \frac{ 1}{\dist(e_1, e_2)}\Big)^2.
\end{align*}
We may assume that near $a$ and $b$, the paths $\gamma_1$ and $\gamma_2$ form straight segments with different directions (say opposite directions), until they reach a fixed positive distance $\alpha$, taken to be small enough that these paths remain at positive distance from the real line. The segments near $a$ and $b$ are then joined by portions of paths staying at distance (in the plane) at least $\alpha/2$ from one another to form $\gamma_1$ and $\gamma_2$. Then
\begin{align}
  \sum_{e_1 \in \gamma_1, e_2 \in \gamma_2}  \Big( \frac{ 1}{\dist(e_1, e_2)}\Big)^2 & \le \sum_{r= 1}^{O( 1/\delta)} \frac1{r^2} \#\{ (e_1, e_2) : \dist(e_1, e_2) = r\}\nonumber \\
  &
    \lesssim \sum_{r= R}^{\alpha/(2\delta)}  \frac1{r^2} r + O(1) \le \log (\delta^{-1}). \label{angles}
\end{align}
We provide a brief explanation for the crucial point above, which is the bound on $\#\{ e_1, e_2 : \dist(e_1, e_2) = r\} $, and which comes from the choice of paths $\gamma_1$ and $\gamma_2$. Indeed, if $1\le r \le \alpha/2$ fixed, then elementary geometric considerations imply that any choice of $e_1$ in the segment of $\gamma_1$ at distance at most $r$ from $a$, will give at most one corresponding point $e_2$ on $\gamma_2$ such that $\dist (e_1, e_2) = r$. On the other hand, for $r \ge \alpha/2$, there are at most $O(\delta^{-2})$ pairs of edges on the whole path, so $ \# \{ (e_1, e_2): \dist (e_1, e_2) \ge \alpha / (2\delta)\} = O(\delta^{-2})$, so the contribution of such edges to the sum is indeed $O(1)$ as claimed above.

This proves \eqref{variance} and therefore completes the proof of Theorem \ref{thm:fieldmoments} in the case of second moments. In the general case of a moment of order $k\ge 2$, the same proof works, where we replace the use of Cauchy--Schwarz in \eqref{CS} by a H\"older inequality, so that it suffices to show that $\mu [ (h^\delta(a;b))^{2k} ] \lesssim (\log 1/\delta)^{C_k}$ (we wrote here a moment of order $2k $ rather than $k$ to account for the possibility that $k$ is odd). We therefore need to choose $2k$ paths leading from $a$ to $b$. As above, these paths may be chosen as being straight line segments up to a small distance $\alpha$ (in the plane) away from $a$ or $b$, and with distinct directions; we simply choose the angles between these segments to be $\pi/k$, and otherwise require that these paths stay at positive distance from one another. It is easy to check that the analogue of \eqref{angles} holds also in this case.
%
\end{proof}

A standard argument says that since all moments of $h^{\delta}$ converge to the corresponding moments of $\tfrac1{\sqrt{2} \pi} \hg^{\textnormal{Neu}}_{\H} $, and since $\hg^{\textnormal{Neu}}_{\H}$ is a Gaussian process, we can conclude that $h^{\delta}\to \tfrac1{\sqrt{2} \pi} \hg^{\textnormal{Neu}}_{\H}$ in distribution as $\delta \to 0$ (in the sense of finite dimensional distributions, where $\hg$ is viewed as a stochastic process indexed by smooth test functions with compact support and mean zero). As a result we have proved Theorem~\ref{T:NGFF_intro}.

\bibliographystyle{abbrv}
\bibliography{MonomerDimer}
\end{document}